\documentclass[12pt]{amsart}
\usepackage{txfonts}      %{article} was 12pt latex e
\usepackage{amssymb}
\usepackage{eucal}
\usepackage{amsmath}
\usepackage{amscd}
\usepackage{xcolor}
\usepackage{multicol}
\usepackage[all]{xy}           %xypic macro for latex2.09
\usepackage{graphicx}
\usepackage{color}
\usepackage{colordvi}
\usepackage{xspace}
\usepackage{tikz}
\usepackage{makecell}
\usepackage{appendix}
\usepackage{amsthm}

\usepackage{ifpdf}
\ifpdf
\usepackage[colorlinks,final,backref=page,hyperindex]{hyperref}
\else
\usepackage[colorlinks,final,hyperindex,hypertex]{hyperref}
\fi

\usepackage[active]{srcltx} %SRC Specials for DVI Searching

\begin{document}

%%%%%%%%%%%%%%%%%%%%%%%% Statements

\newtheorem{thm}{Theorem}[section]
\newtheorem{lem}[thm]{Lemma}
\newtheorem{cor}[thm]{Corollary}
\newtheorem{pro}[thm]{Proposition}
\theoremstyle{definition}
\newtheorem{defi}[thm]{Definition}
\newtheorem{ex}[thm]{Example}
\newtheorem{rmk}[thm]{Remark}
\newtheorem{pdef}[thm]{Proposition-Definition}
\newtheorem{condition}[thm]{Condition}

\renewcommand{\labelenumi}{{\rm(\alph{enumi})}}
\renewcommand{\theenumi}{\alph{enumi}}

\newcommand {\emptycomment}[1]{} %to remove paragraphs

\newcommand{\nc}{\newcommand}
\newcommand{\delete}[1]{}

\nc{\tred}[1]{\textcolor{red}{#1}}
\nc{\tblue}[1]{\textcolor{blue}{#1}}
\nc{\tgreen}[1]{\textcolor{green}{#1}}
\nc{\tpurple}[1]{\textcolor{purple}{#1}}
\nc{\tgray}[1]{\textcolor{gray}{#1}}
\nc{\torg}[1]{\textcolor{orange}{#1}}
\nc{\tmag}[1]{\textcolor{magenta}}
\nc{\btred}[1]{\textcolor{red}{\bf #1}}
\nc{\btblue}[1]{\textcolor{blue}{\bf #1}}
\nc{\btgreen}[1]{\textcolor{green}{\bf #1}}
\nc{\btpurple}[1]{\textcolor{purple}{\bf #1}}

\nc{\todo}[1]{\tred{To do:} #1}

%\delete{
    \nc{\mlabel}[1]{\label{#1}}  % Use this to suppress names
    \nc{\mcite}[1]{\cite{#1}}  % Use this to suppress names
    \nc{\mref}[1]{\ref{#1}}  % Use this to suppress names
    \nc{\meqref}[1]{\eqref{#1}}  % Use this to suppress names
    \nc{\mbibitem}[1]{\bibitem{#1}} % Use this to show number
%}

\delete{
%    \nc{\mlabel}[1]{\label{#1}  % Use the next two lines to show names
%        { {\small\tgreen{\tt{{\ }(#1)}}}}}
    \nc{\mcite}[1]{\cite{#1}{\small{\tt{{\ }(#1)}}}}  % Use this lines to show names
    \nc{\mref}[1]{\ref{#1}{\small{\tred{\tt{{\ }(#1)}}}}}  % Use this lines to show names
    \nc{\meqref}[1]{\eqref{#1}{{\tt{{\ }(#1)}}}}  % Use this lines to show names
    \nc{\mbibitem}[1]{\bibitem[\bf #1]{#1}} % Use this to show name
}

%    \nc{\mlabel}[1]{  % Use the next two lines to show names
%           { {\small\tgreen{\tt{{\ }(#1)}}}}}

\nc{\cm}[1]{\textcolor{red}{Chengming:#1}}
\nc{\yy}[1]{\textcolor{blue}{Yanyong: #1}}
%\nc{\lit}[2]{\textcolor{blue}{#1}{ \textcolor{purple}{(#2)}}}
%\nc{\lit}[2]{\textcolor{blue}{#1}{}} %use this line instead of the previous one to show only the new changes
\nc{\li}[1]{\textcolor{purple}{#1}}
\nc{\lir}[1]{\textcolor{purple}{Li:#1}}

\nc{\revise}[1]{\textcolor{blue}{#1}}

%%%%%%%% new symbols

\nc{\tforall}{\ \ \text{for all }}
\nc{\hatot}{\,\widehat{\otimes} \,}
\nc{\complete}{completed\xspace}
\nc{\wdhat}[1]{\widehat{#1}}

\nc{\ts}{\mathfrak{p}}
\nc{\mts}{c_{(i)}\ot d_{(j)}}

\nc{\NA}{{\bf NA}}
\nc{\LA}{{\bf Lie}}
\nc{\CLA}{{\bf CLA}}

\nc{\cybe}{CYBE\xspace}
\nc{\nybe}{NYBE\xspace}
\nc{\ccybe}{CCYBE\xspace}

\nc{\ndend}{pre-Novikov\xspace}
\nc{\calb}{\mathcal{B}}
\nc{\rk}{\mathrm{r}}
\newcommand{\g}{\mathfrak g}
\newcommand{\h}{\mathfrak h}
\newcommand{\pf}{\noindent{$Proof$.}\ }
\newcommand{\frkg}{\mathfrak g}
\newcommand{\frkh}{\mathfrak h}
\newcommand{\Id}{\rm{Id}}
\newcommand{\gl}{\mathfrak {gl}}
\newcommand{\ad}{\mathrm{ad}}
\newcommand{\add}{\frka\frkd}
\newcommand{\frka}{\mathfrak a}
\newcommand{\frkb}{\mathfrak b}
\newcommand{\frkc}{\mathfrak c}
\newcommand{\frkd}{\mathfrak d}
\newcommand {\comment}[1]{{\marginpar{*}\scriptsize\textbf{Comments:} #1}}

\nc{\vspa}{\vspace{-.1cm}}
\nc{\vspb}{\vspace{-.2cm}}
\nc{\vspc}{\vspace{-.3cm}}
\nc{\vspd}{\vspace{-.4cm}}
\nc{\vspe}{\vspace{-.5cm}}

%%%%%%%%%%%%%%%%%%%%%%% old symbols

\nc{\disp}[1]{\displaystyle{#1}}
\nc{\bin}[2]{ (_{\stackrel{\scs{#1}}{\scs{#2}}})}  %binomial coeff
\nc{\binc}[2]{ \left (\!\! \begin{array}{c} \scs{#1}\\
    \scs{#2} \end{array}\!\! \right )}  %binomial coeff
\nc{\bincc}[2]{  \left ( {\scs{#1} \atop
    \vspace{-.5cm}\scs{#2}} \right )}  %binomial coeff
\nc{\ot}{\otimes}
\nc{\sot}{{\scriptstyle{\ot}}}
\nc{\otm}{\overline{\ot}}
\nc{\ola}[1]{\stackrel{#1}{\la}}%${\Bbb Z}$

\nc{\scs}[1]{\scriptstyle{#1}} \nc{\mrm}[1]{{\rm #1}}

\nc{\dirlim}{\displaystyle{\lim_{\longrightarrow}}\,}
\nc{\invlim}{\displaystyle{\lim_{\longleftarrow}}\,}

\nc{\bfk}{{\bf k}} \nc{\bfone}{{\bf 1}}
\nc{\rpr}{\circ}
%\nc{\apr}{\cdot}
\nc{\dpr}{{\tiny\diamond}}
\nc{\rprpm}{{\rpr}}

%%%%%%%%%%%%%%%%%%%%% roman fonts, in alphabetic order
\nc{\mmbox}[1]{\mbox{\ #1\ }} \nc{\ann}{\mrm{ann}}
\nc{\Aut}{\mrm{Aut}} \nc{\can}{\mrm{can}}
\nc{\twoalg}{{two-sided algebra}\xspace}
\nc{\colim}{\mrm{colim}}
\nc{\Cont}{\mrm{Cont}} \nc{\rchar}{\mrm{char}}
\nc{\cok}{\mrm{coker}} \nc{\dtf}{{R-{\rm tf}}} \nc{\dtor}{{R-{\rm
tor}}}
\renewcommand{\det}{\mrm{det}}
\nc{\depth}{{\mrm d}}
\nc{\End}{\mrm{End}} \nc{\Ext}{\mrm{Ext}}
\nc{\Fil}{\mrm{Fil}} \nc{\Frob}{\mrm{Frob}} \nc{\Gal}{\mrm{Gal}}
\nc{\GL}{\mrm{GL}} \nc{\Hom}{\mrm{Hom}} \nc{\hsr}{\mrm{H}}
\nc{\hpol}{\mrm{HP}}  \nc{\id}{\mrm{id}} \nc{\im}{\mrm{im}}

\nc{\incl}{\mrm{incl}} \nc{\length}{\mrm{length}}
\nc{\LR}{\mrm{LR}} \nc{\mchar}{\rm char} \nc{\NC}{\mrm{NC}}
\nc{\mpart}{\mrm{part}} \nc{\pl}{\mrm{PL}}
\nc{\ql}{{\QQ_\ell}} \nc{\qp}{{\QQ_p}}
\nc{\rank}{\mrm{rank}} \nc{\rba}{\rm{RBA }} \nc{\rbas}{\rm{RBAs }}
\nc{\rbpl}{\mrm{RBPL}}
\nc{\rbw}{\rm{RBW }} \nc{\rbws}{\rm{RBWs }} \nc{\rcot}{\mrm{cot}}
\nc{\rest}{\rm{controlled}\xspace}
\nc{\rdef}{\mrm{def}} \nc{\rdiv}{{\rm div}} \nc{\rtf}{{\rm tf}}
\nc{\rtor}{{\rm tor}} \nc{\res}{\mrm{res}} \nc{\SL}{\mrm{SL}}
\nc{\Spec}{\mrm{Spec}} \nc{\tor}{\mrm{tor}} \nc{\Tr}{\mrm{Tr}}
\nc{\mtr}{\mrm{sk}}

%%%%%%%%%%%%%%%%%% bold face
\nc{\ab}{\mathbf{Ab}} \nc{\Alg}{\mathbf{Alg}}

%%%%%%%%%%%%%%%%%%%Bbb fonts
\nc{\BA}{{\mathbb A}} \nc{\CC}{{\mathbb C}} \nc{\DD}{{\mathbb D}}
\nc{\EE}{{\mathbb E}} \nc{\FF}{{\mathbb F}} \nc{\GG}{{\mathbb G}}
\nc{\HH}{{\mathbb H}} \nc{\LL}{{\mathbb L}} \nc{\NN}{{\mathbb N}}
\nc{\QQ}{{\mathbb Q}} \nc{\RR}{{\mathbb R}} \nc{\BS}{{\mathbb{S}}} \nc{\TT}{{\mathbb T}}
\nc{\VV}{{\mathbb V}} \nc{\ZZ}{{\mathbb Z}}

%%%%%%%%%%%%%%%%%%% cal fonts

\nc{\calao}{{\mathcal A}} \nc{\cala}{{\mathcal A}}
\nc{\calc}{{\mathcal C}} \nc{\cald}{{\mathcal D}}
\nc{\cale}{{\mathcal E}} \nc{\calf}{{\mathcal F}}
\nc{\calfr}{{{\mathcal F}^{\,r}}} \nc{\calfo}{{\mathcal F}^0}
\nc{\calfro}{{\mathcal F}^{\,r,0}} \nc{\oF}{\overline{F}}
\nc{\calg}{{\mathcal G}} \nc{\calh}{{\mathcal H}}
\nc{\cali}{{\mathcal I}} \nc{\calj}{{\mathcal J}}
\nc{\call}{{\mathcal L}} \nc{\calm}{{\mathcal M}}
\nc{\caln}{{\mathcal N}} \nc{\calo}{{\mathcal O}}
\nc{\calp}{{\mathcal P}} \nc{\calq}{{\mathcal Q}} \nc{\calr}{{\mathcal R}}
\nc{\calt}{{\mathcal T}} \nc{\caltr}{{\mathcal T}^{\,r}}
\nc{\calu}{{\mathcal U}} \nc{\calv}{{\mathcal V}}
\nc{\calw}{{\mathcal W}} \nc{\calx}{{\mathcal X}}
\nc{\CA}{\mathcal{A}}

%%%%%%%%%%%%%%%%%%  frak fonts
\nc{\fraka}{{\mathfrak a}} \nc{\frakB}{{\mathfrak B}}
\nc{\frakb}{{\mathfrak b}} \nc{\frakd}{{\mathfrak d}}
\nc{\oD}{\overline{D}}
\nc{\frakF}{{\mathfrak F}} \nc{\frakg}{{\mathfrak g}}
\nc{\frakm}{{\mathfrak m}} \nc{\frakM}{{\mathfrak M}}
\nc{\frakMo}{{\mathfrak M}^0} \nc{\frakp}{{\mathfrak p}}
\nc{\frakS}{{\mathfrak S}} \nc{\frakSo}{{\mathfrak S}^0}
\nc{\fraks}{{\mathfrak s}} \nc{\os}{\overline{\fraks}}
\nc{\frakT}{{\mathfrak T}}
\nc{\oT}{\overline{T}}
%\nc{\frakx}{{\mathfrak x}}
\nc{\frakX}{{\mathfrak X}} \nc{\frakXo}{{\mathfrak X}^0}
\nc{\frakx}{{\mathbf x}}
%\nc{\frakTxo}{{\frakTx}^0}
\nc{\frakTx}{\frakT}      %All rooted trees, correspond to \ncsha(X)
\nc{\frakTa}{\frakT^a}        % rooted trees for \ncsha(A)
\nc{\frakTxo}{\frakTx^0}   % rooted trees for \ncshao(X)
\nc{\caltao}{\calt^{a,0}}   % rooted trees for \ncshao(A)
\nc{\ox}{\overline{\frakx}} \nc{\fraky}{{\mathfrak y}}
\nc{\frakz}{{\mathfrak z}} \nc{\oX}{\overline{X}}

\font\cyr=wncyr10

%%%%%%%%%%%%%%%%%%%%%%%%%%%%%%%%%%%%%%%%%%%%%%%%%%%%%%%%%%%%%%%%%%

%\begin{document}
\title[]{On Poisson conformal bialgebras}

\author{Yanyong Hong}
\address{School of Mathematics, Hangzhou Normal University,
Hangzhou 311121, PR China}
\email{yyhong@hznu.edu.cn}

\author{Chengming Bai}
\address{Chern Institute of Mathematics \& LPMC, Nankai University, Tianjin 300071, PR China}
\email{baicm@nankai.edu.cn}

\subjclass[2010]{17A60, 17B63, 17B62, 53D55} \keywords{Poisson
conformal algebra, Poisson conformal bialgebra, classical
Yang-Baxter equation, $\mathcal{O}$-operator, Novikov-Poisson
algebra, Gel'fand-Dorfman algebra}

\begin{abstract}
 We develop a conformal analog of the theory
of Poisson bialgebras as well as a bialgebra theory of Poisson
conformal algebras. We introduce the notion of Poisson conformal
bialgebras, which are characterized by Manin
triples of Poisson conformal algebras. 
A class of special Poisson conformal bialgebras called coboundary
Poisson conformal bialgebras are constructed from skew-symmetric
solutions of the Poisson conformal Yang-Baxter equation, whose
operator forms are studied. Then we show that the semi-classical
limits of conformal formal deformations of commutative and
cocommutative antisymmetric infinitesimal conformal bialgebras are
Poisson conformal bialgebras. Finally, we extend the
correspondence between Poisson conformal algebras and
Poisson-Gel'fand-Dorfman algebras to the context of bialgebras,
that is, we introduce the notion of Poisson-Gel'fand-Dorfman
bialgebras and show that Poisson-Gel'fand-Dorfman bialgebras
correspond to a class of Poisson conformal bialgebras. Moreover, a
construction of Poisson conformal bialgebras from
pre-Poisson-Gel'fand-Dorfman algebras is given.
\end{abstract}

\maketitle

\vspace{-1.2cm}

\tableofcontents

\vspace{-1.2cm}

\allowdisplaybreaks

\section{Introduction}
A Poisson algebra is a vector space with  a Lie algebra structure
and a commutative associative algebra structure which are entwined
by the Leibniz rule. Poisson algebras play important roles in many
fields in mathematics and mathematical physics, such as Poisson
geometry \cite{V, W},  classical and quantum mechanics \cite{Ar,
BFFLS}, quantization theory \cite{BFFLS1, BFFLS2, Hu, Ko},
algebraic geometry \cite{GK, P} and quantum groups \cite{CP, Dr}.
A bialgebra theory of Poisson algebras with the notion of Poisson
bialgebras was given in \cite{NB} in terms of representation
theory of Poisson algebras, which is related with integrable
systems. A Poisson bialgebra is a combination of a Lie bialgebra
\cite{CP} and a commutative and cocommutative antisymmetric
infinitesimal bialgebra (ASI bialgebra) \cite{A1, Bai1} with some
compatibility conditions, which is characterized by a Manin triple
of Poisson algebras. % or some matched pair of Poisson algebras.
Skew-symmetric solutions of the Poisson Yang-Baxter equation
naturally produce Poisson bialgebras. We also point out that a
bialgebra theory of noncommutative Poisson algebras was presented
in \cite{LBS}.

The notion of Lie conformal algebras  was  introduced  by V. Kac
in \cite{DK1, K1}, which gives an axiomatic description of the
singular part of the operator product expansion of chiral fields
in two-dimensional conformal field theory \cite{BPZ}. Lie
conformal algebras are also called vertex Lie algebras \cite{DLM,
Pr}, which are closely related with vertex algebras,
infinite-dimensional Lie algebras satisfying the locality property
\cite{K}, Hamiltonian formalism in the theory of nonlinear
evolution \cite{BDK} and so on. Note that there is a class of Lie
conformal algebras which correspond to a class of non-associative
algebras called Gel'fand-Dorfman algebras \cite{X1}. Associative
conformal algebras naturally appear in the representation theory
of Lie conformal algebras \cite{K1}. The notion of a Poisson
conformal algebra was introduced in \cite{Ko4}, which consists of
a Lie conformal algebra and a commutative associative conformal
algebra satisfying some compatibility condition. It was shown in
\cite{Ko4, KKP} that there are close relationships between Poisson
conformal algebras and representations of Lie conformal algebras.
Note that in \cite{LZ} the notion of noncommutative Poisson
conformal algebras was introduced and cohomology of
(noncommutative) Poisson conformal algebras was investigated.
Moreover, as the classical case, the semi-classical limits of
conformal formal deformations of commutative associative conformal
algebras are Poisson conformal algebras \cite{LZ}. The notion of
Poisson-Gel'fand-Dorfman algebras (PGD-algebras) was also given in
\cite{LZ}, which correspond to a class of Poisson conformal
algebras.

Motivated by the classical Lie bialgebra theory, a bialgebra
theory of Lie conformal algebras with the notion of Lie conformal
bialgebras was developed by J. Liberati in \cite{L}, including the
introduction of the notions of conformal Manin triples, coboundary
Lie conformal bialgebras and conformal classical Yang-Baxter
equation (CCYBE). The operator forms of the CCYBE were studied in
\cite{HB} where the notion of $\mathcal{O}$-operators on Lie
conformal algebras was introduced, providing skew-symmetric
solutions of the CCYBE. Moreover,  the notion of Gel'fand-Dorfman
bialgebras was introduced in \cite{HBG}, which also naturally
correspond to a class of Lie conformal bialgebras, based on the
correspondence between Gel'fand-Dorfman algebras and a class of
Lie conformal algebras. A conformal analog of ASI bialgebras was
presented in \cite{HB1} with the notion of ASI conformal
bialgebras, showing that they are equivalent to double
constructions of Frobenius conformal algebras, while
skew-symmetric solutions of the associative conformal Yang-Baxter
equation produce ASI conformal bialgebras. Moreover,
$\mathcal{O}$-operators on associative conformal algebras provide
skew-symmetric solutions of the associative conformal Yang-Baxter
equation and hence give rise to ASI conformal bialgebras.
Motivated by these results, it is natural to consider a conformal
analog of Poisson bialgebras, namely  Poisson conformal
bialgebras. Equivalently, they are also the bialgebras for Poisson
conformal algebras. Hence it is also natural to consider the
generalization of the aforementioned results involving Poisson
conformal algebras in the context of bialgebras. For example,
whether Poisson conformal bialgebras can be understood as certain
semi-classical limits of  conformal formal deformations of
commutative and cocommutative ASI conformal bialgebras and whether
there exists a bialgebra theory of PGD-algebras which correspond
to a class of Poisson conformal bialgebras. These are our main
motivations to present this paper.

In this paper, we first develop a bialgebra theory of Poisson
conformal algebras. We introduce the notion  of a Poisson
conformal bialgebra, which is both a Lie conformal bialgebra and a
commutative and cocommutative ASI conformal bialgebra satisfying
certain compatibility conditions.  We show that under some
conditions, a Poisson conformal bialgebra is
equivalent to a Manin triple of Poisson conformal algebras. % (or
%some matched pair of Poisson conformal algebras).
Moreover, we also investigate the coboundary case of Poisson
conformal algebras, which naturally leads to the introduction of
the notion of Poisson conformal Yang-Baxter equation (PCYBE). It
is shown that skew-symmetric solutions of the PCYBE can produce
Poisson conformal bialgebras. We also introduce the notions  of
$\mathcal{O}$-operators on Poisson conformal algebras and
pre-Poisson conformal algebras, while the former provide
skew-symmetric solutions of the PCYBE and the latter naturally
produce $\mathcal{O}$-operators on Poisson conformal algebras.
These results are summarized in the following diagram.
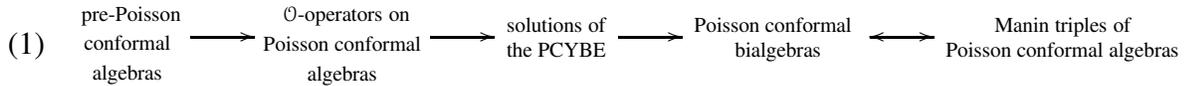
\begin{equation}
    \begin{split}
        \xymatrix{
            \txt{ \tiny pre-Poisson\\ \tiny conformal\\ \tiny algebras} \ar[r]     & \txt{ \tiny $\mathcal{O}$-operators on\\ \tiny Poisson conformal\\ \tiny algebras}\ar[r] &
            \text{solutions of}\atop \text{the PCYBE} \ar[r]  & \text{Poisson conformal }\atop \text{ bialgebras}  \ar@{<->}[r] &
            \text{Manin triples of} \atop \text{Poisson conformal algebras} }
        \vspace{-.1cm}
    \end{split}
    \mlabel{eq:Lieconfdiag1}
\end{equation}

Secondly,  we introduce the notions of a conformal formal
deformation of a commutative and cocommutative ASI conformal
bialgebra and its semi-classical limit, and show that the
semi-classical limit is a Poisson conformal bialgebra.

Thirdly, we introduce the notions of Poisson-Gel'fand-Dorfman
bialgebras (PGD-bialgebras) and differential Novikov-Poisson
bialgebras. Note that  differential Novikov-Poisson bialgebras are
special PGD-bialgebras. Moreover, we show that PGD-bialgebras
correspond to a class of Poisson conformal bialgebras.  Hence
there is the following diagram.
\begin{equation}
           \text{\tiny differential Novikov-Poisson bialgebras} \;\;{\Large \subset } \;\;\text{\tiny PGD-bialgebras}  \;\;\longleftrightarrow\;\;
            \text{\tiny a class of Poisson conformal bialgebras}
        \vspace{-.1cm}
    \mlabel{eq:Lieconfdiag}
\end{equation}
\delete{\begin{equation}
    \begin{split}
        \xymatrix{
           \text{\tiny differential Novikov-Poisson bialgebras $\subset $}&  \text{\tiny PGD-bialgebras}  \ar@{<->}[r] &
            \txt{\tiny a class of\\
          \tiny Poisson conformal bialgebras} }
        \vspace{-.1cm}
    \end{split}
    \mlabel{eq:Lieconfdiag}
\end{equation}}
\delete{ We also introduce the definition of
Poisson-Gel'fand-Dorfman Yang-Baxter equation (PGDYBE) and show
that skew-symmetric solutions of PGDYBE can produce
PGD-bialgebras. Moreover, the operator forms of PGDYBE are
studied, where the $\mathcal{O}$-operators on PGD-algebras can
produce skew-symmetric solutions of PGDYBE.} We also introduce the
notion of a pre-PGD-algebra, which is a vector space with a
pre-Novikov algebra structure \cite{HBG}, a Zinbiel algebra
structure and a left-symmetric algebra structure. A construction
of pre-PGD-algebras from Zinbiel algebras with a derivation is
given. Note that pre-PGD-algebras also correspond to a class of
pre-Poisson conformal algebras. Combining with the aforementioned
theory of Poisson conformal bialgebras illustrated by
diagram~(\ref{eq:Lieconfdiag1}), there is a construction of
Poisson conformal bialgebras from pre-PGD-algebras as well as Zinbiel
algebras with a derivation, which is summarized by the following
diagram.
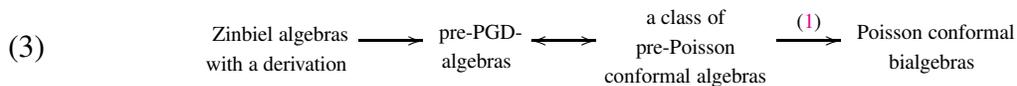
\begin{equation}
    \begin{split}
        \xymatrix{
          \txt{ \tiny Zinbiel algebras\\ \tiny with a derivation} \ar[r] &  \txt{  \tiny pre-PGD-\\ \tiny algebras} \ar@{<->}[r]  & \txt{\tiny a class of\\ \tiny pre-Poisson  \\ \tiny conformal algebras}\ar[r]^{\rm (\mref{eq:Lieconfdiag1})}&
         %   \text{solutions of}\atop \text{the PCYE} \ar[r]   &
            \txt{
          \tiny Poisson conformal\\ \tiny bialgebras} }
        \vspace{-.1cm}
    \end{split}
    \mlabel{eq:Lieconfdiag}
\end{equation}
\delete{Note that a differential Novikov-Poisson algebra (see \cite{BCZ}) is a special PGD-algebra. The definition of differential Novikov-Poisson bialgebras is given in this paper. Since a differential Novikov-Poisson bialgebra is a special PGD-bialgebras, the bialgebra theory of differential Novikov-Poisson algebras is also presented in this paper.}

This paper is organized as follows. In Section 2, we first recall
the notions of Poisson conformal algebras and PGD-algebras, and
the correspondence between PGD-algebras and a class of Poisson
conformal algebras. Then we recall some notions and facts
involving Poisson conformal algebras such as the notions of
representations, matched pairs and Manin triples of Poisson
conformal algebras, and the equivalence between Manin triples of
Poisson conformal algebras and some matched pairs of Poisson
conformal algebras. In Section 3, we introduce the notion of
Poisson conformal bialgebras characterized by Manin triples of
Poisson conformal algebras. We also give a detailed study on a
special class of Poisson conformal bialgebras, namely, coboundary
Poisson conformal bialgebras. In Section 4, we show that the
semi-classical limit of a conformal formal deformation of a
commutative and cocommutative ASI conformal bialgebra is a Poisson
conformal bialgebra. In Section 5, we introduce the notion of
PGD-bialgebras and show that they correspond to a class of Poisson
conformal bialgebras. The notion of pre-PGD-algebras is introduced
and hence a construction of Poisson conformal bialgebras from
pre-PGD-algebras is given.

 \vspace{0.1cm}
\noindent {\bf Notations.} Let ${\bf k}$ be a field of
characteristic $0$ and let $\bfk[\partial]$ be the polynomial
algebra with the variable $\partial$. Let $\mathbb{N}$ be the set of non-negative integers and $\mathbb{Z}$ be the set of integers. All vector spaces and
algebras over ${\bf k}$ (resp. all ${\bf k}[\partial]$-modules and
 conformal algebras) are assumed to be finite-dimensional
(resp. finite) unless otherwise stated, even though many results
still hold in the infinite-dimensional (resp. infinite) cases. For
a vector space $V$, let
$$\tau:V\otimes V\rightarrow V\otimes V,\quad u\otimes v\mapsto v\otimes u,\;\;\;u,v\in V,
$$
be the flip operator and let $V[\lambda]$ denote the set of
polynomials of $\lambda$ with coefficients in $V$. The identity map is denoted by $I$. Let $A$ be a
vector space with a binary operation $\circ$. Define linear maps
$L_{A,\circ}, R_{A,\circ}:A\rightarrow {\rm End}_{\bf k}(A)$
respectively by
\begin{eqnarray*}
L_{A,\circ}(a)b:=a\circ b,\;\; R_{A,\circ}(a)b:=b\circ a, \;\;\;a, b\in A.
\end{eqnarray*}
In particular, when $(A,\circ=[\cdot,\cdot])$ is a Lie algebra, we use the usual notion of the adjoint operator $\ad_A(a)(b):=[a,b]$ for all $a,b\in A$. Let $P={\bf k}[\partial]\otimes_{\bf k} A$ be a free ${\bf k }[\partial]$-module. For convenience, we denote the element $f(\partial)\otimes a\in P$ by $f(\partial)a$, where $f(\partial)\in {\bf k }[\partial]$ and $a\in A$.
\section{Preliminaries on Poisson conformal algebras}

We recall some basic facts on Poisson conformal algebras and
related structures.

\subsection{Poisson conformal algebras and Poisson-Gel'fand-Dorfman algebras}
\begin{defi}\cite{K1}
{\rm An {\bf associative conformal algebra} $(P, \cdot_\lambda
\cdot)$ is a ${\bf k}[\partial]$-module $P$ endowed with a ${\bf
k}$-bilinear map $\cdot_\lambda \cdot: P\times P\rightarrow
 P[\lambda]$ \delete{$(a, b)\mapsto
a_{\lambda} b$ \cm{whether such notation is necessary? If so, for
Lie conformal algebra, whether we also need such a notation?}}
satisfying the following conditions
\begin{eqnarray}
&&(\partial a)_{\lambda}b=-\lambda a_{\lambda}b,  \quad a_{\lambda}(\partial b)=(\partial+\lambda)a_{\lambda}b,\\
&&(a_{\lambda}b)_{\lambda+\mu}c=a_{\lambda}(b_\mu c),\;\;\;
a, b, c\in P.
\end{eqnarray}
An associative conformal algebra $(P, \cdot_\lambda \cdot)$ is
called {\bf commutative}, if it satisfies the following condition
\begin{eqnarray*}
a_\lambda b=b_{-\lambda-\partial}a, \;\;a, b\in P.
\end{eqnarray*}

A {\bf Lie conformal algebra} $(P, [\cdot_\lambda \cdot])$ is a
${\bf k}[\partial]$-module $P$ endowed with a ${\bf k}$-bilinear
map $[\cdot_\lambda \cdot]: P\times P\rightarrow  P[\lambda]$
satisfying the following conditions
\begin{eqnarray}
&&[(\partial a)_{\lambda}b]=-\lambda [a_{\lambda}b],  \quad [a_{\lambda}(\partial b)]=(\partial+\lambda)[a_{\lambda}b],\\
&&[a_\lambda b]=-[b_{-\lambda-\partial}a],\\
&&[a_\lambda[b_\mu c]]=[[a_\lambda b]_{\lambda+\mu} c]+[b_\mu[a_\lambda c]],\;\;\;a, b, c\in P.
\end{eqnarray}
}
\end{defi}

A Lie conformal algebra (or an associative conformal algebra) is called {\bf finite}, if it is finitely generated as a ${\bf k}[\partial]$-module.

\begin{defi}\cite{Ko4}
{\rm If $(P,[\cdot_\lambda\cdot])$ is a Lie conformal algebra, $(P,\cdot_\lambda\cdot)$ is a commutative associative conformal algebra,
and they satisfy the following condition
\begin{eqnarray}
[a_\lambda(b_\mu c)]=[a_\lambda b]_{\lambda+\mu}c+b_\mu[a_\lambda c],\;\;a, b, c\in P,
\end{eqnarray}
then $(P,[\cdot_\lambda\cdot],\cdot_\lambda\cdot)$ is called a {\bf Poisson conformal algebra}.}
\end{defi}

Recall that a {\bf left-symmetric algebra} $(A, \circ)$ is a
vector space with a binary operation $\circ$ satisfying the
following condition
\begin{eqnarray}
&&(a\circ b)\circ c-a\circ (b\circ c)=(b\circ a)\circ c-b\circ (a\circ c),\;\;a, b, c\in A.
\end{eqnarray}
A {\bf Novikov algebra} $(A, \circ)$ is a left-symmetric algebra
satisfying the following condition
\begin{eqnarray}
&&(a\circ b)\circ c=(a\circ c)\circ b, \;\;\;a, b, c\in A.
\end{eqnarray}

Recall \cite{X1} that a {\bf Gel'fand-Dorfman algebra (GD-algebra)} is a triple $(A, \circ, [\cdot,\cdot])$ where $(A, [\cdot,\cdot])$ is a Lie algebra, $(A,\circ)$ is a Novikov algebra
and the following compatibility condition holds.
\begin{eqnarray}\label{eqq3}
[a\circ b, c]-[a\circ c, b]+[a,b]\circ c-[a,c]\circ b-a\circ [b,c]=0, \;\;\;a, b, c\in A.
\end{eqnarray}

Recall that a {\bf Poisson algebra} is a triple $(A, \cdot, [\cdot,\cdot])$ where $(A, \cdot)$ is a commutative associative algebra, $(A, [\cdot,\cdot])$ is a Lie algebra and the following compatibility condition holds.
\begin{eqnarray*}
[a, b\cdot c]=[a,b]\cdot c+b\cdot [a,c],\;\;a, b, c\in A.
\end{eqnarray*}

\begin{defi} \cite{BCZ} A {\bf differential Novikov-Poisson algebra} is a triple $(A, \circ, \cdot)$ where $(A, \circ)$ is a Novikov algebra,
$(A, \cdot)$ is a commutative associative algebra and they satisfy the following conditions
\begin{eqnarray}
\label{eqNP1}&&(a\cdot b)\circ c=a\cdot (b\circ c),\\
\label{eqNP2}&&a\circ (b\cdot c)=(a\circ b)\cdot c+b\cdot (a\circ c),\;\; a, b, c\in A.
\end{eqnarray}
\end{defi}

\begin{rmk}
{\rm In fact, by Eqs. (\ref{eqNP1}) and (\ref{eqNP2}), we have
\begin{eqnarray}
\label{eqNP3}(a\circ b)\cdot c-(b\circ a)\cdot c=a\circ (b\cdot c)-b\circ (a\cdot c),\;\;\;a, b, c\in A.
\end{eqnarray}
Note that a triple $(A, \circ, \cdot)$, where $(A, \circ)$ is a
Novikov algebra, $(A, \cdot)$ is a commutative associative algebra
and they satisfy Eqs. (\ref{eqNP1}) and (\ref{eqNP3}), is called a
{\bf Novikov-Poisson algebra} in \cite{X2}. Therefore, a
differential Novikov-Poisson algebra is a special Novikov-Poisson
algebra.}
\end{rmk}

\begin{defi}\cite{LZ}
A {\bf Poisson-Gel'fand-Dorfman algebra (PGD-algebra)} is a quadruple $(A, \circ, \cdot, [\cdot,\cdot])$ such that $(A, \circ, [\cdot,\cdot])$ is a GD-algebra, $(A, \cdot, [\cdot,\cdot])$ is a Poisson algebra and $(A, \circ, \cdot)$ is a differential Novikov-Poisson algebra.
\end{defi}

Note that GD-algebras, Poisson algebras and differential
Novikov-Poisson algebras are special PGD-algebras. In particular,
$(A, \circ, \cdot, [\cdot,\cdot])$ with $[\cdot,\cdot]$ being
trivial is  a PGD-algebra if and only if $(A,\circ,\cdot)$ is a
differential Novikov-Poisson algebra.

%Next, we recall a construction of Poisson conformal algebras.

\begin{pro}\cite[Proposition 2.13]{LZ}\label{constr-PCA} Let $A$
be a vector space. Let $P={\bf k}[\partial]\otimes_{\bf k} A$ be a free ${\bf
k}[\partial]$-module, and $\cdot$, $\circ$, $[\cdot,\cdot]$,
$\star$ be four binary operations on $A$. Define
\begin{eqnarray}
\label{wq2}[a_\lambda b]:=\partial(b\circ a)+\lambda(b\star
a)+[a,b],\;\;a_\lambda b:=a\cdot b,\ \;\; a,b\in A.
\end{eqnarray}
 Then
$(P,[\cdot_\lambda\cdot],\cdot_\lambda\cdot)$ is a Poisson
conformal algebra if and only if $(A, \circ, \cdot,
[\cdot,\cdot])$  is a PGD-algebra and $a\star b=a\circ b+b\circ a$
for all $a$, $b\in A$. In this case,
$(P,[\cdot_\lambda\cdot],\cdot_\lambda\cdot)$ is called the {\bf
Poisson conformal algebra corresponding to $(A, \circ, \cdot,
[\cdot,\cdot])$}.
%$(A,[\cdot,\cdot],\cdot)$ is a Poisson algebra, $(A,\circ, [\cdot,\cdot])$ is a GD-algebra and the following
%\begin{eqnarray}
%\label{wq7}(a\cdot b)\circ c=a\cdot (b\circ c),\\
%\label{wq8} a\circ (b\cdot c)=(a\circ b)\cdot c+b\cdot (a\circ c),\;\;\; a, b, c\in A.
%\end{eqnarray}
\end{pro}

\begin{cor} Let $A$
be a vector space. Let $P={\bf k}[\partial]\otimes_{\bf k}A$ be a free ${\bf
k}[\partial]$-module, and $\cdot$, $\circ$, $\star$ be three
binary operations on $A$. Define
\begin{eqnarray}
[a_\lambda b]:=\partial(b\circ a)+\lambda(b\star a),\;\;a_\lambda b:=a\cdot b,\ \;\; a,b\in A.
\end{eqnarray}
Then $(P,[\cdot_\lambda\cdot],\cdot_\lambda\cdot)$ is a Poisson
conformal algebra if and only if $(A,\circ, \cdot)$ is a
differential Novikov-Poisson algebra and $a\star b=a\circ b+b\circ
a$ for all $a$, $b\in A$. In this case,
$(P,[\cdot_\lambda\cdot],\cdot_\lambda\cdot)$ is called the {\bf
Poisson conformal algebra corresponding to $(A,\circ, \cdot)$}.
\end{cor}
\begin{proof}
It follows from Proposition \ref{constr-PCA} straightforwardly.
\end{proof}
\delete{
There is a natural construction of differential Novikov-Poisson algebras.
\begin{pro} \cite[Proposition 1]{Ko4}
Let $(A, \cdot)$ be a commutative associative algebra with a derivation $D$. Define a binary operation $\circ: A\otimes A\rightarrow A$ as follows
\begin{eqnarray*}
a\circ b:=a\cdot D(b),\;\;\;a, b\in A.
\end{eqnarray*}
Then $(A, \circ, \cdot)$ is a differential Novikov-Poisson algebra.
\end{pro}}

\subsection{Representations, matched pairs and Manin triples of Poisson conformal algebras}
\begin{defi}\cite{K1}
 {\rm Let $U$ and $V$ be two ${\bf k}[\partial]$-modules. A {\bf conformal linear map} from $U$ to $V$ is a ${\bf k}$-linear map $f: U\rightarrow V[\lambda]$, denoted by $f_\lambda: U\rightarrow V$, such that $[\partial, f_\lambda]=-\lambda f_\lambda$. Denote the ${\bf k}$-vector space of all such maps by $\text{Chom}(U,V)$. It has a canonical structure of a ${\bf k}[\partial]$-module
$$(\partial f)_\lambda =-\lambda f_\lambda.$$ Define the {\bf conformal dual} of a ${\bf k}[\partial]$-module $U$ as $U^{\ast c}=\text{Chom}(U,{\bf k})$, where ${\bf k}$ is viewed as the trivial ${\bf k}[\partial]$-module, that is
$$U^{\ast c}=\{f:U\rightarrow {\bf k}[\lambda]~~|~~\text{$f$ is ${\bf k}$-linear and}~~f_\lambda(\partial b)=\lambda f_\lambda b\}.$$}
\end{defi}

Let $V$ be a finitely generated ${\bf k}[\partial]$-module. Set
$\text{gc}(V):=\text{Chom}(V,V)$. %Then $\text{gc}(V)$ has a Lie
%conformal algebra structure defined by
%\begin{eqnarray}
%[f_\lambda g]_\mu v=f_\lambda (g_{\mu-\lambda} v)-g_{\mu-\lambda}(f_{\lambda} v), ~~~ f,~g\in \text{gc}(V), v\in V.\end{eqnarray}
%$\text{gc}(V)$ is called the {\bf general Lie conformal algebra} of $V$.

\begin{defi}\cite{HB1}
{\rm A {\bf representation of a commutative associative conformal
algebra $(P, \cdot_\lambda \cdot)$} is a pair $(V, l_P)$, where
$V$ is a ${\bf k}[\partial]$-module and $l_P$ is a ${\bf
k}[\partial]$-module homomorphism from $P$ to $\text{gc}(V)$ such
that the following condition holds.
\begin{eqnarray}
l_P(a_\lambda b)_{\lambda+\mu}v=l_P(a)_\lambda (l_P(b)_\mu v),\;\;a, b\in P, v\in V.
\end{eqnarray}}
\end{defi}

%A {\bf module $M$ over a Lie conformal algebra $R$} is a $\mathbb{C}[\partial]$-module endowed with a $\mathbb{C}$-bilinear map
%$R\times M\longrightarrow M[\lambda]$, $(a, v)\mapsto a_\lambda v$, satisfying the following axioms $(a, b\in R, v\in M)$:\\
%(LM1)$\qquad\qquad (\partial a)_\lambda v=-\lambda a_\lambda v,~~~a_\lambda(\partial v)=(\partial+\lambda)a_\lambda v,$\\
%(LM2)$\qquad\qquad [a_\lambda b]_{\lambda+\mu}v=a_\lambda(b_\mu v)-b_\mu(a_\lambda v).$

%Let $V$ be a finite module over a Lie conformal algebra $R$. Then the $\mathbb{C}[\partial]$-module $gc(V):=Chom(V,V)$ has a Lie conformal algebra structure defined by
%\begin{eqnarray}
%[a_\lambda b]_\mu v=a_\lambda (b_{\mu-\lambda} v)-b_{\mu-\lambda}(a_{\lambda} v), ~~~ a,~b\in gc(V), v\in V.\end{eqnarray}
%$gc(V)$ is called the {\bf general Lie conformal algebra} of $V$.

\begin{defi}\cite{K1}
{\rm A {\bf representation of a Lie conformal algebra $(P,
[\cdot_\lambda \cdot])$} is a pair $(V, \rho_P)$, where $V$ is a
${\bf k}[\partial]$-module and $\rho_P$ is a ${\bf
k}[\partial]$-module homomorphism from $P$ to $\text{gc}(V)$ such
that the following condition holds.
\begin{eqnarray}
\rho_P([a_\lambda b])_{\lambda+\mu}v=\rho_P(a)_\lambda (\rho_P(b)_\mu v)-\rho_P(b)_\mu(\rho_P(a)_\lambda v),\;\;\;a, b\in P, v\in V.
\end{eqnarray}
Denote the {\bf adjoint representation} of $(P, [\cdot_\lambda
\cdot])$ by $(P,\mathfrak{ad}_P)$, where $\mathfrak{ad}_P(a)_\lambda b=[a_\lambda b]$,
for all $a$, $b\in P$.}
\end{defi}

%Next, we recall the definition of representations of Poisson conformal algebras.

\begin{defi}\cite{LZ}
{\rm  A {\bf representation of a Poisson conformal algebra
$(P,[\cdot_\lambda \cdot],\cdot_\lambda\cdot)$} is a triple $(V,
\rho_P, l_P)$ such that $(V,\rho_P)$ is a representation of
$(P,[\cdot_\lambda \cdot])$ and
 $(V,l_P)$ is a representation of $(P,\cdot_\lambda \cdot)$
 satisfying the following conditions.
 \begin{eqnarray}
 &&\label{eq-PM1}{\rho_P(a_\mu b)}_{-\lambda-\partial}v={l_P(b)}_{-\lambda-\mu-\partial}({\rho_P(a)}_{-\lambda-\partial} v)+{l_P(a)}_\mu(\rho_P(b)_{-\lambda-\partial}v),\\
 &&\label{eq-PM2}{\rho_P(a)}_\lambda ({l_P(b)}_\mu v)=l_P([a_\lambda b])_{\lambda+\mu}v+{l_P(b)}_\mu({\rho_P(a)}_\lambda v),\;\;\;\;a, b\in P, v\in V.
 \end{eqnarray}
}
\end{defi}

\begin{pro}\label{Poisson-dual-module}
Let $(V,\rho_P, l_P)$ be a representation of a Poisson conformal
algebra $(P,[\cdot_\lambda \cdot],\cdot_\lambda\cdot)$, in which $V$ and
$P$ are free as ${\bf k}[\partial]$-modules. %\cm{$V$ is also
%free?}
Let $\rho_P^\ast$ and $l_P^\ast$ be two ${\bf
k}[\partial]$-module homomorphisms from $P$ to $\text{gc}(V^{\ast
c})$ defined respectively by
\begin{eqnarray*}
({\rho_P^\ast(a)}_\lambda f)_\mu u=-f_{\mu-\lambda}({\rho_P(a)}_\lambda u),~~~({l_P^\ast(a)}_\lambda f)_\mu u=-f_{\mu-\lambda}({l_P(a)}_\lambda u),\;\;\;a\in P, f\in V^{\ast c}.
\end{eqnarray*}
Then $(V^{\ast c}, \rho_P^\ast, -l_P^\ast)$ is a representation of $(P,[\cdot_\lambda \cdot],\cdot_\lambda\cdot)$.
\end{pro}
\begin{proof}
Note that $(V^{\ast c}, \rho_P^\ast)$ is a representation of $(P,
[\cdot_\lambda \cdot])$ \cite{K1} and $(V^{\ast c}, -l_P^\ast)$ is
a representation of $(P, \cdot_\lambda\cdot)$ by \cite[Proposition
2.5]{HB1}. Let $a$, $b\in P$, $f\in V^{\ast c}$ and $u\in V$. Then
we have
\begin{eqnarray*}
&&({\rho_P^\ast(a_\mu b)}_{-\lambda-\partial}f+{l_P^\ast(b)}_{-\lambda-\mu-\partial}(\rho_P^\ast(a)_{-\lambda-\partial}f)
+l_P^\ast(a)_\mu(\rho_P^\ast(b)_{-\lambda-\partial}f))_\nu u\\
&=&-f_\lambda(\rho_P(a_\mu b)_{-\lambda+\nu}u-\rho_P(a)_\mu(l_P(b)_{-\lambda-\mu+\nu}u)-\rho_P(b)_{-\lambda-\mu+\nu}(l_P(a)_\mu u))\\
&=&-f_\lambda(\rho_P(a_\mu b)_{\nu-\partial}u-l_P(b)_{\nu-\mu-\partial}(\rho_P(a)_{\nu-\partial}u)-l_P(a)_\mu(\rho_P(b)_{\nu-\partial}u))\\
&=& 0.
\end{eqnarray*}
Therefore, ${\rho_P^\ast(a_\mu
b)}_{-\lambda-\partial}f+{l_P^\ast(b)}_{-\lambda-\mu-\partial}(\rho_P^\ast(a)_{-\lambda-\partial}f)+l_P^\ast(a)_\mu(\rho_P^\ast(b)_{-\lambda-\partial}f)=0$.
Similarly, we show that ${\rho_P^\ast(a)}_\lambda
({l_P^\ast(b)}_\mu f)=l_P^\ast([a_\lambda
b])_{\lambda+\mu}f+{l_P^\ast(b)}_\mu({\rho_P^\ast(a)}_\lambda f)$.
Then the proof is completed.\end{proof}
\begin{ex}
Let $(P,[\cdot_\lambda \cdot],\cdot_\lambda\cdot)$ be a Poisson
conformal algebra in which $P$ is free as a ${\bf
k}[\partial]$-module. Set ${\mathfrak{L}_P(a)}_\lambda b=a_\lambda b$ for all
$a$, $b\in P$. Then $(P, \mathfrak{ad}_P, \mathfrak{L}_P)$ is a representation of
$(P,[\cdot_\lambda \cdot],\cdot_\lambda\cdot)$. By Proposition
\ref{Poisson-dual-module}, $(P^{\ast c}, \mathfrak{ad}_P^\ast, -\mathfrak{L}_P^\ast)$
is also a representation of $(P,[\cdot_\lambda
\cdot],\cdot_\lambda\cdot)$.
\end{ex}

Next, we recall the notions of matched pairs of commutative
associative conformal algebras and Lie conformal algebras and then
give the notion of matched pair of Poisson conformal algebras.

\begin{defi} \cite{H1}
Let $(P,\cdot_\lambda \cdot)$ and $(Q,\cdot_\lambda \cdot)$ be two commutative associative conformal algebras. Let $l_P: P\rightarrow \text{gc}(Q)$ and $l_Q: Q\rightarrow \text{gc}(P)$ be ${\bf k}[\partial]$-module homomorphisms. If there is a commutative associative conformal algebra structure on the ${\bf k}[\partial]$-module $P\oplus Q$ as the direct sum of ${\bf k}[\partial]$-modules given by
\begin{eqnarray}\label{ass-matched pair}
(a+x)_\lambda (b+y)=(a_\lambda b+l_Q(x)_\lambda
b+l_Q(y)_{-\lambda-\partial} a) +(x_\lambda y+l_P(a)_\lambda
y+l_P(b)_{-\lambda-\partial} x),
\end{eqnarray}
for all $a,b\in P$ and $x,y\in Q$, then $(P, Q,l_P, l_Q)$ is
called a {\bf matched pair  of commutative associative conformal
algebras}.
\end{defi}
\begin{defi} \cite{HL}
Let $(P, [\cdot_\lambda \cdot])$ and $(Q, [\cdot_\lambda \cdot])$ be two Lie conformal algebras. Let $\rho_P: P\rightarrow \text{gc}(Q)$ and $\rho_Q: Q\rightarrow \text{gc}(P)$ be ${\bf k}[\partial]$-module homomorphisms.
If there is a Lie conformal algebra structure on the ${\bf k}[\partial]$-module $P\oplus Q$ as the direct sum of ${\bf k}[\partial]$-modules given by
\begin{eqnarray}\label{107}
&&[(a+x)_\lambda (b+y)]=([a_\lambda b]+\rho_Q(x)_\lambda b-\rho_Q(y)_{-\lambda-\partial}a)+([x_\lambda y]+\rho_P(a)_\lambda y-\rho_P(b)_{-\lambda-\partial} x),
\end{eqnarray}
for all $a$, $b\in P$ and $x$, $y\in Q$, then $(P,Q,\rho_P,\rho_Q)$ is called a {\bf matched pair of Lie conformal algebras}.
\end{defi}
\delete{\begin{pro}\label{pro1} {\rm (\cite[Proposition 4.4]{H1})}
Let $(A,\cdot_\lambda \cdot)$ and $(B,\cdot_\lambda \cdot)$ be two commutative associative conformal algebras. Suppose that there are ${\bf k}[\partial]$-module homomorphisms $l_A: A\rightarrow \text{gc}(B)$ and $l_B: B\rightarrow \text{gc}(A)$ such that $(B, l_A)$ is a module of $(A,\cdot_\lambda \cdot)$ and $(A, l_B)$ is a module of $(B,\cdot_\lambda \cdot)$ and they satisfy the following compatibility conditions
\begin{eqnarray}
\label{es1}&&l_A(a)_\lambda(x_\mu y)=(l_A(a)_\lambda x)_{\lambda+\mu}y
+l_A(r_B(x)_{-\lambda-\partial}a)_{\lambda+\mu}y,\\
\label{es3}&&l_B(x)_\lambda(a_\mu b)=(l_B(x)_\lambda a)_{\lambda+\mu} b
+l_B(r_A(a)_{-\lambda-\partial}x)_{\lambda+\mu}b,\;\;\;a, b\in A, x, y\in B.
\end{eqnarray}
Then there is a commutative associative conformal algebra structure on the ${\bf k}[\partial]$-module $A\oplus B$ as the direct sum of ${\bf k}[\partial]$-modules given by
\begin{eqnarray}\label{ass-matched pair}
(a+x)_\lambda (b+y)=(a_\lambda b+l_B(x)_\lambda
b+l_B(y)_{-\lambda-\partial} a) +(x_\lambda y+l_A(a)_\lambda
y+l_A(b)_{-\lambda-\partial} x),\end{eqnarray} for all $a,b\in A$
and $x,y\in B$. We denote this commutative associative conformal algebra by
$A\bowtie B$. $(A, B,l_A, l_B)$ satisfying the above
conditions is called a {\bf matched pair } of commutative associative conformal
algebras. \delete{Moreover, any commutative associative conformal algebra $E=A\oplus
B$ where the sum is the direct sum of
$\mathbb{C}[\partial]$-modules and $A$, $B$ are two commutative associative
conformal subalgebras of $E$, is $A\bowtie B$ associated to some
matched pair of commutative associative conformal algebras.}
\end{pro}}
\delete{\begin{pro} \cite[Theorem 3.2]{HL}
Let $(A, [\cdot_\lambda \cdot])$ and $(B, [\cdot_\lambda \cdot])$ be two Lie conformal algebras. Suppose that $(B,\rho_A)$ ia a module of $(A, [\cdot_\lambda \cdot])$,  $(A,\rho_B)$ is a module of $(B, [\cdot_\lambda \cdot])$, and they satisfy the following compatibility conditions:
\begin{eqnarray}\label{108}
&&\rho_A(a)_\lambda[x_\mu y]-[(\rho_A(a)_\lambda x)_{\lambda+\mu}y]-[x_\mu(\rho_A(a)_\lambda y)]+\rho_A({\rho_B(x)}_{-\lambda-\partial}a)_{\lambda+\mu}y\\
&&\qquad\qquad-\rho_A(\rho_B(y)_{-\lambda-\partial}a)_{-\mu-\partial}x=0,\nonumber\\
\label{109}
&&{\rho_B(x)}_{-\lambda-\mu-\partial}[a_\lambda b]-[a_\lambda ({\rho_B(x)}_{-\mu-\partial}b)]+[b_\mu({\rho_B(x)}_{-\lambda-\partial}a)]\\
&&\qquad\qquad+\rho_B(\rho_A(a)_{\lambda}x)_{-\mu-\partial}b
-\rho_B(\rho_A(b)_{\mu}x)_{-\lambda-\partial}a=0,\;\;a, b\in A, x, y\in B.\nonumber
\end{eqnarray} Then there is a Lie conformal algebra structure on the ${\bf k}[\partial]$-module $A\oplus B$ given by
\begin{eqnarray}\label{107}
[(a+x)_\lambda (b+y)]=([a_\lambda b]+\rho_B(x)_\lambda b-\rho_B(y)_{-\lambda-\partial}a)+([x_\lambda y]+\rho_A(a)_\lambda y-\rho_A(b)_{-\lambda-\partial} x),
\end{eqnarray}
for all $a$, $b\in A$ and $x$, $y\in B$. We denote this Lie conformal algebra by $A\bowtie B$. $(A,B,\rho_A,\rho_B)$ satisfying the above conditions is called a {\bf matched pair of Lie conformal algebras}.
\end{pro}}

\begin{defi}
Let $(P, [\cdot_\lambda \cdot], \cdot_\lambda \cdot)$ and  $(Q,
[\cdot_\lambda \cdot], \cdot_\lambda \cdot)$ be two Poisson
conformal algebras. Let $\rho_{P}$, $l_{P}: P\rightarrow
\text{gc}(Q)$ and $\rho_{Q}$, $l_{Q}: Q\rightarrow \text{gc}(P)$
be four ${\bf k}[\partial]$-module homomorphisms. If there is a
Poisson conformal algebra structure on the ${\bf
k}[\partial]$-module $P\oplus Q$ as the direct sum of ${\bf
k}[\partial]$-modules defined by Eqs. (\ref{ass-matched pair}) and
(\ref{107}), then $(P,Q,\rho_P, l_P, \rho_Q, l_Q)$ is called a
{\bf matched pair of Poisson conformal algebras}. We denote this
Poisson conformal algebra by
$P\bowtie_{\rho_P,l_P}^{\rho_Q,l_Q}Q$.
\end{defi}
\begin{pro}\label{pro:match pair} The 6-tuple
$(P,Q,\rho_P, l_P, \rho_Q, l_Q)$ is  a matched pair of Poisson conformal algebras if and only if
$(P, Q, \rho_P, \rho_Q)$ is a matched pair of Lie conformal algebras, $(P, Q, l_P, l_Q)$ is a matched pair of commutative associative conformal algebras,  $(P, \rho_Q, l_Q)$ is a representation of $(Q, [\cdot_\lambda \cdot], \cdot_\lambda \cdot)$, $(Q, \rho_P, l_P)$ is a representation of $(P, [\cdot_\lambda \cdot], \cdot_\lambda \cdot)$, and they satisfy the following compatibility conditions
\begin{eqnarray}
&&\label{poisson-matched-1}{l_Q(x)}_{-\lambda-\mu-\partial}([a_\lambda b])=[a_\lambda ({l_Q(x)}_{-\mu-\partial}b)]+b_\mu({\rho_Q(x)}_{-\lambda-\partial} a)\\
&&\qquad\qquad\qquad\quad\quad-\rho_Q({l_P(b)}_\mu x)_{-\lambda-\partial}a-l_Q({\rho_P(a)}_\lambda x)_{-\mu-\partial} b,\nonumber\\
&&\label{poisson-matched-2}{\rho_Q(x)}_\lambda (a_\mu b)=({\rho_Q(x)}_\lambda a)_{\lambda+\mu} b+a_\mu ({\rho_Q(x)}_\lambda b)\\
&&\qquad\qquad\qquad-l_Q({\rho_P(a)}_{-\lambda-\partial}x)_{\lambda+\mu}b-l_Q(\rho_P(b)_{-\lambda-\partial}x)_{-\mu-\partial}a,\nonumber\\
&&\label{poisson-matched-3}{l_P(a)}_\mu([x_\mu y])=[x_\lambda({l_P(a)}_\mu y)]+({\rho_P(a)}_{-\lambda-\partial}x)_{\lambda+\mu}y\\
&&\qquad\qquad\qquad-\rho_P({l_Q(y)}_{-\mu-\partial}a)_{-\lambda-\partial}x-l_P({\rho_Q(x)}_\lambda a)_{\lambda+\mu}y,\nonumber\\
&&\label{poisson-matched-4}{\rho_P(a)}_\lambda (x_\mu y)=({\rho_P(a)}_\lambda x)_{\lambda+\mu}y+x_\mu({\rho_P(a)}_\lambda y)\\
&&\qquad\qquad\qquad-l_P({\rho_Q(x)}_{-\lambda-\partial}
a)_{\lambda+\mu}y-l_P({\rho_Q(y)}_{-\lambda-\partial}a)_{-\mu-\partial}
x,\;\;a, b\in P, x, y\in Q.\nonumber
\end{eqnarray}
 Moreover, any Poisson conformal algebra that is decomposed into a ${\bf k}[\partial]$-module direct sum of two Poisson conformal subalgebras is obtained from a matched pair of Poisson conformal algebras.
\end{pro}
\begin{proof}
It is straightforward.
\end{proof}

\begin{rmk}
If $\rho_Q$, $l_Q$ and the Poisson conformal algebra structure on
$Q$ are trivial, then $P\bowtie_{\rho_P,l_P}^{\rho_Q,l_Q}Q$ is
exactly the {\bf semi-direct product} of $(P, [\cdot_\lambda
\cdot], \cdot_\lambda \cdot)$ and the representation $(Q, \rho_P,
l_P)$, which is denoted by $P\ltimes_{\rho_P,l_P} Q$.
\end{rmk}

At the end of this section, we introduce the notion of Manin
triples of Poisson conformal algebras and interpret them in terms
of matched pairs of Poisson conformal algebras.

 \begin{defi}\cite{L} {
\rm Let $V$ be a ${\bf k}[\partial]$-module. A {\bf conformal
bilinear form} on $V$ is a ${\bf k}$-bilinear map $\langle \cdot,
\cdot \rangle_\lambda: V\times V\rightarrow {\bf k}[\lambda]$
satisfying the following conditions
\begin{eqnarray}
\langle \partial a, b\rangle_\lambda=-\lambda\langle a,
b\rangle_\lambda, ~~~\langle a, \partial
b\rangle_\lambda=\lambda\langle a, b\rangle_\lambda,\;\;a,b\in V.
\end{eqnarray}
A conformal bilinear form is called {\bf symmetric} if $\langle
a,b\rangle_\lambda=\langle b,a\rangle_{-\lambda}$ for all $a$,
$b\in V$. }
\end{defi}

\begin{defi}
A {\bf Manin triple} of Poisson conformal algebras is a triple of
Poisson conformal algebras $(P, P_1, P_2)$ where $P$ is equipped
with a nondegenerate symmetric conformal bilinear form $\langle
\cdot, \cdot \rangle_\lambda$ which is also {\bf invariant} in the
sense that
\begin{eqnarray*}
\langle[a_\lambda b], c\rangle_\mu=\langle a, [b_{\mu-\partial} c]\rangle_\lambda,~~~\langle a_\lambda b, c\rangle_\mu=\langle a, b_{\mu-\partial} c\rangle_\lambda,\;\;a, b, c\in P,
\end{eqnarray*}
such that
\begin{enumerate}
\item[(1)] $P_1$ and $P_2$ are Poisson conformal subalgebras of
$P$ and $P=P_1\oplus P_2$ as ${\bf k}[\partial]$-modules;
\item[(2)] $P_1$ and $P_2$ are isotropic with respect to $\langle
\cdot, \cdot \rangle_\lambda$, i.e. $\langle P_i, P_i
\rangle_\lambda=0$ for $i=0$, $1$.
\end{enumerate}

\end{defi}
\begin{thm}\label{poisson-thm1}
Let $(P, [\cdot_\lambda \cdot], \cdot_\lambda \cdot)$ be a Poisson
conformal algebra in which $P$ is free as a ${\bf
k}[\partial]$-module. Suppose that $(P^{\ast c}, [\cdot_\lambda
\cdot], \cdot_\lambda \cdot)$ is also a Poisson conformal algebra.
Then $(P\oplus P^{\ast c}, P, P^{\ast c})$ is a Manin triple of
Poisson conformal algebras with the invariant symmetric conformal bilinear
form $\langle \cdot, \cdot\rangle_\lambda$ on $P\oplus P^{\ast c}$
given by
\begin{eqnarray}\label{poisson-inv}
\langle a+f, b+g\rangle_\lambda=f_\lambda(b)+g_{-\lambda}(a),\;\;\;\;a, b\in P, ~~f, g\in P^{\ast c},
\end{eqnarray}
if and only if $(P, P^{\ast c}, \mathfrak{ad}_P^{\ast}, -\mathfrak{L}_P^\ast, \mathfrak{ad}_{P^{\ast c}}^\ast, -\mathfrak{L}_{P^{\ast c}}^\ast)$ is a matched pair of Poisson conformal algebras.
\end{thm}

\begin{proof}
It follows from a similar proof as the one for \cite[Theorem
3.6]{HB1}.
\end{proof}

\section{Poisson conformal bialgebras and coboundary Poisson conformal bialgebras}
 We introduce the notion of Poisson conformal bialgebras and
 then give the equivalence between them and certain Manin triples of Poisson conformal algebras. Moreover, the theory of coboundary Poisson conformal bialgebras is studied in detail.

 \delete{First, we recall the definitions of cocommutative associative conformal coalgebras, commutative and cocommutative infinitesimal conformal bialgebras, Lie conformal coalgebras and Lie conformal bialgebras.}
\subsection{Poisson conformal bialgebras}\mbox{}

For an associative conformal algebra $(P, \cdot_\lambda\cdot)$, define a ${\bf k}[\partial]$-module homomorphism $\mathfrak{R}_P: P\rightarrow \text{gc}(P)$ as follows.
\begin{eqnarray*}
\mathfrak{R}_P(a)_\lambda(b):=b_{-\lambda-\partial} a,\;\;a, b\in P.\end{eqnarray*}
Set \begin{eqnarray*}
\partial^{\otimes^2}:=\partial\otimes I+I\otimes \partial,\;\;\partial^{\otimes^3}:=\partial\otimes I\otimes I+I\otimes \partial\otimes I+I\otimes I\otimes \partial.
\end{eqnarray*}

 \begin{defi}\cite{HB1} {\rm
 An {\bf associative conformal coalgebra} is a ${\bf k}[\partial]$-module $P$ endowed with a ${\bf k}[\partial]$-module homomorphism $\Delta: P\longrightarrow P\otimes P$
\delete{\cm{For algebra, it is a ${\bf k}$-bilinear map $P\otimes
P\rightarrow P[\lambda]$, whereas for a coproduct, it is  ${\bf
k}[\partial]$-module homomorphism $\Delta: P\longrightarrow
P\otimes P$. It is correct or a ``standard" expression? Whether we
can unify the notations?} \yy{It is a ``standard" expression.}}
 such
 that the following equation holds.
\begin{eqnarray}
&&\label{eq:coass}(I\otimes\Delta)\Delta(a)=(\Delta\otimes I)\Delta(a), \;\;a\in P,
\end{eqnarray}
where the module action of ${\bf k}[\partial]$ on $P\otimes P$ is defined as
$\partial (a\otimes b)=(\partial a)\otimes b+a\otimes (\partial b)$ for all $a$, $b\in P$. Denote it by $(P, \Delta)$. An associative conformal coalgebra $(P, \Delta)$ is called {\bf cocommutative} if $\Delta(a)=\tau\Delta(a)$ for all $a\in P$.

An {\bf antisymmetric infinitesimal conformal bialgebra (ASI
conformal bialgebra)} is a triple $(P, \cdot_\lambda\cdot,
\Delta)$ such that $(P,\cdot_\lambda \cdot)$ is an associative
conformal algebra, $(P,\Delta)$ is an associative conformal
coalgebra and they satisfy the following conditions
\begin{eqnarray}
\label{1-thq}&&\Delta(a_\lambda b)=(I\otimes {\mathfrak{L}_P(a)}_\lambda)\Delta(b)+({\mathfrak{R}_P(b)}_{-\lambda-\partial^{\otimes^2}}\otimes I)\Delta(a),\\
\label{2-thq}&&({\mathfrak{L}_P(b)}_{-\lambda-\partial^{\otimes^2}}\otimes I-I\otimes {\mathfrak{R}_P(b)}_{-\lambda-\partial^{\otimes^2}})\Delta(a)+\tau({\mathfrak{L}_P(a)}_\lambda\otimes I-I\otimes {\mathfrak{R}_P(a)}_\lambda)\Delta(b)=0,\;a, b\in P.
\end{eqnarray}

An ASI conformal bialgebra $(P, \cdot_\lambda\cdot, \Delta)$ is called {\bf commutative and cocommutative}, if $(P, \cdot_\lambda\cdot)$ is commutative and $(P, \Delta)$ is cocommutative.

 \delete{A commutative and cocommutative ASI conformal bialgebra $(P, \cdot_\lambda\cdot, \Delta)$ is called {\bf coboundary} if  there exists $r\in R\otimes R$ such that
\begin{eqnarray}\label{eQ1}
\Delta(a)=(I\otimes L_P(a)_\lambda-L_P(a)_{\lambda}\otimes
I)r|_{\lambda=-\partial^{\otimes^2}}, \;\;\;a\in P.
\end{eqnarray}}
}
 \end{defi}
\begin{rmk}
Note that for a commutative and cocommutative ASI conformal
bialgebra $(P$, $\cdot_\lambda\cdot$, $\Delta)$, Eq. (\ref{2-thq})
naturally holds and Eq. (\ref{1-thq}) becomes
 \begin{eqnarray}
\label{thq1}\Delta(a_\lambda b)=(I\otimes {\mathfrak{L}_P(a)}_\lambda)\Delta(b)+({\mathfrak{L}_P(b)}_{-\lambda-\partial^{\otimes^2}}\otimes I)\Delta(a),\;\;a, b\in P.
\end{eqnarray}
\end{rmk}

{\rm
\begin{defi}\cite{L}
A {\bf Lie conformal coalgebra} $(P, \delta)$ is a ${\bf
k}[\partial]$-module $P$ endowed with a ${\bf k}[\partial]$-module
homomorphism $\delta: P\rightarrow P\otimes P$ such that the
following conditions hold.
\begin{eqnarray*}
&&\delta(a)=-\tau \delta(a), \\
&&(I\otimes \delta)\delta(a)-(\tau\otimes I)(I\otimes \delta)\delta(a)=(\delta\otimes I)\delta(a),\;\;a\in P.
\end{eqnarray*}

A {\bf Lie conformal bialgebra} is a triple $(P, [\cdot_\lambda
\cdot],\delta)$ such that $(P, [\cdot_\lambda \cdot])$ is a Lie
conformal algebra, $(P, \delta)$ is a Lie conformal coalgebra and
they satisfy the following condition {\small \begin{eqnarray*}
\delta([a_\lambda b])=({\mathfrak{ad}_P(a)}_\lambda\otimes I+I\otimes
{\mathfrak{ad}_P(a)}_\lambda)\delta(b)-({\mathfrak{ad}_P(b)}_{-\lambda-\partial^{\otimes^2}}\otimes
I+I\otimes  {\mathfrak{ad}_P(b)}_{-\lambda-\partial^{\otimes^2}})\delta(a),
\;a, b\in P.
\end{eqnarray*}}
\end{defi}}

Next, we introduce the notions of Poisson conformal coalgebras and
Poisson conformal bialgebras.
\begin{defi}
{\rm A {\bf Poisson conformal coalgebra} $(P, \delta, \Delta)$ is
a ${\bf k}[\partial]$-module $P$ endowed with two ${\bf
k}[\partial]$-module homomorphisms $\delta$, $\Delta: P\rightarrow
P\otimes P$ such that $(P,\delta)$ is a Lie conformal coalgebra,
$(P, \Delta)$ is a cocommutative associative conformal coalgebra
and they satisfy the following compatibility condition
\begin{eqnarray}\label{coalgebra-1}
(I\otimes \Delta)\delta(a)=(\delta\otimes I)\Delta(a)+(\tau\otimes I)(I\otimes \delta)\Delta (a),\;\;\; a\in P.
\end{eqnarray}
}
\end{defi}

\begin{pro}
Let $(P, \delta, \Delta)$ be a Poisson conformal coalgebra. Then
$P^{\ast c}$ is a Poisson conformal algebra with the following
$\lambda$-bracket and $\lambda$-product
\begin{eqnarray*}
([f_\lambda g])_\mu (a)=\sum f_\lambda(a_{[1]})g_{\mu-\lambda}(a_{[2]})=(f\otimes g)_{\lambda,\mu-\lambda}(\delta (a)),\\
(f_\lambda g)_\mu (a)=\sum f_\lambda(a_{(1)})g_{\mu-\lambda}(a_{(2)})=(f\otimes g)_{\lambda,\mu-\lambda}(\Delta (a)),
\end{eqnarray*}
where $f, g\in P^{\ast c}$, $a\in P$, $\delta(a)=\sum a_{[1]}\otimes a_{[2]}$ and $\Delta(a)=\sum a_{(1)}\otimes a_{(2)}$.
\end{pro}
\begin{proof}
It follows from \cite[Proposition 2.13]{L} and \cite[Proposition
4.2]{HB1} with a similar proof as that in \cite[Proposition
2.13]{L}.
\end{proof}
\begin{pro}
Let $(P, [\cdot_\lambda \cdot], \cdot_\lambda \cdot)$ be a Poisson
conformal algebra in which $P$ is free as a ${\bf
k}[\partial]$-module, that is $P=\sum_{i=1}^n{\bf
k}[\partial]e_i$. Then $P^{\ast c}=\text{Chom}(A,{\bf
k})=\sum_{i=1}^n{\bf k}[\partial]e_i^\ast$ is a Poisson conformal
coalgebra with the following co-products
\begin{eqnarray}
&&\delta(f)=\sum_{i,j}f_\mu([{e_i}_\lambda e_j])(e_i^\ast\otimes e_j^\ast)|_{\lambda=\partial\otimes I,\mu=-\partial^{\otimes^2}},\;\;\;\Delta(f)=\sum_{i,j}f_\mu({e_i}_\lambda e_j)(e_i^\ast\otimes e_j^\ast)|_{\lambda=\partial\otimes I,\mu=-\partial^{\otimes^2}},
\end{eqnarray}  where $\{e_1^\ast, \cdots, e_n^\ast\}$ is the
 dual ${\bf k}[\partial]$-basis of $P^{\ast c}$.  More precisely, if $[{e_i}_\lambda e_j]=\sum_kR_k^{ij}(\lambda,\partial)e^k$ and ${e_i}_\lambda e_j=\sum_kT_k^{ij}(\lambda,\partial)e^k$, then
\begin{eqnarray*}
\delta(e_k^\ast)=\sum_{i,j}S_k^{ij}(\partial\otimes I, I\otimes\partial)(e_i^\ast\otimes e_j^\ast),\;\;\Delta(e_k^\ast)=\sum_{i,j}Q_k^{ij}(\partial\otimes I, I\otimes\partial)(e_i^\ast\otimes e_j^\ast),
\end{eqnarray*}
where $S_k^{ij}(x,y)=R_k^{ij}(x,-x-y)$ and $Q_k^{ij}(x,y)=T_k^{ij}(x,-x-y)$.
\end{pro}
\begin{proof}
It follows from \cite[Proposition 2.13]{L} and \cite[Proposition
4.3]{HB1} with a similar proof as that in \cite[Proposition
2.13]{L}.
\end{proof}

%Next, we introduce the definition of Poisson conformal bialgebras.

\begin{defi}
{\rm A {\bf Poisson conformal bialgebra} is  a 5-tuple $(P,
[\cdot_\lambda \cdot], \cdot_\lambda \cdot, \delta, \Delta)$ such
that $(P, [\cdot_\lambda \cdot]$, $\cdot_\lambda \cdot)$ is a
Poisson conformal algebra, $(P, \delta, \Delta)$ is a Poisson
conformal coalgebra, $(P, [\cdot_\lambda \cdot], \delta)$ is a Lie
conformal bialgebra, $(P,\cdot_\lambda \cdot, \Delta)$ is a
commutative and cocommutative ASI conformal bialgebra, and they
satisfy the following compatibility conditions
\begin{eqnarray}
&&\label{Poisson-bialgebra-1}\delta(a_\lambda b)=({\mathfrak{L}_P(a)}_\lambda \otimes I)\delta (b)+({\mathfrak{L}_P(b)}_{-\lambda-\partial^{\otimes^2}}\otimes I)\delta(a)
+(I\otimes {\mathfrak{ad}_P(a)}_\lambda)\Delta(b)\\
&&\qquad\qquad+(I\otimes {\mathfrak{ad}_P(b)}_{-\lambda-\partial^{\otimes^2}})\Delta(a),\nonumber\\
&&\label{Poisson-bialgebra-2}\Delta([a_\lambda b])=({\mathfrak{ad}_P(a)}_\lambda \otimes I+I\otimes {\mathfrak{ad}_P(a)}_\lambda)\Delta(b)\\
&&\qquad\qquad-(I\otimes {\mathfrak{L}_P(b)}_{-\lambda-\partial^{\otimes^2}})\delta(a)
+({\mathfrak{L}_P(b)}_{-\lambda-\partial^{\otimes^2}} \otimes I)\delta(a),\;\;a, b\in P.\nonumber
\end{eqnarray}
}
\end{defi}

\begin{thm}\label{Poisson-thm2}
Let $(P, [\cdot_\lambda \cdot], \cdot_\lambda \cdot)$ be a Poisson conformal algebra and $(P, \delta, \Delta)$ be a Poisson conformal coalgebra, where $P$ is free as a ${\bf k}[\partial]$-module. Denote the Poisson conformal algebra structure on $P^{\ast c}$ obtained from $\delta$ and $\Delta$ by $(P^{\ast c}, [\cdot_\lambda \cdot], \cdot_\lambda \cdot)$. Then the following conditions are equivalent.
\begin{enumerate}
\item \label{1}$(P, [\cdot_\lambda \cdot], \cdot_\lambda \cdot,
\delta, \Delta)$ is a Poisson conformal bialgebra. \item \label{2}
$(P, P^{\ast c}, \mathfrak{ad}_P^{\ast}, -\mathfrak{L}_P^\ast, \mathfrak{ad}_{P^{\ast c}}^\ast,
-\mathfrak{L}_{P^{\ast c}}^\ast)$ is a matched pair of Poisson conformal
algebras. \item \label{3} $(P\oplus P^{\ast c}, P, P^{\ast c})$ is
a Manin triple of Poisson conformal algebras with the invariant
symmetric conformal bilinear form $\langle \cdot, \cdot\rangle_\lambda$ on
$P\oplus P^{\ast c}$ defined by Eq.~{\rm (\ref{poisson-inv})}.
\end{enumerate}
\end{thm}
\begin{proof}
By Theorem \ref{poisson-thm1}, Item (\ref{2}) holds if and only if
Item (\ref{3}) holds. Therefore we only need to prove that Item
(\ref{1}) holds if and only if  Item (\ref{2}) holds. By
\cite[Theorem 3.7]{HB1}, $(P,\cdot_\lambda \cdot,\Delta)$ is a
commutative and cocommutative ASI conformal bialgebra if and only
if $(P, P^{\ast c},-\mathfrak{L}_P^\ast, -\mathfrak{L}_{P^{\ast c}}^\ast)$ is a matched
pair of commutative associative conformal algebras. With a similar
proof as that in \cite[Theorem 3.7]{HB1}, we show that $(P,
[\cdot_\lambda \cdot], \delta)$ is a Lie conformal bialgebra if
and only if $(P, P^{\ast c},\mathfrak{ad}_P^\ast, \mathfrak{ad}_{P^{\ast
c}}^\ast)$ is a matched pair of Lie conformal algebras. Then with
a similar proof as that in \cite[Theorem 3.7]{HB1} again, we show
that when $\rho_P=\mathfrak{ad}_P^\ast$, $l_P=-\mathfrak{L}_P^\ast$,
$\rho_Q=\mathfrak{ad}_{P^{\ast c}}^\ast$ and $l_Q=-\mathfrak{L}_{P^{\ast
c}}^\ast$,
\begin{eqnarray*}
\text{Eq.}~(\ref{poisson-matched-1})\Longleftrightarrow \text{Eq.}
~(\ref{Poisson-bialgebra-2})\Longleftrightarrow \text{Eq.}
~(\ref{poisson-matched-4}), \;\;\text{Eq.}~
(\ref{poisson-matched-2})\Longleftrightarrow
\text{Eq.}~(\ref{Poisson-bialgebra-1})\Longleftrightarrow
\text{Eq.}~(\ref{poisson-matched-3}).
\end{eqnarray*}
Therefore the conclusion holds.
\end{proof}

%Next, we study the coboundary case of Poisson conformal bialgebras.

\subsection{Coboundary Poisson conformal bialgebras}
\begin{defi}
A Poisson conformal bialgebra $(P, [\cdot_\lambda \cdot],
\cdot_\lambda \cdot, \delta, \Delta)$ is called {\bf coboundary},
if both the Lie conformal bialgebra $(P, [\cdot_\lambda \cdot],
\delta)$ and the commutative and cocommutative ASI conformal
bialgebra $(P, \cdot_\lambda \cdot, \Delta)$ are coboundary, that
is, there exists an $r\in P\otimes P$ such that
\begin{eqnarray}
\label{poisson-coboundary-11}\delta(a)=({\mathfrak{ad}_P(a)}_\lambda \otimes I+I\otimes {\mathfrak{ad}_P(a)}_\lambda) r|_{\lambda=-\partial^{\otimes^2}},\\
\label{poisson-coboundary-2}\Delta(a)=(I\otimes
{\mathfrak{L}_P(a)}_\lambda-{\mathfrak{L}_P(a)}_\lambda\otimes
I)r|_{\lambda=-\partial^{\otimes^2}},
\end{eqnarray}
 for all $a\in P$.
\end{defi}

Set $r=\sum_i r_i\otimes l_i\in P\otimes P$. Set
\begin{eqnarray*}
&&[[r,r]]:=\sum_{i,j}([{r_i}_\mu r_j]\otimes l_i \otimes l_j|_{\mu=I\otimes \partial \otimes I}-r_i\otimes [{r_j}_\mu l_i]\otimes l_j|_{\mu=I\otimes I\otimes \partial}-r_i\otimes r_j\otimes [{l_j}_\mu l_i]|_{\mu=I\otimes \partial\otimes I}),\\
&&r\bullet r:=\sum_{i,j} (r_i\otimes r_j\otimes {l_i}_\mu l_j|_{\mu=\partial\otimes I\otimes I}-r_i\otimes {r_j}_\mu l_i\otimes l_j|_{\mu=-\partial^{\otimes^2}\otimes I}+{r_i}_\mu r_j\otimes l_i\otimes l_j|_{\mu=I\otimes \partial\otimes I}).
\end{eqnarray*}

\begin{thm}\label{Poisson-thm3}
Let $(P, [\cdot_\lambda \cdot], \cdot_\lambda \cdot)$ be a Poisson
conformal algebra and $r=\sum_i r_i\otimes l_i\in P\otimes P$.
Define $\delta$ and $\Delta$ by Eqs. {\rm
(\ref{poisson-coboundary-11})} and {\rm
(\ref{poisson-coboundary-2})} respectively. Then $(P,
[\cdot_\lambda \cdot], \cdot_\lambda \cdot, \delta, \Delta)$ is a
Poisson conformal bialgebra if and only if the following
conditions are satisfied
\begin{enumerate}
        \item \label{cond1}$({\mathfrak{ad}_P(a)}_\lambda \otimes I+I\otimes {\mathfrak{ad}_P(a)}_\lambda )(r+\tau r)|_{\lambda=-\partial^{\otimes^2}}=0$;
        \item \label{cond2}$(I\otimes \mathfrak{L}_P(a)_{-\partial^{\otimes^2}}-\mathfrak{L}_P(a)_{-\partial^{\otimes^2}}\otimes I)(r+\tau r)=0$;
        \item \label{cond3}$({\mathfrak{ad}_P(a)}_\lambda\otimes I\otimes I+I\otimes {\mathfrak{ad}_P(a)}_\lambda\otimes I+I\otimes I\otimes {\mathfrak{ad}_P(a)}_\lambda)[[r,r]]|_{\lambda=-\partial^{\otimes^3}}=0$;
        \item \label{cond4}$(I\otimes I\otimes \mathfrak{L}_P(a)_{-\partial^{\otimes^3}}-\mathfrak{L}_P(a)_{-\partial^{\otimes^3}}\otimes I\otimes I) (r\bullet r)=0$;
        \item \label{cond5}$({\mathfrak{ad}_P(a)}_{-\partial^{\otimes^3}} \otimes I\otimes I)(r\bullet r)-(I\otimes {\mathfrak{L}_P(a)}_{-\partial^{\otimes^3}} \otimes I-I\otimes I\otimes {\mathfrak{L}_P(a)}_{-\partial^{\otimes^3}} )[[r,r]]
+\sum_i(({\mathfrak{ad}_P(r_i)}_{I\otimes I\otimes \partial} \otimes
I)({\mathfrak{L}_P(a)}_{-\partial^{\otimes^2}\otimes I} \otimes I-I\otimes
{\mathfrak{L}_P(a)}_{-\partial^{\otimes^2}\otimes I})(r+\tau r)\otimes
l_i)=0$,
      \end{enumerate}
for all $a\in P$.
\end{thm}
\begin{proof}
By \cite[Theorem 3.4]{L}, $\delta$ defined by
Eq.~(\ref{poisson-coboundary-11}) makes $(P, [\cdot_\lambda
\cdot], \delta)$ into a Lie conformal bialgebra if and only if
Items (\ref{cond1}) and (\ref{cond3}) hold. Moreover, by
\cite[Theorem 5.3]{HB1}, $(P, \cdot_\lambda \cdot, \Delta)$ is a
commutative and cocommutative ASI conformal bialgebra if and only
if Items (\ref{cond2}) and (\ref{cond4}) hold. By Eqs.
(\ref{poisson-coboundary-11}) and (\ref{poisson-coboundary-2}), it
is straightforward to show that Eqs. (\ref{Poisson-bialgebra-1})
and (\ref{Poisson-bialgebra-2}) naturally hold. Then we only need
to prove that Eq. (\ref{coalgebra-1}) holds if and only if Item
(\ref{cond5}) holds. Let $a\in P$. By Eqs.
(\ref{poisson-coboundary-11}) and (\ref{poisson-coboundary-2}), we
obtain
\begin{eqnarray*}
0&=&(I\otimes \Delta)\delta(a)-(\delta \otimes I)\Delta (a)-(\tau\otimes I)(I\otimes \delta)\Delta(a)\\
&=&\sum_{i,j}([a_{-\partial^{\otimes^3}}r_i]\otimes r_j\otimes {l_i}_{-I\otimes \partial^{\otimes^2}} l_j
-[a_{-\partial^{\otimes^3}}r_i]\otimes {l_i}_{-I\otimes \partial^{\otimes^2}}r_j\otimes l_j\\
&&+[(a_\lambda r_i)_{-\partial^{\otimes^2}\otimes I}r_j]\otimes l_j\otimes l_i|_{\lambda=-\partial^{\otimes^3}}
-[(a_\lambda l_i)_{-\partial\otimes I\otimes I-I\otimes I\otimes \partial} r_j]\otimes r_i\otimes l_j|_{\lambda=-\partial^{\otimes^3}}\\
&&-r_i\otimes [a_\lambda l_i]_{-\partial^{\otimes^2}\otimes I}r_j\otimes l_j|_{\lambda=-\partial^{\otimes^3}}+r_j\otimes
[(a_\lambda r_i)_{-\partial^{\otimes^2}\otimes I}l_j]\otimes l_j\\
&&+[{l_j}_{-\partial\otimes I\otimes I-I\otimes I\otimes \partial}r_j]\otimes a_{-\partial^{\otimes^3}}r_i\otimes l_j+r_j\otimes a_{-\partial^{\otimes^3}}r_i\otimes [{l_i}_{-\partial\otimes I\otimes I-I\otimes I\otimes \partial} l_j]\\
&&+r_i\otimes r_j\otimes [a_\lambda l_i]_{-I\otimes \partial^{\otimes^2}}l_j|_{\lambda=-\partial^{\otimes^3}}-[{r_i}_{-\partial^{\otimes^2}\otimes I} r_j]\otimes l_j\otimes a_{-\partial^{\otimes^3}} l_i\\
&&-r_j\otimes [{r_i}_{-\partial^{\otimes^2}\otimes I} l_j]\otimes a_{-\partial^{\otimes^3}}l_i
-r_j\otimes r_i\otimes [(a_\lambda l_i)_{-\partial\otimes I\otimes I-I\otimes I\otimes \partial} l_j]|_{\lambda=-\partial^{\otimes^3}}).
\end{eqnarray*}
It is straightforward to show that the sum of the first four terms
is $$({\mathfrak{ad}_P(a)}_{-\partial^{\otimes^3}} \otimes I\otimes
I)(r\bullet r)+\sum_i((({\mathfrak{ad}_P(r_i)}_{I\otimes I\otimes
\partial}\otimes I)(\mathfrak{L}_P(a)_{-\partial^{\otimes^2}\otimes I}\otimes
I)(r+\tau r))\otimes l_i),$$ the sum of next four terms is
$$-(I\otimes {\mathfrak{L}_P(a)}_{-\partial^{\otimes^3}}\otimes
I)[[r,r]]-\sum_i((({\mathfrak{ad}_P(r_i)}_{I\otimes I\otimes \partial}
\otimes I)(I\otimes \mathfrak{L}_P(a)_{-\partial^{\otimes^2}\otimes I} )(r+\tau
r))\otimes l_i),$$ and the sum of the last four terms is
$(I\otimes I\otimes {\mathfrak{L}_P(a)}_{-\partial^{\otimes^3}})[[r,r]]$.
Then the proof is completed.
\end{proof}

\begin{cor}\label{cob-Poisson}
Let $(P, [\cdot_\lambda \cdot], \cdot_\lambda \cdot)$ be a Poisson
conformal algebra and $r\in P\otimes P$. Suppose that $r$ is
skew-symmetric. Then the maps given by Eqs. {\rm
(\ref{poisson-coboundary-11})} and {\rm
(\ref{poisson-coboundary-2})} define a Poisson conformal coalgebra
$(P, \delta, \Delta)$ such that $(P, [\cdot_\lambda \cdot],
\cdot_\lambda \cdot, \delta, \Delta)$ is a Poisson conformal
bialgebra if
\begin{eqnarray}
\label{Yang-Baxter-1}[[r,r]]\equiv 0~~~~\text{mod}~~(\partial^{\otimes^3}),\\
\label{Yang-Baxter-2}r\bullet r \equiv 0~~~~\text{mod}~~(\partial^{\otimes^3}).
\end{eqnarray}
\end{cor}
\begin{proof}
It follows from Theorem \ref{Poisson-thm3} immediately.
\end{proof}
\begin{defi}
\delete{(\ref{Yang-Baxter-1}) is called {\bf classical conformal
Yang-Baxter equation} (see \cite{L}).} Let $(P, [\cdot_\lambda
\cdot], \cdot_\lambda \cdot)$ be a Poisson conformal algebra and
$r\in P\otimes P$. If $r$ satisfies both Eqs.
(\ref{Yang-Baxter-1}) and (\ref{Yang-Baxter-2}), then $r$ is
called {\bf a solution of the Poisson conformal Yang-Baxter
equation (PCYBE) in $P$}.
\end{defi}

\begin{thm}
Let $(P,[\cdot_\lambda \cdot], \cdot_\lambda \cdot, \delta,
\Delta)$ be a Poisson conformal bialgebra in which $P$ is free as
a ${\bf k}[\partial]$-module. Then there is a canonical Poisson
conformal bialgebra structure on the ${\bf k}[\partial]$-module
$P\oplus P^{\ast c}$ as the direct sum of ${\bf
k}[\partial]$-modules.
\end{thm}
\begin{proof}
Let $\{e_1,\cdots, e_n\}$ be a ${\bf k}[\partial]$-basis of $P$
and $\{ e_1^\ast, \cdots, e_n^\ast\}$ be the dual ${\bf
k}[\partial]$-basis in $P^{\ast c}$. By Theorem
\ref{poisson-thm1}, $(P, P^{\ast c}, \mathfrak{ad}_P^{\ast},
-\mathfrak{L}_P^\ast, \mathfrak{ad}_{P^{\ast c}}^\ast, -\mathfrak{L}_{P^{\ast c}}^\ast)$ is
a matched pair of Poisson conformal algebras where the Lie
conformal algebra structure and associative conformal algebra
structure on $P^{\ast c}$ are obtained from $\delta$ and $\Delta$
respectively.  Then by Proposition \ref{pro:match pair}, there is
a Poisson conformal algebra structure on $P\oplus P^{\ast c}$
associated with this matched pair. Set $r=\sum_{i=1}^n e_i \otimes
e_i^\ast \in (P\oplus P^{\ast c})\otimes (P\oplus P^{\ast c})$.
Then by \cite[Theorem 5.8]{HB1} and \cite[Theorem 3.10]{L}, $r$ is
a solution of the PCYBE in $P\oplus P^{\ast c}$ and Items
(\ref{cond1}) and (\ref{cond2}) in Theorem \ref{Poisson-thm3}
hold. Hence Items (\ref{cond1})-(\ref{cond4}) in Theorem
\ref{Poisson-thm3} hold. Since $r$ is a solution of the PCYBE in
$P\oplus P^{\ast c}$, Item (\ref{cond5}) in Theorem
\ref{Poisson-thm3} is reduced to
\begin{eqnarray}
\label{condd}\sum_i(({\mathfrak{ad}_P(r_i)}_{I\otimes I\otimes \partial} \otimes
I)({\mathfrak{L}_P(a)}_{-\partial^{\otimes^2}\otimes I} \otimes I-I\otimes
{\mathfrak{L}_P(a)}_{-\partial^{\otimes^2}\otimes I})(r+\tau r)\otimes
l_i)=0.
\end{eqnarray} By Item (\ref{cond2}) in Theorem \ref{Poisson-thm3}, Eq. (\ref{condd}) holds. Therefore by Theorem \ref{Poisson-thm3}, there is a
coboundary Poisson conformal bialgebra structure on $P\oplus
P^{\ast c}$ obtained from $r$.
\end{proof}

%\subsection{$\mathcal{O}$-operators on Poisson conformal algebras and pre-Poisson conformal algebras}
\delete{In this subsection, some results about the relations between skew-symmetric solutions of Poisson conformal Yang-Baxter equation and $\mathcal{O}$-operators on Poisson conformal algebras are presented.}
\begin{defi}
Let $(V, \rho_P, l_P)$  be a representation of a Poisson conformal
algebra $(P, [\cdot_\lambda \cdot], \cdot_\lambda \cdot)$. If a
${\bf k}[\partial]$-module homomorphism $T: V\rightarrow P$ is
both an $\mathcal{O}$-operator on the commutative associative
conformal algebra $(P,\cdot_\lambda \cdot)$ associated with the
representation $(V, l_P)$ and an $\mathcal{O}$-operator on the Lie
conformal algebra $(P, [\cdot_\lambda \cdot])$ associated with the
representation $(V, \rho_P)$, that is, $T$ satisfies the following
conditions respectively
\begin{eqnarray*}
&&{T(u)}_\lambda T(v)=T({l_P(T(u))}_\lambda v)+T({l_P(T(v))}_{-\lambda-\partial}u),\\
&&[{T(u)}_\lambda T(v)]=T({\rho_P(T(u))}_\lambda v)+T({\rho_P(T(v))}_{-\lambda-\partial}u),\;\;u, v\in V,
\end{eqnarray*}
then $T$ is called an {\bf $\mathcal{O}$-operator on the Poisson
conformal algebra} $(P, [\cdot_\lambda \cdot], \cdot_\lambda
\cdot)$ associated with the representation $(V, \rho_P, l_P)$.
\end{defi}
\delete{\begin{rmk} If $T: P\rightarrow P$ is an
$\mathcal{O}$-operator on the Poisson conformal algebra $(P,
[\cdot_\lambda \cdot], \cdot_\lambda \cdot)$  associated with the
module $(P, \ad_P, L_P)$, then $T$ is called a {\bf Rota-Baxter
operator} on the Poisson conformal algebra $(P, [\cdot_\lambda
\cdot], \cdot_\lambda \cdot)$.
\end{rmk}}

Let $(P, [\cdot_\lambda \cdot], \cdot_\lambda \cdot)$ be a Poisson
conformal algebra in which $P$ is free as a ${\bf
k}[\partial]$-module. Define a linear map $\varphi: P\otimes
P\rightarrow \text{Chom}(P^{\ast c}, P)$ as
\begin{eqnarray}
\varphi(a\otimes
b)_\lambda(f)=f_{-\lambda-\partial^P}(a)b,\;\;\ a, b\in
P, f\in P^{\ast c}.
\end{eqnarray}
Here $\partial^P$ represents the action of $\partial$ on $P$.
Obviously,  $\varphi$ is a ${\bf k}[\partial]$-module
homomorphism. Similar to \cite[Proposition 6.1]{BKL}, we show that
$\varphi$ is a ${\bf k}[\partial]$-module isomorphism.

%\yy{Here, I add the condition that $A$ is a finite associative conformal algebra which is free as a $\mathbb{C}[\partial]$-module. Then we obtain that $\varphi$ is a $\mathbb{C}[\partial]$-module isomorphism with a similar proof as that in Proposition 6.1 in \cite{BKL}.}

Set $r=\sum_i r_i\otimes l_i\in P\otimes P$. By $\varphi$, we
associate a conformal linear map $T^r\in \text{Chom}(P^{\ast c},P)$ given
by
\begin{eqnarray*}
T^r_\lambda(f)=\sum_if_{-\lambda-\partial^P}(r_i)l_i,\;\;
f\in P^{\ast c}.
\end{eqnarray*}

\begin{pro}
Let $(P, [\cdot_\lambda \cdot], \cdot_\lambda \cdot)$ be a Poisson
conformal algebra in which $P$ is free as a ${\bf
k}[\partial]$-module and $r\in P\otimes P$ be skew-symmetric. Then
$r$ is a solution of the PCYBE in $P$ if and only if
$T^r_0=T^r_\lambda|_{\lambda=0}$ is an $\mathcal{O}$-operator on
the Poisson conformal algebra $(P, [\cdot_\lambda \cdot],
\cdot_\lambda \cdot)$  associated with the representation
$(P^{\ast c}, \mathfrak{ad}_P^\ast, -\mathfrak{L}_P^\ast)$.
\end{pro}
\begin{proof}
It follows directly from \cite[Theorem 6.1]{HB1} and \cite[Theorem 3.1]{HB}.
\end{proof}

\delete{Recall that a conformal bilinear form $\langle \cdot,\cdot\rangle_\lambda$ on a Lie conformal algebra $R$ is call {\bf invariant}, if
for any $a$, $b$, $c\in R$,
\begin{eqnarray*}
\langle[a_\mu b], c\rangle_\lambda =\langle a, [b_{\lambda-\partial} c]\rangle_\mu=-\langle a, [c_{-\lambda} b]\rangle_\mu.
\end{eqnarray*}
This definition can be referred to \cite{L}.

Let $(P, [\cdot_\lambda \cdot], \cdot_\lambda \cdot)$ be a finite Poisson conformal algebra with a non-degenerate conformal bilinear form $\langle \cdot,\cdot\rangle_\lambda$, and $r=\sum_i r_i\otimes l_i\in P\otimes P$. Similar to that in Section 6, we can define a conformal linear map $T^r\in Cend(P)$ as $\langle r, u\otimes v\rangle_{(\lambda,\mu)}=\langle T^r_{\lambda-\partial}(u),v\rangle_\mu$. Then by Theorem \ref{th2} and Corollary 3.3 in \cite{HB}, we can obtain the following proposition.

\begin{pro}
Let $(P, [\cdot_\lambda \cdot], \cdot_\lambda \cdot)$ be a finite Poisson conformal algebra with a non-degenerate symmetric conformal bilinear form $\langle \cdot,\cdot\rangle_\lambda$ such that this conformal bilinear form is invariant both on the Lie conformal algebra $(P, [\cdot_\lambda \cdot])$ and on the associative conformal algebra $(P, \cdot_\lambda \cdot)$. Assume that $r\in P\otimes P$ is antisymmetric. Then $r$ is a solution of Poisson conformal Yang-Baxter equation if and only if $T^r_0=T^r_\lambda|_{\lambda=0}$ is a Rota-Baxter operator of $(P, [\cdot_\lambda \cdot], \cdot_\lambda \cdot)$.
\end{pro}}

Suppose that $(V, \rho_P, l_P)$ is a representation of $(P,
[\cdot_\lambda \cdot], \cdot_\lambda \cdot)$, in which $V$ and $P$ are free
as ${\bf k}[\partial]$-modules. Note that $(V^{\ast c},
\rho^\ast_P, -l_P^\ast)$ is also a representation of $(P,
[\cdot_\lambda \cdot], \cdot_\lambda \cdot)$. By \cite[Proposition
6.1]{BKL}, $V^{\ast c}\otimes P\cong \text{Chom}(V,P)$ as ${\bf
k}[\partial]$-modules through the isomorphism $\varphi$ defined as
$$\varphi(f\otimes a)_\lambda
v=f_{\lambda+\partial^P}(v)a,\;\;a\in P, v\in V, f\in V^{\ast
c},$$ where $\partial^p$ represents the action of $\partial$ on
$P$. By the ${\bf k}[\partial]$-module actions on $V^{\ast
c}\otimes P$, we also get $V^{\ast c}\otimes P\cong P\otimes
V^{\ast c}$ as ${\bf k}[\partial]$-modules. Therefore as ${\bf
k}[\partial]$-modules, $\text{Chom}(V,P)\cong P\otimes V^{\ast
c}$. Consequently, for any $T\in \text{Chom}(V,P)$, we associate
an $r_T\in P\otimes V^{\ast c}\subset (P\ltimes_{\rho_P^\ast,
-l_P^\ast} V^{\ast c})\otimes (P\ltimes_{\rho_P^\ast, -l_P^\ast}
V^{\ast c})$.

\begin{thm}\label{Poisson-Yang-Baxter-1}
Let $(P, [\cdot_\lambda \cdot], \cdot_\lambda \cdot)$ be a Poisson
conformal algebra and $(V, \rho_P, l_P)$ be a representation.
Assume that $P$ and $V$ are free as ${\bf k}[\partial]$-modules.
Let $T\in \text{Chom}(V,P)$ and $r_T\in P\otimes V^{\ast c}\in
(P\ltimes_{\rho^\ast_P, -l_P^\ast}V )\otimes
(P\ltimes_{\rho^\ast_P, -l_P^\ast}V) $ be the corresponding
element of $T$. Then $r=r_T-\tau r_T$ is a skew-symmetric solution
of the PCYBE in $P\ltimes_{\rho^\ast_P, -l_P^\ast}V$ if and only
if $T_0=T_\lambda|_{\lambda=0}$ is an $\mathcal{O}$-operator on
$(P, [\cdot_\lambda \cdot], \cdot_\lambda \cdot)$ associated with
$(V, \rho_P, l_P)$.
\end{thm}
\begin{proof}
It follows directly from \cite[Theorem 6.7]{HB1} and \cite[Theorem 3.7]{HB}.
\end{proof}

\delete{Finally, we can present a similar result as Theorem \ref{thms1}.
\begin{pro}
Let $(P, [\cdot_\lambda \cdot], \cdot_\lambda \cdot)$ be a finite Poisson conformal algebra  which is free as a $\mathbb{C}[\partial]$-module and $r\in P\otimes P$ is non-degenerate. Then $r$ is a non-degenerate antisymmetric solution of Poisson conformal Yang-Baxter
equation if and only if the bilinear map given by
\begin{eqnarray*}
\alpha_\lambda(a,b)=(T_0^r)^{-1}(a)_\lambda (b),~~~~~~~~~a,~~b\in P,
\end{eqnarray*}
is a conformal Connes cocycle and satisfies
\begin{eqnarray*}
\alpha_\lambda(a,[b_\mu c])-\alpha_\mu(b, [a_\lambda c])=\alpha_{\lambda+\mu}([a_\lambda b], c), ~~~~~~a,~~b,~~c\in P,
\end{eqnarray*}
where $T^r\in \text{Chom}(P^{\ast c},P)$ is the element corresponding to $r$ through the isomorphism $P\otimes P\cong \text{Chom}(P^{\ast c}, P)$.
\end{pro}}

Recall \cite{HL1} that a {\bf left-symmetric conformal algebra}
$(P, \circ_\lambda)$ is a ${\bf k}[\partial]$-module $P$ together with
a {\bf k}-bilinear  map $\circ_\lambda: P\times P\rightarrow
P[\lambda]$ such that the following condition holds.
\begin{eqnarray}
&&(\partial a)\circ_\lambda b=-\lambda a\circ_\lambda b,\;\; a\circ_\lambda (\partial b)=(\partial+\lambda)a\circ_\lambda b,\\
&&a\circ_\lambda (b\circ_\mu c)-(a\circ_\lambda b)\circ_{\lambda+\mu}c=b\circ_\mu (a\circ_\lambda c)-(b\circ_\mu a)\circ_{\lambda+\mu}c,\;\;a, b, c\in P.
\end{eqnarray}

\begin{defi}
A {\bf Zinbiel conformal algebra} $(P, \succ_\lambda)$ is a ${\bf
k}[\partial]$-module $P$ together with a {\bf k}-bilinear map
$\succ_\lambda: P\times P\rightarrow P[\lambda]$ such that the
following condition holds.
\begin{eqnarray}
&&(\partial a)\succ_\lambda b=-\lambda a\succ_\lambda b,\;\; a\succ_\lambda (\partial b)=(\partial+\lambda)a\succ_\lambda b,\\
&&a\succ_\lambda(b\succ_\mu c)=(a\succ_\lambda b+b\succ_{-\lambda-\partial}a)\succ_{\lambda+\mu}c,\;\;a, b, c\in P.
\end{eqnarray}

A {\bf pre-Poisson conformal algebra} is a triple $(P, \circ_\lambda,\succ_\lambda)$ such that $(P, \succ_\lambda)$ is a Zinbiel conformal algebra, $(P, \circ_\lambda)$ is a left-symmetric conformal algebra and the following conditions hold.
\begin{eqnarray}
\label{pre-P1}(a\circ_\lambda b-b\circ_{-\lambda-\partial}a)\succ_{\lambda+\mu}c=a\circ_\lambda(b\succ_\mu c)-b\succ_\mu(a\circ_\lambda c),\\
\label{pre-P2}(a\succ_\lambda b+b\succ_{-\lambda-\partial}a)\circ_{\lambda+\mu}c=a\succ_\lambda(b\circ_\mu c)+b\succ_\mu(a\circ_\lambda c),\;\;a, b, c\in P.
\end{eqnarray}
\end{defi}
\delete{\begin{rmk}
This is a conformal analog of pre-Poisson algebra given in \cite{A3}.
\end{rmk}}

Then the following conclusion follows immediately.
\begin{pro}\label{pre-Poisson-1}
Let $(P, \circ_\lambda, \succ_\lambda)$ be a pre-Poisson conformal algebra. Define
\begin{eqnarray}\label{eq:48}
[a_\lambda b]:=a\circ_\lambda b-b\circ_{-\lambda-\partial}a, \;\;a_\lambda b:=a\succ_\lambda b+b\succ_{-\lambda-\partial}a,\; \;\;a, b\in P.
\end{eqnarray}
Then $(P,[\cdot_\lambda \cdot], \cdot_\lambda \cdot)$ is a Poisson
conformal algebra, which is called the {\bf associated Poisson
conformal algebra} of $(P, \circ_\lambda, \succ_\lambda)$.
Moreover, $(P,  \mathfrak{L}_\circ, \mathfrak{L}_\succ)$ is a representation of
$(P,[\cdot_\lambda \cdot], \cdot_\lambda \cdot)$ and the identity
map $I: P\rightarrow P$ is an $\mathcal{O}$-operator on $(P,
[\cdot_\lambda \cdot], \cdot_\lambda \cdot)$ associated with $(P,
\mathfrak{L}_\circ, \mathfrak{L}_\succ)$, where ${\mathfrak{L}_\circ(a)}_\lambda b=a\circ_\lambda
b$ and ${\mathfrak{L}_\succ(a)}_\lambda b=a\succ_\lambda b$ for all $a$,
$b\in P$.
\end{pro}
\delete{
\cm{It seems that the notations $L_\circ, L_\succ$ might be
confused?}

\cm{In Theorem 5.21, there might be confusions}
\yy{Yes, following your suggestion, I change the notations.}}

\begin{pro}
Let $(P, [\cdot_\lambda \cdot], \cdot_\lambda \cdot)$ be a Poisson conformal algebra and $T: V\rightarrow P$ be an $\mathcal{O}$-operator on $(P, [\cdot_\lambda \cdot], \cdot_\lambda \cdot)$ associated with a representation $(V, \rho_P, l_P)$. Then the following $\lambda$-products
\begin{eqnarray*}
u\circ_\lambda v:={\rho_P(T(u))}_\lambda v,~~u\succ_\lambda v:={l_P(T(u))}_\lambda v,~~~~~u,~~v\in V,
\end{eqnarray*}
endow a pre-Poisson conformal algebra structure on $V$.
\end{pro}
\begin{proof}
It is straightforward.
\end{proof}
\delete{\begin{pro}\label{pre-Poisson-1}
Let $(P, [\cdot_\lambda \cdot], \cdot_\lambda \cdot)$ be a Poisson conformal algebra. There is a pre-Poisson conformal algebra on $P$ such that its subjacent Poisson conformal algebra is $(P, [\cdot_\lambda \cdot], \cdot_\lambda \cdot)$ if and only if there exists a bijective $\mathcal{O}$-operator on $(P, [\cdot_\lambda \cdot], \cdot_\lambda \cdot)$ associated with some module $(V, \rho_P, l_P)$.
\end{pro}
\begin{proof}
Let $(P, \circ_\lambda, \succ_\lambda)$ be a pre-Poisson conformal algebra and $(P, [\cdot_\lambda \cdot], \cdot_\lambda \cdot)$ be the associated Poisson conformal algebra. Then it is easy to see that the identity map $I: P\rightarrow P$ is a bijective $\mathcal{O}$-operator on $(P, [\cdot_\lambda \cdot], \cdot_\lambda \cdot)$ associated with $(P,  L_\circ, L_\succ)$.

Conversely, suppose that there exists a bijective $\mathcal{O}$-operator $T$ on $(P, [\cdot_\lambda \cdot], \cdot_\lambda \cdot)$ associated with a module $(V, \rho_P, l_P)$. Then
\begin{eqnarray*}
a\circ_\lambda b=T({\rho_P(a)}_\lambda T^{-1}(b)),~~a_\lambda b=T({l_P(a)}_\lambda T^{-1}(b)),~~~~~~~~~~a,~~b\in P,
\end{eqnarray*}
define a pre-Poisson conformal algebra such that $(P, [\cdot_\lambda \cdot], \cdot_\lambda \cdot)$ is its subjacent Poisson conformal algebra.
\end{proof}}
\begin{thm}\label{constr-PCB-pre-P}
Let $(P, \circ_\lambda, \succ_\lambda)$ be a pre-Poisson conformal
algebra in which $P$ is free as a ${\bf k}[\partial]$-module and
$(P, [\cdot_\lambda \cdot], \cdot_\lambda\cdot)$ be the associated
Poisson conformal algebra of $(P, \circ_\lambda, \succ_\lambda)$.
Then $r=\sum_{i=1}^n(e_i\otimes e_i^\ast-e_i^\ast \otimes e_i)$ is
a solution of the PCYBE in $P\ltimes_{\mathfrak{L}_\circ^\ast,
-\mathfrak{L}_\succ^\ast}P^{\ast c}$, where $\{e_1,\ldots, e_n\}$ is a ${\bf
k}[\partial]$-basis of $P$ and $\{e_1^\ast, \ldots, e_n^\ast\}$ is
the dual ${\bf k}[\partial]$-basis of $P^{\ast c}$.
\end{thm}
\begin{proof}
Note that the identity map $I:P\rightarrow P$ corresponds to
$\sum_{i=1}^ne_i\otimes e_i^\ast$. Then it follows directly from
Proposition \ref{pre-Poisson-1} and Theorem
\ref{Poisson-Yang-Baxter-1}.
\end{proof}

\delete{
\begin{pro}
Let $(P, [\cdot_\lambda \cdot], \cdot_\lambda \cdot)$ be a Poisson conformal algebra and $T: V\rightarrow P$ be an $\mathcal{O}$-operator on $(P, [\cdot_\lambda \cdot], \cdot_\lambda \cdot)$ associated with a representation $(V, \rho_P, l_P)$. Then the following $\lambda$-products
\begin{eqnarray*}
u\circ_\lambda v={\rho_P(T(u))}_\lambda v,~~u\succ_\lambda v={l_P(T(u))}_\lambda v,~~~~~u,~~v\in V,
\end{eqnarray*}
endow a pre-Poisson conformal algebra structure on $V$. Moreover, $T(V)\subset P$ is a Poisson conformal subalgebra of $P$ and there is a natural pre-Poisson conformal algebra structure on $T(V)$ given as follows
\begin{eqnarray*}
T(u)\circ_\lambda T(v)=T(u\circ_\lambda v),~~~T(u)\succ_\lambda T(v)=T(u\succ_\lambda v),~~~~u,~~v\in V.
\end{eqnarray*}
In addition, the subjacent Poisson conformal algebra of $T(V)$ is a subalgebra of $(P, [\cdot_\lambda \cdot], \cdot_\lambda \cdot)$ and $T: V\rightarrow P$ is a homomorphism of pre-Poisson conformal algebras.
\end{pro}
\begin{proof}
It can be directly checked.
\end{proof}}

\section{Semi-classical limits of conformal formal deformations of commutative and cocommutative ASI conformal bialgebras}
We show that the semi-classical limits of conformal formal
deformations of commutative and cocommutative ASI conformal
bialgebras are Poisson conformal bialgebras.

First, we recall some notations about formal deformations
\cite{CBG, ES}.

Let ${\bf K}={\bf k}[[h]]$ be the formal power series algebra and $V$ be a vector space. Consider the set of formal series
\begin{eqnarray*}
V[[h]]:=\{\sum_{i=0}^\infty v_ih^i| v_i\in V\}.
\end{eqnarray*}
Let $\sum_{i=0}^\infty v_ih^i$, $\sum_{i=0}^\infty w_ih^i\in V[[h]]$ and $\sum_{i=0}^\infty k_i h^i\in {\bf K}$. Set
\begin{eqnarray*}
\sum_{i=0}^\infty v_ih^i+\sum_{i=0}^\infty w_ih^i:=\sum_{i=0}^\infty(v_i+w_i)h^i,\;\; (\sum_{i=0}^\infty k_i h^i ) (\sum_{i=0}^\infty v_ih^i):=\sum_{i=0}^\infty(\sum_{m+n=i}k_mv_n)h^i.
\end{eqnarray*}
Then $V[[h]]$ becomes a ${\bf K}$-module. $V[[h]]$ is called a {\bf topologically free ${\bf K}$-module} (with the $h$-adic topology).

\begin{defi}\cite{LZ}
Let $P$ be a vector space. An {\bf $h$-adic associative conformal algebra} is a topologically free ${\bf K}$-module $P[[h]]$ equipped with the $h$-adic topology, a ${\bf k}[\partial]$-module structure and a continuous ${\bf K}$-bilinear $\lambda$-product $\circ_\lambda^h: P[[h]]\times P[[h]]\rightarrow P[[h]][[\lambda]]$ such that for each $n\in \mathbb{N}$,
$(P[[h]]/h^nP[[h]], \overline{\circ}_\lambda^h)$ is an associative conformal algebra, where $\overline{\circ}_\lambda^h$ is the natural quotient map of $\circ_\lambda^h$, i.e.
\begin{eqnarray*}
(a+h^nP[[h]])\overline{\circ}_\lambda^h (b+h^nP[[h]]):=a\circ_\lambda^h b+h^nP[[h]],\;\;a, b\in P.
\end{eqnarray*}
Denote it by $(P[[h]], \circ_\lambda^h)$.
\end{defi}

Let $V$ be a ${\bf k}[\partial]$-module. Then $V[[h]]$ is endowed
with a natural ${\bf k}[\partial]$-module structure as follows.
\begin{eqnarray*}
\partial \sum_{i=0}^\infty v_ih^i:=\sum_{i=0}^\infty (\partial v_i)h^i,\;\; \sum_{i=0}^\infty v_ih^i\in V[[h]].
\end{eqnarray*}

\begin{defi}\cite{LZ}
Let $(P, \cdot_\lambda\cdot)$ be a commutative associative
conformal algebra. A {\bf conformal formal deformation} of $(P,
\cdot_\lambda\cdot)$ is an $h$-adic associative conformal algebra
$(P[[h]], \circ_\lambda^h)$ such that the following condition
holds.
\begin{eqnarray}\label{alg-defor}
a \circ_\lambda^h b=a_\lambda b+\sum_{i=1}^\infty h^i\{a_\lambda
b\}_{\mu_i},\;\; a,b\in P,
\end{eqnarray}
where $\{\cdot_\lambda \cdot\}_{\mu_i}: P\times P\rightarrow
P[\lambda]$ for $i\geq 1$ are $\bf k$-bilinear
maps.
\end{defi}

By the notion of $h$-adic associative conformal algebras, it is
straightforward to show
\begin{eqnarray}
\{\partial a_\lambda b\}_{\mu_i}=-\lambda \{a_\lambda b\}_{\mu_i},\;\; \{ a_\lambda \partial b\}_{\mu_i}=(\partial+\lambda)\{a_\lambda b\}_{\mu_i},\;\; a, b\in P, i\geq 1.
\end{eqnarray}

\begin{pro}\cite[Theorem 3.8]{LZ}\label{scl-A}
Let $(P[[h]], \circ_\lambda^h)$ be a conformal formal deformation of a commutative associative conformal algebra $(P, \cdot_\lambda\cdot)$. Define
\begin{eqnarray}\label{bracket}
[a_\lambda b]:=\frac{a\circ_\lambda^hb-b\circ_{-\lambda-\partial}^h a }{h}\;\text{(mod h)}=\{a_\lambda b\}_{\mu_1}-\{b_{-\lambda-\partial}a\}_{\mu_1},\;\; a, b\in P.
\end{eqnarray}
Then $(P, [\cdot_\lambda \cdot], \cdot_\lambda \cdot)$ is a Poisson conformal algebra, which is called the {\bf semi-classical limit} of $(P[[h]], \circ_\lambda^h)$.
\end{pro}

\begin{defi}
Let $P$ be a vector space. An {\bf $h$-adic associative conformal coalgebra} is a topologically free ${\bf K}$-module $P[[h]]$ equipped with the $h$-adic topology, a ${\bf k}[\partial]$-module structure and a continuous ${\bf K}$-linear and ${\bf k}[\partial]$-module homomorphism $\Delta_h: P[[h]]\rightarrow P[[h]]\widehat{\otimes} P[[h]]$ such that $(P[[h]], \Delta_h)$ is an associative conformal coalgebra, where $P[[h]]\widehat{\otimes} P[[h]]$ is the completion of $P[[h]]\otimes_{{\bf K}} P[[h]]$ under the $h$-adic topology.
Denote it by $(P[[h]], \Delta_h)$.
\end{defi}

\begin{defi}
Let $(P, \Delta)$ be a cocommutative associative conformal
coalgebra. A {\bf conformal formal deformation} of $(P, \Delta)$
is an $h$-adic associative conformal coalgebra $(P[[h]],
\Delta_h)$ such that the following condition holds.
\begin{eqnarray}\label{co-defor}
\Delta_h(a)=\Delta(a)+\sum_{i=1}^\infty \Delta_i(a)h^i,\;\;a\in P, %i\geq 1,
\end{eqnarray}
where $\Delta_i: P\rightarrow P\otimes P$ for $i\geq 1$ are ${\bf k}[\partial]$-module homomorphisms.
\end{defi}

\delete{Since $\Delta_h$ is a ${\bf k}[\partial]$-module homomorphism,
$\Delta_i$ is also a ${\bf k}[\partial]$-module homomorphism for
each $i\geq 1$. \cm{whether we use this property directly in the
above definition?}}

\begin{pro}\label{scl-coA}
Let $(P[[h]], \Delta_h)$ be a conformal formal deformation of a cocommutative associative conformal coalgebra $(P, \Delta)$. Define
\begin{eqnarray}\label{cobracket}
\delta(a):=\frac{\Delta_h(a)-\tau \Delta_h(a)}{h}\;\;({\text mod}\;\; h)=\Delta_1-\tau\Delta_1,\;\; a\in P.
\end{eqnarray}
Then $(P, \delta, \Delta)$ is a Poisson conformal coalgebra, which is called the {\bf semi-classical limit} of $(P[[h]], \Delta_h)$.
\end{pro}

\begin{proof}
Define a ${\bf K}$-linear map $\delta_h: P[[h]]\rightarrow P[[h]]\widehat{\otimes} P[[h]]$ by
\begin{eqnarray}
\delta_h(a):=\Delta_h(a)-\tau\Delta_h(a),\;\;a\in P.
\end{eqnarray}
Since $\Delta_h$ is a ${\bf k}[\partial]$-module homomorphism,
$\delta_h$ is a ${\bf k}[\partial]$-module homomorphism. Then $\delta$ is a ${\bf k}[\partial]$-module homomorphism. Since
$(P[[h]], \Delta_h)$ is an associative conformal coalgebra, it is
straightforward to show that $(P[[h]], \delta_h)$ is a Lie
conformal coalgebra and $(P[[h]], \delta_h, \Delta_h)$ is a
Poisson conformal coalgebra. Therefore, we obtain
\begin{eqnarray}
\label{d-coLie}(\delta_h\otimes I)\delta_h(a)=(I\otimes \delta_h)\delta_h(a)-(\tau\otimes I)(I\otimes \delta_h)\delta_h(a),\\
\label{d-cop}(I\otimes \Delta_h)\delta_h(a)=(\delta_h\otimes I)\Delta_h(a)+(\tau\otimes I)(I\otimes \delta_h)\Delta_h(a),\;\;a\in P.\end{eqnarray}
Considering the $h^2$-terms in Eq. (\ref{d-coLie}), we have
\begin{eqnarray*}
(\delta\otimes I)\delta(a)=(I\otimes \delta)\delta(a)-(\tau\otimes I)(I\otimes \delta)\delta(a),\;\;a\in P.
\end{eqnarray*}
Considering the $h$-terms in Eq. (\ref{d-cop}), we obtain
\begin{eqnarray*}
(I\otimes \Delta)\delta(a)=(\delta\otimes I)\Delta(a)+(\tau\otimes I)(I\otimes \delta)\Delta(a),\;\;a\in P.
\end{eqnarray*}
Therefore, $(P, \Delta, \delta)$ is a Poisson conformal coalgebra.
\end{proof}

\begin{defi}
Let $P$ be a vector space. An {\bf $h$-adic ASI conformal bialgebra} is a topologically free ${\bf K}$-module $P[[h]]$ equipped with the $h$-adic topology such that $(P[[h]], \circ_\lambda^h)$ is an $h$-adic associative conformal algebra, $(P[[h]], \Delta_h)$ is an $h$-adic associative conformal coalgebra and for each
$n\in \mathbb{N}$,
$(P[[h]]/h^nP[[h]], \overline{\circ}_\lambda^h, \overline{\Delta_h})$ is an ASI conformal bialgebra,  where $\overline{\Delta_h}$ is the natural quotient map of $\Delta_h$, i.e.
\begin{eqnarray*}
\overline{\Delta_h} (a+h^nP[[h]]):=\Delta_h(a)+h^nP[[h]]\widehat{\otimes}P[[h]],\;\;a\in P.
\end{eqnarray*}
Denote it by $(P[[h]], \circ_\lambda^h, \Delta_h)$.
\end{defi}

\begin{defi}
Let $(P, \cdot_\lambda\cdot, \Delta)$ be a commutative and cocommutative ASI conformal bialgebra. A {\bf conformal formal deformation} of $(P, \cdot_\lambda\cdot, \Delta)$ is an $h$-adic ASI conformal bialgebra $(P[[h]], \circ_\lambda^h, \Delta_h)$ such that Eqs. (\ref{alg-defor}) and (\ref{co-defor}) hold.
\end{defi}

\begin{thm}
Let $(P[[h]], \circ_\lambda^h, \Delta_h)$ be a conformal formal
deformation of a commutative and cocommutative ASI conformal
bialgebra $(P, \cdot_\lambda\cdot, \Delta)$. Then $(P,
[\cdot_\lambda \cdot],\cdot_\lambda\cdot, \delta,\Delta)$ is a
Poisson conformal bialgebra, where $[\cdot_\lambda \cdot]$ and
$\delta$ are defined by Eqs. {\rm (\ref{bracket})} and {\rm
(\ref{cobracket})} respectively.  This Poisson conformal bialgebra
is called the {\bf semi-classical limit} of $(P[[h]],
\circ_\lambda^h, \Delta_h)$.
\end{thm}
\begin{proof}
Define the $\lambda$-bracket $[\cdot_\lambda \cdot]_h$ on $P[[h]]/h^3P[[h]]$ by
\begin{eqnarray*}
[a_\lambda b]_h:=a\overline{\circ}_\lambda^h b-b\overline{\circ}_{-\lambda-\partial}^h a,\;\;a, b\in P.
\end{eqnarray*}
Since $(P[[h]]/h^3P[[h]], \overline{\circ}_\lambda^h,
\overline{\Delta_h})$ is an ASI conformal bialgebra,  we show that
$(P[[h]]/h^3P[[h]],$ $ [\cdot_\lambda \cdot]_h,
\overline{\delta_h}:=\overline{\Delta_h}-\tau\overline{\Delta_h})$
is a Lie conformal bialgebra. Therefore, we obtain
\begin{eqnarray}
\label{comp-Lie}&&\overline{\delta_h}([a_\lambda b]_h)=({\mathfrak{ad}^h(a)}_\lambda \otimes I+I
\otimes {\mathfrak{ad}^h(a)}_\lambda)\overline{\delta_h}(b)-({\mathfrak{ad}^h(b)}_{-\lambda-\partial^{\otimes^2}}\otimes
I+I\otimes  {\mathfrak{ad}^h(b)}_{-\lambda-\partial^{\otimes^2}})\overline{\delta_h}(a),
\end{eqnarray}
where ${\mathfrak{ad}^h(a)}_\lambda b=[a_\lambda b]_h$ for $a$, $b\in P$.
Considering the $h^2$-terms in Eq. (\ref{comp-Lie}), we obtain
\begin{eqnarray*}
\delta([a_\lambda b])=({\mathfrak{ad}_P(a)}_\lambda\otimes I+I\otimes
{\mathfrak{ad}_P(a)}_\lambda)\delta(b)-({\mathfrak{ad}_P(b)}_{-\lambda-\partial^{\otimes^2}}\otimes
I+I\otimes  {\mathfrak{ad}_P(b)}_{-\lambda-\partial^{\otimes^2}})\delta(a),\;\;a, b\in P.
\end{eqnarray*}
Therefore, $(P, [\cdot_\lambda \cdot], \delta)$ is a Lie conformal bialgebra. By Propositions \ref{scl-A} and \ref{scl-coA}, we only need to prove that Eqs. (\ref{Poisson-bialgebra-1}) and (\ref{Poisson-bialgebra-2}) hold.

For the associative conformal algebra $(P[[h]]/h^3P[[h]], \overline{\circ}_{\lambda}^h)$, set
\begin{eqnarray*}
\mathfrak{L}^h(a)_{\lambda}(b):=a\overline{\circ}_\lambda^h b,\;\; \mathfrak{R}^h(a)_{\lambda}(b):=b\overline{\circ}_{-\lambda-\partial}^h a,\;\;a, b\in P.
\end{eqnarray*}
Let $a,b\in P$. Since $(P[[h]]/h^3P[[h]],
\overline{\circ}_{\lambda}^h, \overline{\Delta_h})$ is an
associative conformal bialgebra, we obtain
\begin{eqnarray}
\label{eq:Bi1}&&\overline{\Delta_h}(a\overline{\circ}_{\lambda}^h b)=(I\otimes \mathfrak{L}^h(a)_{\lambda})\overline{\Delta_h}(b)+(\mathfrak{R}^h(b)_{-\lambda-\partial^{\otimes^2}}\otimes I)\overline{\Delta_h}(a),\\
&&\label{eq:Bi2}(\mathfrak{L}^h(b)_{-\lambda-\partial^{\otimes^2}}\otimes
I-I\otimes
\mathfrak{R}^h(b)_{-\lambda-\partial^{\otimes^2}})\overline{\Delta_h}(a)+(I\otimes
\mathfrak{L}^h(a)_\lambda-\mathfrak{R}^h(a)_\lambda\otimes I)\tau
\overline{\Delta_h}(b)=0.%,\;\;a, b\in A.
\end{eqnarray}
By Eq. (\ref{eq:Bi1}), we have
\begin{eqnarray}
\label{eq:Bi3}\overline{\Delta_h}(b\overline{\circ}_{-\lambda-\partial}^h a)=(I\otimes \mathfrak{L}^h(b)_{-\lambda-\partial^{\otimes^2}})\overline{\Delta_h}(a)+(\mathfrak{R}^h(a)_{\lambda}\otimes I)\overline{\Delta_h}(b).
\end{eqnarray}
By Eqs. (\ref{eq:Bi1})-(\ref{eq:Bi3}), we have
\begin{eqnarray*}
&&\overline{\Delta_h}(a\overline{\circ}_{\lambda}^h b-b\overline{\circ}_{-\lambda-\partial}^h a)\\
&=&(I\otimes \mathfrak{L}^h(a)_{\lambda}-\mathfrak{R}^h(a)_{\lambda}\otimes I)
\overline{\Delta_h}(b)+(\mathfrak{R}^h(b)_{-\lambda-\partial^{\otimes^2}}\otimes I-I\otimes \mathfrak{L}^h(b)_{-\lambda-\partial^{\otimes^2}})\overline{\Delta_h}(a)\\
&=&(I\otimes \mathfrak{L}^h(a)_{\lambda}-\mathfrak{R}^h(a)_{\lambda}\otimes I)
\overline{\Delta_h}(b)+(\mathfrak{R}^h(b)_{-\lambda-\partial^{\otimes^2}}\otimes I-I\otimes \mathfrak{L}^h(b)_{-\lambda-\partial^{\otimes^2}})\overline{\Delta_h}(a)\\
&&\quad-(\mathfrak{L}^h(b)_{-\lambda-\partial^{\otimes^2}}\otimes I-I\otimes \mathfrak{R}^h(b)_{-\lambda-\partial^{\otimes^2}})\overline{\Delta_h}(a)-(I\otimes \mathfrak{L}^h(a)_\lambda-\mathfrak{R}^h(a)_\lambda\otimes I)\tau \overline{\Delta_h}(b)\\
&=&(I\otimes \mathfrak{L}^h(a)_{\lambda}-\mathfrak{R}^h(a)_\lambda\otimes I)(\overline{\Delta_h}(b)-\tau\overline{\Delta_h}(b))\\
&&\quad +(\mathfrak{R}^h(b)_{-\lambda-\partial^{\otimes^2}}\otimes I-I\otimes \mathfrak{L}^h(b)_{-\lambda-\partial^{\otimes^2}}-\mathfrak{L}^h(b)_{-\lambda-\partial^{\otimes^2}}\otimes I+I\otimes \mathfrak{R}^h(b)_{-\lambda-\partial^{\otimes^2}})\overline{\Delta_h}(a).
\end{eqnarray*}
Set $\Delta(a)=\sum a_{(1)}\otimes a_{(2)}$ for $a\in P$.
Considering the $h$-terms in the above equation, we get
\begin{eqnarray*}
\Delta([a_\lambda b])&=&(I\otimes \mathfrak{L}_P(a)_\lambda-\mathfrak{R}_P(a)_\lambda\otimes I)\delta(b)\\
&&\quad+(\mathfrak{R}_P(a)_{-\lambda-\partial^{\otimes^2}}\otimes I-I\otimes\mathfrak{ L}_P(b)_{-\lambda-\partial^{\otimes^2}}-\mathfrak{L}_P(b)_{-\lambda-\partial^{\otimes^2}}\otimes I+I\otimes \mathfrak{R}_P(b)_{-\lambda-\partial^{\otimes^2}})\Delta_1(a)\\
&&\quad+\sum(\{{a_{(1)}}_{\lambda+I\otimes \partial} b\}_{\mu_1}\otimes a_{(2)}-a_{(1)}\otimes \{b_{-\lambda-\partial^{\otimes^2}}a_{(2)}\}_{\mu_1}\\
&&\quad-\{b_{-\lambda-\partial^{\otimes^2}}a_{(1)}\}_{\mu_1}\otimes a_{(2)}+a_{(1)}\otimes \{{a_{(2)}}_{\lambda+\partial \otimes I}b\}_{\mu_1}).
\end{eqnarray*}
Since $(P, \cdot_\lambda\cdot)$ is commutative, we have
\begin{eqnarray*}
\Delta([a_\lambda b])=(I\otimes \mathfrak{L}_P(a)_\lambda-\mathfrak{L}_P(a)_\lambda\otimes I)\delta(b)-({\mathfrak{ad}_P(b)}_{-\lambda-\partial^{\otimes^2}}\otimes I+I\otimes{\mathfrak{ ad}_P(b)}_{-\lambda-\partial^{\otimes^2}})\Delta(a).
\end{eqnarray*}
Then Eq. (\ref{Poisson-bialgebra-2}) holds.

By Eq. (\ref{eq:Bi1}), we have
\begin{eqnarray}
\label{eq:Bi4}\tau\overline{\Delta_h}(a\overline{\circ}_{\lambda}^h b)=( \mathfrak{L}^h(a)_{\lambda}\otimes I)\tau\overline{\Delta_h}(b)+(I\otimes \mathfrak{R}^h(b)_{-\lambda-\partial^{\otimes^2}})\tau\overline{\Delta_h}(a).
\end{eqnarray}
Then by Eqs. (\ref{eq:Bi1}), (\ref{eq:Bi2}) and (\ref{eq:Bi4}), we
obtain {\small\begin{eqnarray*}
&&(\overline{\Delta_h}-\tau\overline{\Delta_h})(a\overline{\circ}_{\lambda}^h b)\\
&&\quad=(I\otimes \mathfrak{L}^h(a)_{\lambda})\overline{\Delta_h}(b)+(\mathfrak{R}^h(b)_{-\lambda-\partial^{\otimes^2}}\otimes I)\overline{\Delta_h}(a)-( \mathfrak{L}^h(a)_{\lambda}\otimes I)\tau\overline{\Delta_h}(b)-(I\otimes \mathfrak{R}^h(b)_{-\lambda-\partial^{\otimes^2}})\tau\overline{\Delta_h}(a)\\
&&\quad=(I\otimes \mathfrak{L}^h(a)_{\lambda})\overline{\Delta_h}(b)+(\mathfrak{R}^h(b)_{-\lambda-\partial^{\otimes^2}}\otimes I)\overline{\Delta_h}(a)-( \mathfrak{L}^h(a)_{\lambda}\otimes I)\tau\overline{\Delta_h}(b)-(I\otimes\mathfrak{ R}^h(b)_{-\lambda-\partial^{\otimes^2}})\tau\overline{\Delta_h}(a)\\
&&\qquad-(\mathfrak{L}^h(b)_{-\lambda-\partial^{\otimes^2}}\otimes I-I\otimes \mathfrak{R}^h(b)_{-\lambda-\partial^{\otimes^2}})\overline{\Delta_h}(a)-(I\otimes \mathfrak{L}^h(a)_\lambda-\mathfrak{R}^h(a)_\lambda\otimes I)\tau \overline{\Delta_h}(b)\\
&&\quad=(\mathfrak{R}^h(b)_{-\lambda-\partial^{\otimes^2}}\otimes I-\mathfrak{L}^h(b)_{-\lambda-\partial^{\otimes^2}}\otimes I)\overline{\Delta_h}(a)+(I\otimes \mathfrak{R}^h(b)_{-\lambda-\partial^{\otimes^2}})(\overline{\Delta_h}(a)-\tau\overline{\Delta_h}(a))\\
&&\qquad+(I\otimes \mathfrak{L}^h(a)_{\lambda})(\overline{\Delta_h}(b)-\tau\overline{\Delta_h}(b))-(\mathfrak{L}^h(a)_{\lambda}\otimes I-\mathfrak{R}^h(a)_\lambda\otimes I)\tau\overline{\Delta_h}(b).
\end{eqnarray*}}
Considering the $h$-terms in the above equation, we have
\begin{eqnarray*}
\delta(a_\lambda b)&=&\sum(\{{a_{(1)}}_{\lambda+\partial\otimes I}b\}_{\mu_1}\otimes a_{(2)}-\{b_{-\lambda-\partial^{\otimes^2}}a_{(1)}\}_{\mu_1}\otimes a_{(2)})\\
&&\quad+(\mathfrak{R}_P(b)_{-\lambda-\partial^{\otimes^2}}\otimes I-\mathfrak{L}_P(b)_{-\lambda-\partial^{\otimes^2}}\otimes I)\Delta_1(a)\\
&&\quad+(I\otimes \mathfrak{R}_P(b)_{-\lambda-\partial^{\otimes^2}})\delta(a)+(I\otimes \mathfrak{L}_P(a)_\lambda)\delta(b)\\
&&\quad-\sum (\{a_\lambda b_{(2)}\}_{\mu_1}\otimes b_{(1)}-\{{b_{(2)}}_{-\lambda-\partial}a\}_{\mu_1}\otimes b_{(1)})-(\mathfrak{L}_P(a)_{\lambda}\otimes I-\mathfrak{R}_P(a)_\lambda\otimes I)\tau \Delta_1(b).
\end{eqnarray*}
Since $(P, \cdot_\lambda\cdot)$ is commutative, we obtain {\small
\begin{eqnarray*} \delta(a_\lambda
b)=-(\mathfrak{ad}_P(b)_{-\lambda-\partial^{\otimes^2}}\otimes
I)\Delta(a)+(I\otimes
\mathfrak{L}_P(b)_{-\lambda-\partial^{\otimes^2}})\delta(a)+(I\otimes
\mathfrak{L}_P(a)_\lambda)\delta(b)-(\mathfrak{ad}_P(a)_\lambda\otimes I)\Delta(b).
\end{eqnarray*}}
Then Eq. (\ref{Poisson-bialgebra-1}) follows directly. Therefore
the proof is completed.
\end{proof}

\section{Poisson conformal bialgebras from PGD-bialgebras and pre-PGD-algebras}
We introduce the notion of Poisson-Gel'fand-Dorfman bialgebras and
show that there is a correspondence between them and a class of
Poisson conformal bialgebras. Moreover, a construction of Poisson
conformal bialgebras from pre-Poisson-Gel'fand-Dorfman algebras is
given.

\subsection{A class of Poisson conformal bialgebras corresponding to
PGD-bialgebras}\

Recall \cite{HBG, KUZ} that a {\bf Novikov coalgebra} $(A,
\Delta)$ is a vector space $A$ with a linear map $\Delta:
A\rightarrow A\otimes A$ such that the  following conditions hold.
\begin{eqnarray}
\label{Lc3}(I\otimes \Delta)\Delta(a)-(\tau\otimes I)(I\otimes \Delta)\Delta(a)&=&(\Delta\otimes I)\Delta(a)-(\tau\otimes I)(\Delta\otimes I)\Delta(a),\\
\label{Lc4}(\tau\otimes I)(I\otimes \Delta)\tau
     \Delta(a)&=&(\Delta\otimes I)\Delta(a), \;\;a\in A.
\end{eqnarray}

A {\bf cocommutative associative coalgebra } $(A, \Delta)$ is a
vector space with a linear map $\Delta: A\rightarrow A\otimes A$
such that the  following condition holds.
\begin{eqnarray}
(\Delta\otimes I)\Delta(a)=(I\otimes \Delta)\Delta(a),\;\;\;\Delta(a)=\tau \Delta(a),\;\;a\in A.
\end{eqnarray}

A {\bf Lie coalgebra} $(A, \delta)$ is a vector space with a
linear map $\delta: A\rightarrow A\otimes A$ such that the
following condition holds.
\begin{eqnarray}
\delta(a)=-\tau \delta(a),\;\; (I\otimes \delta)\delta(a)-(\tau\otimes I)(I\otimes \delta)\delta(a)=(\delta\otimes I)\delta(a),\;\; a\in A.
\end{eqnarray}

\begin{defi} \cite{HBG1}\quad
A {\bf Gel'fand-Dorfman coalgebra (GD-coalgebra)} is a triple $(A,
\Delta, \delta)$ where $(A, \Delta)$ is a Novikov coalgebra, $(A,
\delta)$ is a Lie coalgebra and they satisfy the following
condition
\begin{eqnarray}
\label{Lc2}&&(I \otimes \delta)\Delta(a)-(\tau\otimes I)(I\otimes \Delta)\delta(a)
+(\tau\otimes I)(I \otimes \delta)\tau \Delta(a)=(\Delta\otimes I)\delta(a)+(\delta\otimes I)\Delta(a),\;\;a\in A.
\end{eqnarray}
\end{defi}

\begin{defi} \cite{NB}
A {\bf Poisson coalgebra} is a triple $(A, \Delta, \delta)$ where
$(A, \Delta)$ is a cocommutative associative coalgebra, $(A,
\delta)$ is a Lie coalgebra and they satisfy the following
condition
\begin{eqnarray}
\label{Pc}&& (I\otimes \Delta)\delta(a)=(\delta\otimes I)\Delta(a)+(\tau \otimes I)(I\otimes \delta)\Delta(a),\;\; a\in A.
\end{eqnarray}
\end{defi}

Let $\langle \cdot, \cdot \rangle$ be the usual pairing between a
vector space $V$ and the dual vector space $V^*$. Let $V$ and $W$
be two vector spaces. For a linear map $\varphi: V\rightarrow W$,
denote the transpose map by $\varphi^\ast: W^\ast\rightarrow
V^\ast$ given by
\begin{eqnarray*}
\langle \varphi^\ast(f),v\rangle=\langle f, \varphi(v)\rangle,\;\;\; f\in W^\ast, v\in V.
\end{eqnarray*}
Note that $(A, \Delta, \delta)$ is a GD-coalgebra if and only if $(A^\ast, \Delta^\ast, \delta^\ast)$ is a GD-algebra, and $(A, \Delta, \delta)$ is a Poisson coalgebra if and only if $(A^\ast, \Delta^\ast, \delta^\ast)$ is a Poisson algebra.

%Next, we introduce the definition of differential Novikov-Poisson coalgebras.
\begin{defi}
A {\bf differential Novikov-Poisson coalgebra} is a triple $(A,
\Delta_1, \Delta_2)$ such that $(A, \Delta_1)$ is a Novikov
coalgebra, $(A, \Delta_2)$ is a cocommutative associative
coalgebra and they satisfy the following conditions
\begin{eqnarray}
&&\label{DNPC-1}(I\otimes \Delta_2)\Delta_1(a)=(\Delta_1\otimes I)\Delta_2(a)+(\tau\otimes I)(I\otimes \Delta_1)\Delta_2(a),\\
&&\label{DNPC-2}(I\otimes \Delta_2)\tau \Delta_1(a)=(\tau\otimes I)(\Delta_1\otimes I)\Delta_2(a),\;\;a\in A.
\end{eqnarray}
\end{defi}

It is straightforward to show that $(A, \Delta_1, \Delta_2)$ is a
differential Novikov-Poisson coalgebra if and only if $(A^\ast,
\Delta_1^\ast, \Delta_2^\ast)$ is a differential Novikov-Poisson
algebra.

\begin{defi}
A {\bf Poisson-Gel'fand-Dorfman coalgebra (PGD-coalgebra)} is a
quadruple $(A, \Delta_1,$ $ \Delta_2, \delta)$ such that  $(A,
\Delta_1, \delta)$ is a GD-coalgebra, $(A$, $\Delta_2$, $\delta)$
is a Poisson coalgebra and $(A, \Delta_1, \Delta_2)$ is a
differential Novikov-Poisson coalgebra.

\end{defi}
%Next, we present a correspondence between PGD-coalgebras and a class of Poisson conformal coalgebras.
\begin{pro} Let $A$ be a vector space and $P={\bf k}[\partial]\otimes_{\bf k} A$ be a free ${\bf k}[\partial]$-module. Suppose that
 $\Delta_1$, $\Delta_2$ and $\delta_0$ are ${\bf k}[\partial]$-module homomorphisms from $P$ to $P\otimes P$ such that $\Delta_i(A)\subseteq A\otimes A$ for $i=1$, $2$ and $\delta_0(A)\subseteq A\otimes A$ .
Define two ${\bf k}[\partial]$-module homomorphisms  $\delta$ and
$\Delta$ by first taking
\begin{eqnarray}\label{corr-co}
\delta(a):=(\partial \otimes I)\Delta_1(a)-\tau (\partial \otimes I)\Delta_1(a)+\delta_0(a),\;\;\Delta(a):=\Delta_2(a),\;\;a\in A,
\end{eqnarray}
and then extending to $P$. %\cm{$\delta$ and $\Delta$ are defined on $P$ or $A$?}
Then
$(P,\delta, \Delta)$ is a Poisson conformal coalgebra if and only
if $(A, \Delta_1, \Delta_2,$ $ \delta_0)$ is a PGD-coalgebra.
\end{pro}

\begin{proof}
By \cite[Proposition 2.8]{HBG1}, $(P, \delta)$ is a Lie conformal
coalgebra if and only if $(A, \Delta_1, \delta_0)$ is a
GD-coalgebra. It is obvious that $(P, \Delta)$ is a cocommutative
associative conformal coalgebra if and only if $(A, \Delta_2)$ is
a cocommutative associative coalgebra. Therefore, we only need to
consider $(I\otimes \Delta)\delta(a)=(\delta\otimes
I)\Delta(a)+(\tau\otimes I)(I\otimes \delta)\Delta(a)$ for all
$a\in A$. Let $a\in A$. Set $\Delta_1(a)=\sum a_{(1)}\otimes
a_{(2)}$, $\Delta_2(a)=\sum a_{[1]}\otimes a_{[2]}$ and
$\delta_0(a)=\sum a_{\{1\}}\otimes a_{\{2\}}$.   Then we obtain
\begin{eqnarray*}
&&(I\otimes \Delta)\delta(a)-(\delta\otimes I)\Delta(a)-(\tau\otimes I)(I\otimes \delta)\Delta(a)\\
&=&\sum\sum(\partial a_{(1)}\otimes a_{(2)[1]}\otimes a_{(2)[2]}-a_{(2)}\otimes \partial a_{(1)[1]}\otimes a_{(1)[2]}-a_{(2)}\otimes a_{(1)[1]}\otimes \partial a_{(1)[2]}\\
&&-\partial a_{[1](1)}\otimes a_{[1](2)}\otimes a_{[2]}+a_{[1](2)}\otimes \partial a_{[1](1)}\otimes  a_{[2]}\\
&&-\partial a_{[2](1)}\otimes a_{[1]}\otimes  a_{[2](2)}+a_{[2](2)}\otimes a_{[1]}\otimes \partial a_{[2](1)})\\
&&+\sum\sum (a_{\{1\}}\otimes a_{\{2\}[1]}\otimes a_{\{2\}[2]}-a_{[1]\{1\}}\otimes a_{[1]\{2\}}\otimes a_{[2]}-a_{[2]\{1\}}\otimes a_{[1]}\otimes a_{[2]\{2\}}).
\end{eqnarray*}
Comparing the terms in $\partial A\otimes A\otimes A$, $A\otimes
\partial A\otimes A$, $A\otimes A\otimes \partial A$ and $A\otimes
A\otimes A$, we obtain that $(I\otimes
\Delta)\delta(a)=(\delta\otimes I)\Delta(a)+(\tau\otimes
I)(I\otimes \delta)\Delta(a)$ holds if and only if the following
equations hold.
\begin{eqnarray}
&&\label{corr-co-1}\sum\sum a_{(1)}\otimes a_{(2)[1]}\otimes a_{(2)[2]}=\sum\sum (a_{[1](1)}\otimes a_{[1](2)}\otimes a_{[2]}+a_{[2](1)}\otimes a_{[1]}\otimes  a_{[2](2)}),\\
&&\label{corr-co-2}\sum\sum a_{(2)}\otimes a_{(1)[1]}\otimes a_{(1)[2]}=\sum\sum a_{[1](2)}\otimes a_{[1](1)}\otimes  a_{[2]},\\
&&\label{corr-co-3}\sum\sum a_{(2)}\otimes a_{(1)[1]}\otimes a_{(1)[2]}=\sum\sum a_{[2](2)}\otimes a_{[1]}\otimes a_{[2](1)},\\
&&\label{corr-co-4}\sum\sum a_{\{1\}}\otimes a_{\{2\}[1]}\otimes a_{\{2\}[2]}=\sum \sum (a_{[1]\{1\}}\otimes a_{[1]\{2\}}\otimes a_{[2]}+a_{[2]\{1\}}\otimes a_{[1]}\otimes a_{[2]\{2\}}).
\end{eqnarray}
Note that Eq. (\ref{corr-co-1}) is just Eq. (\ref{DNPC-1}),  Eq.
(\ref{corr-co-2}) is just Eq. (\ref{DNPC-2}) and Eq.
(\ref{corr-co-4}) is just Eq. (\ref{Pc}). Since $(A,\Delta_2)$ is
cocommutative, Eq. (\ref{corr-co-2}) holds if and only if
 Eq. (\ref{corr-co-3}) holds. Then the proof is completed.
\end{proof}
Let $(A, \circ, \cdot, [\cdot,\cdot])$ be a PGD-algebra. Set $a\star b:=a\circ b+b\circ a$ for all $a$, $b\in A$.

 Recall \cite{HBG} that a {\bf Novikov bialgebra} is a tripe $(A,\circ,\Delta_1)$ where  $(A,\circ)$ is a Novikov algebra and $(A, \Delta_1)$ is a Novikov coalgebra such that
the following conditions are satisfied
\begin{eqnarray}
            &\label{Lb5}\Delta_1(a\circ b)=(R_{A,\circ} (b)\otimes I)\Delta_1 (a)+(I\otimes L_{A,\star}(a))(\Delta_1(b)+\tau\Delta_1(b)),&\\
            &\label{Lb6}(L_{A,\star}(a)\otimes I)\Delta_1(b)-(I\otimes L_{A,\star}(a))\tau\Delta_1 (b)=(L_{A,\star}(b)\otimes I)\Delta_1(a)-(I\otimes L_{A,\star}(b))\tau\Delta_1(a),&\\
            &\label{Lb7}\quad (I\otimes R_{A,\circ}(a)-R_{A,\circ}(a)\otimes
            I)(\Delta_1(b)+\tau\Delta_1(b))=(I\otimes R_{A,\circ}(b)-R_{A,\circ}(b)\otimes
            I)(\Delta_1(a)+\tau\Delta_1(a)),&
    \end{eqnarray}
for all $a$, $b\in A$.

    Let $(A, \cdot)$ be a commutative associative algebra and $(A, \Delta_2)$ be a cocommutative
associative coalgebra. If $\Delta_2$ satisfies the following
condition
\begin{eqnarray}
&&\mlabel{ASI1}\Delta_2(a\cdot b)=(I\otimes L_{A,\cdot}(a))\Delta_2(b)+(L_{A,\cdot}(b)\otimes I)\Delta_2(a),\;\;a, b\in A,
\end{eqnarray}
then $(A, \cdot, \Delta_2)$ is called a {\bf commutative and
cocommutative ASI bialgebra} \cite{Bai1}.

A {\bf Lie bialgebra} is a triple $(A, [\cdot,\cdot], \delta)$
such that $(A, [\cdot,\cdot])$ is a Lie algebra, $(A, \delta)$ is
a Lie coalgebra and they satisfy the following condition
\begin{eqnarray*}
\delta([a,b])=(\ad_A(a)\otimes I+I\otimes \ad_A(a))\delta(b)-(\ad_A(b)\otimes I+I\otimes \ad_A(b))\delta(a), \;\;a, b\in A.
\end{eqnarray*}

%We recall the definitions of Gel'fand-Dorfman bialgebras and Poisson bialgebras.
\begin{defi}\cite{HBG1}
Let $(A, \circ, [\cdot,\cdot])$ be a GD-algebra. If there are two linear maps $\delta$, $\Delta: A\rightarrow A\otimes A$ such that $(A, \Delta, \delta)$ is a GD-coalgebra, $(A, \circ, \Delta)$ is a Novikov bialgebra, $(A, [\cdot,\cdot], \delta)$ is a Lie bialgebra, and they satisfy the following compatibility condition
\begin{eqnarray}
%\label{Lb2}&&\delta(a\star b)=(L_{A,\star}(a)\otimes \id+\id\otimes L_{A,\star} (a))\delta(b)
%+(L_{A,\star}(b)\otimes \id+\id\otimes L_{A,\star} (b))\delta(a)\\
%&&\;\;\;\;\;\;\;\;\;+(\ad(a)\otimes \id)\Delta(b)-(\id\otimes \ad(a))\tau \Delta(b)+(\ad(b)\otimes \id)\Delta(a)-(\id\otimes \ad(b))\tau \Delta(a),\nonumber\\\
\label{Lb4}&&\delta(b\circ a)+\Delta([a,b])=(R_{A,\circ}(a)\otimes I)\delta(b)+(L_{A,\circ}(b)\otimes I)\delta(a)+(I\otimes L_{A,\star}(b))\delta(a)\\
&&\;\;\;\;\;\;\;\;\;+(\ad_A(a)\otimes I+I\otimes \ad_A(a))\Delta(b)-(I\otimes \ad_A(b))(\Delta(a)+\tau\Delta(a)),\;\;a, b\in A,\nonumber
\end{eqnarray}
then $(A, \circ, [\cdot,\cdot], \Delta, \delta)$ is called a {\bf Gel'fand-Dorfman bialgebra (GD-bialgebra)}.
\end{defi}

\begin{defi} \cite{NB}
A {\bf Poisson bialgebra} is a 5-tuple $(A, \cdot, [\cdot,\cdot],
\Delta, \delta)$ such that $(A, \cdot, [\cdot,\cdot])$ is a
Poisson algebra, $(A, \Delta, \delta)$ is a Poisson coalgebra, $(A, \cdot, \Delta)$ is a commutative and cocommutative ASI bialgebra, $(A, [\cdot,\cdot], \delta)$ is a Lie bialgebra and
they satisfy the following compatibility conditions
\begin{eqnarray}
\label{PB1}&&\delta(a\cdot b)=(L_{A,\cdot}(a)\otimes I)\delta(b)+(L_{A,\cdot}(b)\otimes I)\delta(a)+(I\otimes \ad_A(a))\Delta(b)+(I\otimes \ad_A(b))\Delta(a),\\
\label{PB2}&&\Delta([a,b])=(\ad_A(a)\otimes I+I\otimes
\ad_A(a))\Delta(b)-(I\otimes
L_{A,\cdot}(b))\delta(a)+(L_{A,\cdot}(b)\otimes I)\delta(a),
\end{eqnarray}
for all $a$, $b\in A$.
\end{defi}

%Next, we introduce the definition of differential Novikov-Poisson bialgebras.
\begin{defi}
Let $(A, \circ, \cdot)$ be a differential Novikov-Poisson algebra and $(A, \Delta_1, \Delta_2)$ be a differential Novikov-Poisson coalgebra. If $(A, \circ, \Delta_1)$ is a Novikov bialgebra, $(A, \cdot, \Delta_2)$ is a commutative and cocommutative ASI bialgebra and they satisfy the following compatibility conditions
\begin{eqnarray}
\label{DNPB1}\Delta_1(a\cdot b)=(L_{A,\cdot}(a)\otimes I)\Delta_1(b)-(I\otimes L_{A,\star}(b))\Delta_2(a),\\
\label{DNPB2}\Delta_2(a\circ b)=(R_{A,\circ}(b)\otimes I)\Delta_2(a)-(I\otimes L_{A,\cdot}(a))(\Delta_1(b)+\tau \Delta_1(b)),\\
\label{DNPB4}\;\;(L_{A,\circ}(a)\otimes I)\Delta_2(b)+(I\otimes L_{A,\circ}(b))\Delta_2(a)+(I\otimes L_{A,\cdot}(b))\Delta_1(a)+(L_{A,\cdot}(a)\otimes I)\tau \Delta_1(b)=0,
\end{eqnarray}
for all $a$, $b\in A$, then $(A, \circ, \cdot, \Delta_1, \Delta_2)$ is called a {\bf differential Novikov-Poisson bialgebra}.
\end{defi}
\begin{rmk}Since $(A, \cdot)$ is commutative, by Eq. (\ref{DNPB1}),  we obtain
{\small\begin{eqnarray} \label{DNPB3}(L_{A,\cdot}(a)\otimes
I)\Delta_1(b)+(I\otimes
L_{A,\star}(a))\Delta_2(b)=(L_{A,\cdot}(b)\otimes
I)\Delta_1(a)+(I\otimes L_{A,\star}(b))\Delta_2(a),\;\;a,b\in A.
\end{eqnarray}}
\end{rmk}
\begin{defi}
A {\bf Poisson-Gel'fand-Dorfman bialgebra (PGD-bialgebra)} is a
7-tuple $(A, \circ$, $\cdot$, $[\cdot,\cdot], \Delta_1, \Delta_2,
\delta)$ such that $(A, \circ, [\cdot,\cdot], \Delta_1, \delta)$
is a GD-bialgebra, $(A, \cdot, [\cdot,\cdot], \Delta_2, \delta)$
is a Poisson bialgebra and $(A, \circ, \cdot, \Delta_1, \Delta_2)$
is a differential Novikov-Poisson bialgebra.
\end{defi}
%\begin{rmk}
Obviously,  GD-bialgebras, Poisson bialgebras and differential
Novikov-Poisson bialgebras are special PGD-bialgebras.
%\end{rmk}
Moreover, there is a correspondence between a class of Poisson
conformal bialgebras and PGD-bialgebras.
\begin{thm}\label{corr-bialg}
Let $(P={\bf k}[\partial]\otimes_{\bf k} A, [\cdot_\lambda \cdot], \cdot_\lambda
\cdot)$ be the Poisson conformal algebra corresponding to a
PGD-algebra $(A,\circ, \cdot, [\cdot,\cdot])$. Suppose that
 $\Delta_1$, $\Delta_2$ and $\delta_0$ are ${\bf k}[\partial]$-module homomorphisms from $P$ to $P\otimes P$ such that $\Delta_i(A)\subseteq A\otimes A$ for $i=1$, $2$ and $\delta_0(A)\subseteq A\otimes A$.
Define two ${\bf k}[\partial]$-module homomorphisms $\delta$ and
$\Delta$ by Eq. {\rm (\ref{corr-co})}.
%\begin{eqnarray}\label{cobracket}
%\delta(a):=(\partial \otimes I)\Delta_1(a)-\tau (\partial \otimes I)\Delta_1(a),\;\;\Delta(a):=\Delta_2(a),\;\;a\in A.
%\end{eqnarray}
Then $(P={\bf k}[\partial]\otimes_{\bf k} A, [\cdot_\lambda \cdot], \cdot_\lambda
\cdot, \delta, \Delta)$ is a Poisson conformal bialgebra if and
only if $(A, \circ, \cdot, [\cdot,\cdot], \Delta_1, \Delta_2,
\delta_0)$ is a PGD-bialgebra.
\end{thm}
\begin{proof}
By \cite[Theorem 2.15]{HBG1}, $({\bf k}[\partial]\otimes_{\bf k} A, [\cdot_\lambda
\cdot], \delta)$ is a Lie conformal bialgebra if and only if $(A,
\circ, [\cdot,\cdot], \Delta_1, \delta_0)$ is a GD-bialgebra.
Moreover, it is obvious that $({\bf k}[\partial]\otimes_{\bf k} A, \cdot_\lambda
\cdot, \Delta)$ is a commutative and cocommutative ASI conformal
bialgebra if and only if $(A, \cdot, \Delta_2)$ is commutative and
cocommutative ASI bialgebra. Assume that $\Delta_1(a)=\sum
a_{(1)}\otimes a_{(2)}$, $\Delta_2(a)=\sum a_{[1]}\otimes a_{[2]}$
and $\delta_0(a)=\sum a_{\{1\}}\otimes a_{\{2\}}$ for any $a\in
A$. Then for all $a$, $b\in A$, we obtain
\begin{eqnarray*}
 &&\delta(a_\lambda b)-({\mathfrak{L}_P(a)}_\lambda\otimes I)\delta(b)-({\mathfrak{L}_P(b)}_{-\lambda-\partial^{\otimes^2}}\otimes I)\delta(a)-(I\otimes {\mathfrak{ad}_P(a)}_\lambda)\Delta(b)\\
 &&\quad\quad-(I\otimes {\mathfrak{ad}_P(b)}_{-\lambda-\partial^{\otimes^2}})\Delta(a)\\
&=& \sum (\partial (a\cdot b)_{(1)}\otimes (a\cdot b)_{(2)}-(a\cdot b)_{(2)}\otimes \partial (a\cdot b)_{(1)})+\sum (a\cdot b)_{\{1\}}\otimes (a\cdot b)_{\{2\}}\\
&&-\sum ((\partial+\lambda)a\cdot b_{(1)}\otimes b_{(2)}-a\cdot b_{(2)}\otimes \partial b_{(1)})-\sum a\cdot b_{\{1\}}\otimes b_{\{2\}}\\
&&+\sum ((\lambda+I\otimes \partial)b\cdot a_{(1)}\otimes a_{(2)}+b\cdot a_{(2)}\otimes \partial a_{(1)})-\sum b\cdot a_{\{1\}}\otimes a_{\{2\}}\\
&&-\sum b_{[1]}\otimes (\partial (b_{[2]}\circ a)+\lambda(a\star b_{[2]})+[a,b_{[2]}])\\
&&-\sum a_{[1]}\otimes (\partial (a_{[2]}\circ b)+(-\lambda-\partial^{\otimes^2})(a_{[2]}\star b)+[b,a_{[2]}).
\end{eqnarray*}
Therefore, by considering the terms in  $\partial A\otimes A$,
$A\otimes \partial A$, $\lambda A\otimes A$ and $A\otimes A$, Eq.
(\ref{Poisson-bialgebra-1}) holds if and only if the following
equations hold.
\begin{eqnarray}
&&\label{condd1}\sum (a\cdot b)_{(1)}\otimes (a\cdot b)_{(2)}=\sum a\cdot b_{(1)}\otimes b_{(2)}-\sum a_{[1]}\otimes a_{[2]}\star b,\\
&&\label{condd2}\sum (a\cdot b)_{(2)}\otimes (a\cdot b)_{(1)}=\sum a\cdot b_{(2)}\otimes b_{(1)}+\sum (b\cdot a_{(1)}\otimes a_{(2)}+b\cdot a_{(2)}\otimes a_{(1)})\\
&&\qquad\quad-\sum b_{[1]}\otimes b_{[2]}\circ a+\sum a_{[1]}\otimes b\circ a_{[2]},\nonumber\\
&&\label{condd3}\sum a\cdot b_{(1)}\otimes b_{(2)}-\sum b\cdot a_{(1)}\otimes a_{(2)}+\sum b_{[1]}\otimes a\star b_{[2]}-\sum a_{[1]}\otimes a_{[2]}\star b=0,\\
&&\label{Pcondd} \sum (a\cdot b)_{\{1\}}\otimes (a\cdot b)_{\{2\}}=\sum a\cdot b_{\{1\}}\otimes b_{\{2\}} +\sum  b\cdot a_{\{1\}}\otimes a_{\{2\}}\\
&&\quad\quad+\sum b_{[1]}\otimes [a, b_{[2]}]+\sum a_{[1]}\otimes
[b, a_{[2]}],\nonumber
\end{eqnarray}
for all $a$, $b\in A$. Note that Eq. (\ref{condd1}) is just Eq.
(\ref{DNPB1}), Eq. (\ref{condd3}) is just Eq. (\ref{DNPB3}) and
Eq. (\ref{Pcondd}) is just Eq. (\ref{PB1}). Similarly, for all
$a,b\in A$, we get
\begin{eqnarray*}
&&\Delta([a_\lambda b])-({\mathfrak{ad}_P(a)}_\lambda\otimes I+I\otimes {\mathfrak{ad}_P(a)}_\lambda)\Delta(b)+(I\otimes {\mathfrak{L}_P(b)}_{-\lambda-\partial^{\otimes^2}})\delta(a)-({\mathfrak{L}_P(b)}_{-\lambda-\partial^{\otimes^2}}\otimes I)\delta(a)\\
&=& (\partial\otimes I+I\otimes \partial)\Delta_2(b\circ a)+\lambda \Delta_2(a\star b)+\Delta_2([a,b])\\
&&-\sum (\partial(b_{[1]}\circ a)+\lambda(a\star b_{[1]})+[a, b_{[1]}])\otimes b_{[2]}\\
&&-\sum b_{[1]}\otimes (\partial (b_{[2]}\circ a)+\lambda(a\star b_{[2]})+[a,b_{[2]}])\\
&&+\sum (\partial a_{(1)}\otimes b\cdot a_{(2)}-a_{(2)}\otimes (-\lambda-\partial\otimes I)b\cdot a_{(1)})+\sum a_{\{1\}}\otimes b\cdot a_{\{2\}}\\
&&-\sum ((-\lambda-I\otimes \partial) b\cdot a_{(1)}\otimes a_{(2)}-b\cdot a_{(2)}\otimes \partial a_{(1)})-\sum b\cdot a_{\{1\}}\otimes a_{\{2\}}.
\end{eqnarray*}
Considering the terms of $\partial A\otimes A$, $A\otimes \partial
A$, $\lambda A\otimes A$ and $A\otimes A$, we obtain that Eq.
(\ref{Poisson-bialgebra-2}) holds if and only if the following
equations hold.
\begin{eqnarray}
&&\label{condd4}\Delta_2(b\circ a)=\sum b_{[1]}\circ a\otimes b_{[2]}-\sum (a_{(1)}\otimes b\cdot a_{(2)}+a_{(2)}\otimes b\cdot a_{(1)}),\\
&&\label{condd5}\Delta_2(b\circ a)=\sum b_{[1]}\otimes b_{[2]}\circ a-\sum (b\cdot a_{(1)}\otimes a_{(2)}+b\cdot a_{(2)}\otimes a_{(1)}),\\
&&\label{condd6}\Delta_2(a\star b)=\sum (a\star b_{[1]}\otimes b_{[2]}+b_{[1]}\otimes a\star b_{[2]})-\sum (a_{(2)}\otimes b\cdot a_{(1)}+b\cdot a_{(1)}\otimes a_{(2)}),\\
&&\label{Pcondd1} \Delta_2([a,b])=\sum ([a, b_{[1]}]\otimes b_{[2]}+b_{[1]}\otimes [a, b_{[2]}])-\sum (a_{\{1\}}\otimes b\cdot a_{\{2\}}-b\cdot a_{\{1\}}\otimes a_{\{2\}}),
\end{eqnarray}
for all $a$, $b\in A$. Note that Eq. (\ref{Pcondd1}) is just Eq.
(\ref{PB2}), and Eq. (\ref{condd4}) holds if and only if Eq.
(\ref{DNPB2}) holds. Since $\Delta_2=\tau \Delta_2$, Eq.
(\ref{condd4}) holds if and only if (\ref{condd5}) holds. By Eq.
(\ref{DNPB3}), we get {\small \begin{eqnarray*}
&&\Delta_2(a\star b)-\sum (a\star b_{[1]}\otimes b_{[2]}+b_{[1]}\otimes a\star b_{[2]})+\sum (a_{(2)}\otimes b\cdot a_{(1)}+b\cdot a_{(1)}\otimes a_{(2)})\\
&&=\Delta_2(a\circ b)+\Delta_2(b\circ a)-\sum (a\star b_{[1]}\otimes b_{[2]}+b_{[1]}\otimes a\star b_{[2]})+\sum (a_{(2)}\otimes b\cdot a_{(1)}+b\cdot a_{(1)}\otimes a_{(2)})\\
&&=\sum b_{[1]}\circ a\otimes b_{[2]}-\sum (a_{(1)}\otimes b\cdot a_{(2)}+a_{(2)}\otimes b\cdot a_{(1)})+\sum a_{[2]}\otimes a_{[1]}\circ b\\
&&\quad -\sum (a\cdot b_{(2)}\otimes b_{(1)}+a\cdot b_{(1)}\otimes b_{(2)})\\
&&\quad -\sum (a\star b_{[1]}\otimes b_{[2]}+b_{[1]}\otimes a\star b_{[2]})+\sum (a_{(2)}\otimes b\cdot a_{(1)}+b\cdot a_{(1)}\otimes a_{(2)})\\
&&=-\sum a_{(1)}\otimes b\cdot a_{(2)}+\sum a_{[2]}\otimes a_{[1]}\circ b-\sum (a\cdot b_{(2)}\otimes b_{(1)}+a\cdot b_{(1)}\otimes b_{(2)})\\
&&\quad -\sum a\circ b_{[1]}\otimes b_{[2]}-(\sum b_{[1]}\otimes a \star b_{[2]}-\sum b\cdot a_{(1)}\otimes a_{(2)})\\
&&=-\sum a_{(1)}\otimes b\cdot a_{(2)}+\sum a_{[2]}\otimes a_{[1]}\circ b-\sum (a\cdot b_{(2)}\otimes b_{(1)}+a\cdot b_{(1)}\otimes b_{(2)})\\
&&\quad -\sum a\circ b_{[1]}\otimes b_{[2]}-(\sum a_{[1]}\otimes b \star a_{[2]}-\sum a\cdot b_{(1)}\otimes b_{(2)})\\
&&=-((L_{A,\circ}(a)\otimes I)\Delta_2(b)+(I\otimes
L_{A,\circ}(b))\Delta_2(a)+(I\otimes
L_{A,\cdot}(b))\Delta_1(a)+(L_{A,\cdot}(a)\otimes I)\tau
\Delta_1(b)),
\end{eqnarray*}}
for all $a,b\in A$. Therefore, by Eqs. (\ref{DNPB2}) and
(\ref{DNPB3}), Eq. (\ref{condd6}) holds if and only if Eq.
(\ref{DNPB4}) holds. By Eqs. (\ref{DNPB1}), (\ref{DNPB2}) and
(\ref{DNPB4}), we get {\small \begin{eqnarray*}
&&\sum (a\cdot b)_{(2)}\otimes (a\cdot b)_{(1)}-\sum a\cdot b_{(2)}\otimes b_{(1)}-\sum (b\cdot a_{(1)}\otimes a_{(2)}+b\cdot a_{(2)}\otimes a_{(1)})\\
&&\qquad+\sum b_{[1]}\otimes b_{[2]}\circ a-\sum a_{[1]}\otimes b\circ a_{[2]}\\
&=&\sum b_{(2)}\otimes a\cdot b_{(1)}-\sum a_{[2]}\star b\otimes a_{[1]}-\sum a\cdot b_{(2)}\otimes b_{(1)}\\
&&\quad-\sum (b\cdot a_{(1)}\otimes a_{(2)}+b\cdot a_{(2)}\otimes a_{(1)})+\sum b_{[1]}\otimes b_{[2]}\circ a-\sum a_{[1]}\otimes b\circ a_{[2]}\\
&=& \sum b_{(2)}\otimes a\cdot b_{(1)}-\sum a_{[2]}\star b\otimes a_{[1]}+\sum a\circ b_{[1]}\otimes b_{[2]}-\sum a_{(2)}\otimes b\cdot a_{(1)}+\sum b_{[2]}\circ a\otimes b_{[1]}\\
&=&\sum b_{(2)}\otimes a\cdot b_{(1)}-\sum a_{[2]}\star b\otimes a_{[1]}+\sum a \star b_{[2]}\otimes b_{[1]}-\sum a_{(2)}\otimes b\cdot a_{(1)}\\
&=& 0,
\end{eqnarray*}}
for all $a,b\in A$. Therefore, Eqs. (\ref{DNPB1}), (\ref{DNPB2})
and (\ref{DNPB4}) mean that Eq. (\ref{condd2}) holds. Then the
proof is completed.
\end{proof}

By Theorem \ref{corr-bialg}, we obtain the following conclusion
directly.
\begin{cor}\label{constr-PC}
Let $(P={\bf k}[\partial]\otimes_{\bf k} A, [\cdot_\lambda \cdot], \cdot_\lambda
\cdot)$ be the Poisson conformal algebra corresponding to a
differential Novikov-Poisson algebra $(A,\circ, \cdot)$. Suppose that
 $\Delta_1$ and $\Delta_2$ are ${\bf k}[\partial]$-module homomorphisms from $P$ to $P\otimes P$ such that $\Delta_i(A)\subseteq A\otimes A$ for $i=1$, $2$.
Define two ${\bf k}[\partial]$-module homomorphisms $\delta$ and
$\Delta$ first by taking
\begin{eqnarray}\label{cobracket1}
\delta(a):=(\partial \otimes I)\Delta_1(a)-\tau (\partial \otimes I)\Delta_1(a),\;\;\Delta(a):=\Delta_2(a),\;\;a\in A,
\end{eqnarray}
and then extending to $P$. Then $(P={\bf k}[\partial]\otimes_{\bf k} A, [\cdot_\lambda \cdot], \cdot_\lambda \cdot, \delta, \Delta)$ is a Poisson conformal algebra if and only if
$(A, \circ, \cdot, \Delta_1, \Delta_2)$ is a differential Novikov-Poisson bialgebra.
\end{cor}

\delete{\begin{defi}
\cite{LLB}
Let $(A, \cdot, \Delta)$ be a commutative and cocommutative
ASI bialgebra, $D: A\rightarrow A$ be a
derivation on $(A, \cdot)$ and $Q:A\rightarrow A$ be a
coderivation on $(A, \Delta)$, i.e. $\Delta(Q(a))=(I\otimes Q+Q\otimes I)\Delta(a)$ for all $a\in A$. If $D$ and $Q$ satisfy
\begin{eqnarray}
&& \label{DASI1} Q(a\cdot b)=Q(a)\cdot b-a\cdot D(b),\\
&& \label{DASI2}\Delta(D(a))=(D\otimes I-I\otimes Q)\Delta(a),\;\;\; a, b\in A,
\end{eqnarray}
%conditions
%\begin{eqnarray}
%&\mlabel{eq:RSI1}Q(a)\cdot b=a\cdot D(b)+Q(a\cdot b),&\\
%&\mlabel{eq:RSI2}  (D\otimes \id-\id \otimes Q)\delta(a)=\delta D(a),& \;\;\;\;a, b\in A,
%\end{eqnarray}
then $(A,\cdot, \Delta, D, Q)$ is called a commutative and cocommutative {\bf
differential antisymmetric infinitesimal bialgebra} or simply a commutative and cocommutative  {\bf differential ASI bialgebra}.
\end{defi}

Finally, we present a construction of differential Novikov-Poisson bialgebras from commutative and cocommutative differential ASI algebras.
\begin{pro}\label{constr-DNPB}
Let $(A,\cdot, \Delta, D, Q)$ be a commutative and cocommutative differential ASI bialgebra. Define a binary operation $\circ_q$ on $A$ and a linear map $\Delta_q: A\rightarrow A\otimes A$ as follows.
\begin{eqnarray}
a\circ_q b:=a\cdot (D+qQ)b,\;\; \Delta_q(a):=(I\otimes (Q+qD))D(a),\;\;a, b\in A,
\end{eqnarray}
where $q\in {\bf k}$. %Then we have the following conclusions.
If
  \begin{eqnarray}
  a\cdot Q(b)=-a\cdot D(b),\;\; (I\otimes Q)\Delta=-(I\otimes D)\Delta,\;\;a, b\in A,
  \end{eqnarray}
  then $(A, \circ_q, \cdot, \Delta_q, \Delta)$ is a differential Novikov-Poisson bialgebra for any $q\in {\bf k}$. In particular, if $Q=-D$, then Eqs. (\ref{DASI1}) and (\ref{DASI2}) hold and hence $(A, \circ_q, \cdot, \Delta_q, \Delta)$ is a differential Novikov-Poisson bialgebra.
 % \item $(A, \circ_{-\frac{1}{2}}, \cdot, \Delta_{-\frac{1}{2}}, \Delta)$ is a differential Novikov-Poisson bialgebra.
%\end{enumerate}
\end{pro}
\begin{proof}
For all $a$, $b\in A$, we have
\begin{eqnarray*}
&&\Delta_q(a\cdot b)-(L_{A, \cdot}(a)\otimes I)\Delta_q(b)+(I\otimes L_{A,\star_q}(b))\Delta(a)\\
&=&(I\otimes (Q+qD))\Delta(a\cdot b)-(L_{A,\cdot}(a)\otimes I)(I\otimes (Q+qD))\Delta(b)\\
&&\quad+\sum a_{[1]}\otimes (b\cdot (D+qQ)(a_{[2]})+(D+qQ)(b)\cdot a_{[2]})\\
&=&  (I\otimes (Q+qD))(\sum a_{[1]}\otimes b\cdot a_{[2]}+\sum a\cdot b_{[1]}\otimes b_{[2]})\\
&&\quad -\sum (a\cdot b_{[1]}\otimes (Q+qD)(b_{[2]}))+\sum a_{[1]}\otimes (b\cdot (D+qQ)(a_{[2]})+(D+qQ)(b)\cdot a_{[2]})\\
&=& \sum a_{[1]}\otimes (Q+qD)(b\cdot a_{[2]})+\sum a_{[1]}\otimes (b\cdot (D+qQ)(a_{[2]})+(D+qQ)(b)\cdot a_{[2]})\\
&=& \sum a_{[1]}\otimes (b\cdot (D+Q)(a_{[2]})+qb\cdot (D+Q)(a_{[2]})+q(D+Q)(b)\cdot a_{[2]})\\
&=& 0.
\end{eqnarray*}
Similarly, one can check that Eqs. (\ref{DNPC-1}), (\ref{DNPC-2}), (\ref{DNPB2})-(\ref{DNPB4}) hold. Then the conclusion follows.
\end{proof}
\yy{If my calculations are correct, this proposition does not hold when $q=-\frac{1}{2}$ without no extra conditions. When we consider Eq. (\ref{DNPB4}), we obtain
\begin{eqnarray*}
&&(L_{A,\circ_q}(a)\otimes I)\Delta(b)+(I\otimes L_{A, \circ_q}(b))\Delta(a)+(I\otimes L_{A,\cdot}(b))\Delta_q(a)+(L_{A,\cdot}(a)\otimes I)\tau \Delta_q(b)\\
&=&\frac{1}{2}(\sum a\cdot (D+Q)b_{[1]}\otimes b_{[2]}+\sum a_{[1]}\otimes b\cdot (D+Q)a_{[2]})\\
&=&\frac{1}{2}\Delta((D+Q)(b)\cdot a).
\end{eqnarray*}}}

At the end of this subsection,  we give an example of Poisson
conformal bialgebras constructed from differential Novikov-Poisson
bialgebras.
\begin{ex}\mlabel{conex1}
Let $({\bf k}[x],\cdot)$ be the polynomial algebra. Define a
linear map $\Delta_2: {\bf k}[x]\rightarrow {\bf k}[x]\otimes {\bf
k}[x]$ as follows.
\begin{eqnarray*}
\Delta_2(1)=0,\;\;\Delta_2(x^n)=x^{n-1}\otimes 1+x^{n-2}\otimes x+\cdots+x\otimes x^{n-2}+1\otimes x^{n-1}, \;\;\;\; n\geq 1.
\end{eqnarray*}
Let $q\in {\bf k}$. By \cite[Example 2.27]{HBG2},  $({\bf k}[x], \cdot, \Delta_2)$ is a commutative and cocommutative ASI bialgebra and $({\bf k}[x], \circ, \Delta_1)$ is a Novikov bialgebra given by
\begin{eqnarray*}
&&x^m\circ x^n=(1-q)nx^{m+n-1},\;\;\Delta_1(1)=0, \\
&&\Delta_1(x)=0,\;\;\Delta_1(x^n)=(q-1)\sum_{i=1}^{n-1}ix^{n-1-i}\otimes x^{i-1}, \;\;\; n\geq 2.
\end{eqnarray*}
It is straightforward to show that $({\bf k}[x], \circ, \cdot,
\Delta_1, \Delta_2)$ is a differential Novikov-Poisson bialgebra
for any $q\in {\bf k}$. Then by Corollary \ref{constr-PC}, there
is a Poisson conformal bialgebra $({\bf k}[\partial]\otimes_{\bf k} {\bf k}[x],
[\cdot_\lambda \cdot]$, $\cdot_\lambda \cdot$, $\delta$, $\Delta)$
given by
\begin{eqnarray*}
&&[{x^m}_\lambda x^n]=(m\partial+(m+n)\lambda)x^{m+n-1},\;\;{x^m}_\lambda x^n=x^{m+n},\\
&&\delta (1)=\delta(x)=0,\;\;\delta(x^n)=-\sum_{i=1}^{n-1} i(\partial x^{n-1-i}\otimes x^{i-1}-x^{i-1}\otimes \partial x^{n-1-i}), \;\; n\geq 2,\\
&&\Delta(1)=0,\;\;\Delta(x^n)=x^{n-1}\otimes 1+x^{n-2}\otimes
x+\cdots+x\otimes x^{n-2}+1\otimes x^{n-1}, \;\;\;\; n\geq 1,
\end{eqnarray*}
 for all $m$, $n\in \mathbb{Z}$.
%\cm{$\Delta(x)=0$ since $\delta(x)=1\otimes 1$?}
%\yy{Yes, $\Delta(x)=0$.}
\end{ex}
\delete{
\subsection{A class of PGD-bialgebras and coboundary Poisson conformal bialgebras}
Let $(A,  \bullet)$ be a vector space with a binary operation $\bullet$ and $r=\sum_i x_i\otimes y_i\in A\otimes A$.
Define
\begin{eqnarray*}
&&r_{13}\bullet r_{23}=\sum_{i, j}x_i\otimes x_j\otimes y_i \bullet y_j,
~~r_{12}\bullet r_{23}=\sum_{i, j} x_i \otimes y_i\bullet x_j\otimes y_j,\\
&&r_{13}\bullet r_{12}=\sum_{i, j}x_i\bullet x_j \otimes y_j \otimes y_i.
\end{eqnarray*}

\begin{lem}\label{dnp-co1}
Let $(A, \circ, \cdot)$ be a differential Novikov-Poisson algebra and $r\in A\otimes A$. Let $\Delta_{1,r}$ and $\Delta_{2,r}: A\rightarrow A\otimes A$ be the linear maps defined by
\begin{eqnarray}
\label{cobnov}&&\Delta_{1,r}(a):=-(L_{A,\circ}(a)\otimes \id+\id\otimes L_{A,\star}(a))r,\\
\mlabel{cobass}
&&\Delta_{2,r}(a)=(I\otimes L_{A,\cdot}(a)-L_{A,\cdot}(a)\otimes I)
r,\;\;a\in A.
\end{eqnarray}
Then Eqs. (\ref{DNPB1}) naturally holds and Eqs. (\ref{DNPB2}) and (\ref{DNPB4}) hold if and only if
\begin{eqnarray}
&&\label{coeq1}-(I\otimes L_{A,\cdot}(a)L_{A,\circ}(b))(r+\tau r)+(L_{A,\star}(b)\otimes L_{A,\cdot}(a))(r+\tau r))=0,\\
&&\label{coeq2}(I\otimes L_{A,\cdot}(a)L_{A,\star}(b))(r+\tau r)+(L_{A,\cdot}(a)\otimes L_{A,\circ}(b))(r+\tau r)=0,\;\;a, b\in A.
\end{eqnarray}
\end{lem}
\begin{proof}
Let $a$, $b\in A$. We get
\begin{eqnarray*}
&&\Delta_{1,r}(a\cdot b)-(L_{A,\cdot}(a)\otimes I)\Delta_{1,r}(b)+(I\otimes L_{A,\star}(b))\Delta_{2,r}(a)\\
&=& -\sum_i((a\cdot b)\circ x_i\otimes y_i+x_i\otimes (a\cdot b)\star y_i-a\cdot (b\circ x_i)\otimes y_i-a\cdot x_i\otimes b\star y_i\\
&&\qquad-x_i\otimes b\star (a\cdot y_i)+a\cdot x_i\otimes b\star y_i)\\
&=& -\sum x_i\otimes ((a\cdot b)\star y_i-b\star (a\cdot y_i)).
\end{eqnarray*}
Notice that
\begin{eqnarray*}
b \star (a\cdot y_i)-(a\cdot b)\star y_i&=&b\circ (a\cdot y_i)+(a\cdot y_i)\circ b-(a\cdot b)\circ y_i-y_i\circ (a\cdot b)\\
&=&b\circ(a\cdot y_i)+(y_i \circ b)\cdot a-(b\circ y_i)\cdot a+b\circ (a\cdot y_i)\\
&=& 0.
\end{eqnarray*}
Therefore, Eq. (\ref{DNPB1}) naturally holds. Similarly, we obtain
\begin{eqnarray*}
&&\Delta_{2,r}(a\circ b)-(R_{A,\circ}(b)\otimes I)\Delta_{2,r}(a)+(I\otimes L_{A,\cdot}(a))(\Delta_{1,r}(b)+\tau \Delta_{1,r}(b))\\
&=&\sum_i(x_i\otimes (a\circ b)\cdot y_i-(a\circ b)\cdot x_i\otimes y_i-x_i\circ b\otimes a\cdot y_i+(a\cdot x_i)\circ b\otimes y_i\\
&&\qquad-b\circ x_i\otimes a\cdot y_i-x_i\otimes a\cdot (b\star y_i)-y_i\otimes a\cdot (b\circ x_i)-b\star y_i\otimes a\cdot x_i)\\
&=&\sum_i(x_i\otimes ((a\circ b)\cdot y_i-a\cdot (b\star y_i))+((a\cdot x_i)\circ b-(a\circ b)\cdot x_i)\otimes y_i\\
&&\qquad-y_i\otimes a\cdot (b\circ x_i)+(L_{A,\star}(b)\otimes L_{A,\cdot}(a))(r+\tau r))\\
&=&\sum_i(x_i\otimes ((a\circ b)\cdot y_i-a\cdot (b\star y_i)+a\cdot (b\circ y_i)-(I\otimes L_{A,\cdot}(a)L_{A,\circ}(b))(r+\tau r)\\
&&\qquad\quad+(L_{A,\star}(b)\otimes L_{A,\cdot}(a))(r+\tau r))\\
&=&-(I\otimes L_{A,\cdot}(a)L_{A,\circ}(b))(r+\tau r)+(L_{A,\star}(b)\otimes L_{A,\cdot}(a))(r+\tau r)).
\end{eqnarray*}
Therefore, Eq. (\ref{DNPB2}) is equivalent to Eq. (\ref{coeq1}). One can similarly check that Eq. (\ref{DNPB4}) is equivalent to Eq. (\ref{coeq2}).
\end{proof}

\begin{lem}\label{dnp-co2}
Let $(A, \circ, \cdot)$ be a differential Novikov-Poisson algebra and $r\in A\otimes A$. Let $\Delta_{1,r}$ and $\Delta_{2,r}: A\rightarrow A\otimes A$ be the linear maps defined by Eqs. (\ref{cobnov}) and (\ref{cobass}). Then Eqs. (\ref{DNPC-1}) and (\ref{DNPC-2}) hold  if and only if for all $a\in A$, the following equalities hold
\begin{eqnarray}
&&\label{cocdnp1} -(L_{A,\circ}(a)\otimes I\otimes I)(r_{13}\cdot r_{23}-r_{12}\cdot r_{23}+r_{13}\cdot r_{12})\\
&&\quad+\sum_j (L_{A,\circ}(x_j)L_{A,\cdot}(a)\otimes I\otimes I)((r+\tau r)\otimes y_j)\nonumber\\
&&\quad +(I\otimes L_{A,\cdot}(a)\otimes I)(r_{12}\circ r_{23}+(\tau r)_{12}\circ r_{13}+r_{23}\star r_{13})\nonumber\\
&&\quad +(I\otimes I\otimes L_{A,\cdot}(a))(r_{13}\circ r_{23}+r_{13}\circ r_{12}+r_{23}\star r_{12})=0,\nonumber\\
&&\label{cocdnp2}-(L_{A,\star}(a)\otimes I\otimes I)((\tau r)_{13}\cdot r_{23}-(\tau r)_{12}\cdot r_{23}-r_{13}\cdot (\tau r)_{12})\\
&&\quad+(I\otimes I\otimes L_{A,\cdot}(a))(-r_{23}\circ (\tau r)_{13}+r_{23}\circ (\tau r)_{12}+r_{13}\star (\tau r)_{12})=0.\nonumber
\end{eqnarray}
\end{lem}
\begin{proof}
Let $a\in A$. We have
\begin{eqnarray*}
&& (I\otimes \Delta_{2,r})\Delta_{1,r}(a)-(\Delta_{1,r}\otimes I)\Delta_{2,r}(a)-(\tau\otimes I)(I\otimes \Delta_{1,r})\Delta_{2,r}(a)\\
&=&\sum_{i,j}(-a\circ x_i\otimes x_j\otimes y_i\cdot y_j+a\circ x_i\otimes y_i\cdot x_j\otimes y_j-x_i\otimes x_j\otimes (a\star y_i)\cdot y_j\\
&&\quad +x_i\otimes (a\star y_i)\cdot x_j\otimes y_j+x_i\circ x_j\otimes y_j\otimes a\cdot y_i+x_j\otimes x_i\star y_j\otimes a\cdot y_i\\
&&\quad -(a\cdot x_i)\circ x_j\otimes y_j\otimes y_i-x_j\otimes (a\cdot x_i)\star y_j\otimes y_i+(a\cdot y_i)\circ x_j\otimes x_i\otimes y_j\\
&&\quad+x_j\otimes x_i\otimes (a\cdot y_i)\star y_j-y_i\circ x_j\otimes a\cdot x_i\otimes y_j-x_j\otimes a\cdot x_i\otimes y_i\star y_j)\\
&=&\sum_{i,j}(-a\circ x_i\otimes x_j\otimes y_i\cdot y_j+a\circ x_i\otimes y_i\cdot x_j\otimes y_j-(a\cdot x_i)\circ x_j\otimes y_j\otimes y_i\\
&&\quad+(a\cdot y_i)\circ x_j\otimes x_i\otimes y_j)+\sum_{i,j}(x_i\otimes (a\star y_i)\cdot x_j\otimes y_j-x_j\otimes (a\cdot x_i)\star y_j\otimes y_i\\
&&\quad-y_i\circ x_j\otimes a\cdot x_i\otimes y_j-x_j\otimes a\cdot x_i\otimes y_i\star y_j)+\sum_{i,j}(-x_i\otimes x_j\otimes (a\star y_i)\cdot y_j\\
&&\quad+x_i\circ x_j\otimes y_j\otimes a\cdot y_i+x_j\otimes x_i\star y_j\otimes a\cdot y_i+x_j\otimes x_i\otimes (a\cdot y_i)\star y_j)\\
&=&-(L_{A,\circ}(a)\otimes I\otimes I)(r_{13}\cdot r_{23}-r_{12}\cdot r_{23}+r_{13}\cdot r_{12})\\
&&\quad+\sum_j (L_{A,\circ}(x_j)L_{A,\cdot}(a)\otimes I\otimes I)((r+\tau r)\otimes y_j)\\
&&\quad +(I\otimes L_{A,\cdot}(a)\otimes I)(r_{12}\circ r_{23}+(\tau r)_{12}\circ r_{13}+r_{23}\star r_{13})\\
&&\quad +(I\otimes I\otimes L_{A,\cdot}(a))(r_{13}\circ r_{23}+r_{13}\circ r_{12}+r_{23}\star r_{12}).
\end{eqnarray*}
Therefore,  Eq. (\ref{DNPC-1}) holds if and only if Eq. (\ref{cocdnp1}) holds. Similarly, one can check that Eq. (\ref{DNPC-2}) holds if and only if Eq. (\ref{cocdnp2}) holds.
\end{proof}

Let $(A, \circ, \cdot, [\cdot,\cdot])$ be a PGD-algebra. Then the
equation
\begin{eqnarray}
{\bf N}(r):= r_{13}\circ r_{23} +r_{12}\star r_{23}+r_{13}\circ
r_{12}=0
\end{eqnarray}
is called the {\bf Novikov Yang-Baxter equation (NYBE)} in $(A, \circ)$ (see \cite{HBG}), the following equation
\begin{equation}\mlabel{eq:AYBE}
{\bf A}(r):=r_{13}\cdot r_{12}+r_{13}\cdot r_{23}-r_{12}\cdot r_{23}=0
\end{equation}
is called {\bf associative Yang-Baxter equation (AYBE) in $(A, \cdot)$} (see \cite{Bai1}) and the following equation
\begin{eqnarray}
{\bf C}(r):=  -[r_{13}, r_{12}] +[r_{12}, r_{23}]+[r_{13},r_{23}]=0,
\end{eqnarray}
is called the {\bf classical Yang-Baxter equation (CYBE)} in $(A, [\cdot,\cdot])$ (see \cite{CP}).

\delete{\begin{pro}\cite[Corollary 3.12]{HBG} \mlabel{Nov-coboundary} If $r$ is a skew-symmetric solution of Novikov Yang-Baxter equation in a Novikov algebra $(A, \circ)$, then $(A, \circ, \Delta_{1,r})$ is a Novikov bialgebra with $\Delta_r$ defined by
\begin{eqnarray}\mlabel{cobnov}
\Delta_{1,r}(a):=-(L_{A,\circ}(a)\otimes \id+\id\otimes L_{A,\star}(a))r,\;\; a\in A.
\end{eqnarray}
\end{pro}}

\delete{Recall that a commutative and cocommutative antisymmetric
infinitesimal bialgebra $(A, \cdot, \Delta_2)$ is called {\bf
coboundary} (see \mcite{Bai1}) if there exists a $r\in A\otimes A$
such that
\begin{eqnarray}\mlabel{cobass}
\Delta_2(a)=(I\otimes L_{A,\cdot}(a)-L_{A,\cdot}(a)\otimes I)
r,\;\;a\in A.
\end{eqnarray}

A Lie bialgebra $(A, [\cdot,\cdot], \delta)$ is called {\bf
coboundary} (see \mcite{CP}) if there exists a $r\in A\otimes A$
such that
\begin{eqnarray}\mlabel{cobass-Lie}
\delta(a)=(\ad_A(a)\otimes I+I\otimes \ad_{A}(a))r,\;\;a\in A.
\end{eqnarray}}

\delete{\begin{pro}\label{ass-coboundary}\cite[Theorem 2.3.5]{Bai1}
 If $r$ is a skew-symmetric solution of AYBE in a commutative associative algebra $(A, \cdot)$, then $(A, \cdot, \Delta_{2,r})$ is a commutative and cocommutative antisymmetric
infinitesimal bialgebra with $\Delta_{2,r}$ defined by
\begin{eqnarray}\mlabel{cobass}
\Delta_{2,r}(a)=(I\otimes L_{A,\cdot}(a)-L_{A,\cdot}(a)\otimes I)
r,\;\;a\in A.
\end{eqnarray}
\end{pro}

\begin{pro}\label{Lie-coboundary}\cite{CP}
 If $r$ is a skew-symmetric solution of CYBE in a Lie algebra $(A, [\cdot,\cdot])$, then $(A, [\cdot,\cdot], \delta_r)$ is a Lie bialgebra with $\delta_r$ defined by \begin{eqnarray}\mlabel{coLie}
\delta_r(a)=(\ad_A(a)\otimes I+I\otimes \ad_{A}(a))r,\;\;a\in A.
\end{eqnarray}
\end{pro}}

\begin{defi}
Let $(A, \circ, \cdot)$ be a differential Novikov-Poisson algebra and $r\in A\otimes A$. The equation
\begin{eqnarray}
{\bf N}(r)={\bf A}(r)=0
\end{eqnarray}
is called the {\bf differential Novikov-Poisson Yang-Baxter equation (DNPYBE)} in $A$.

Let $(A, \circ, \cdot, [\cdot,\cdot])$ be a PGD-algebra and $r\in A\otimes A$. The equation
\begin{eqnarray}
{\bf N}(r)={\bf A}(r)={\bf C}(r)=0
\end{eqnarray}
is called the {\bf Poisson-Gel'fand-Dorfman Yang-Baxter equation (PGDYBE)} in $A$.
\end{defi}

\begin{pro}\label{cob-GD}\cite[Proposition 4.5]{HBG1}
Let $(A, \circ, [\cdot,\cdot])$ be a GD-algebra and $r\in A\otimes A$ be a skew-symmetric solution of both NYBE and CYBE in $A$. Define $\Delta_{1,r}$ and $\delta_r$  by Eq. (\ref{cobnov}) and
\begin{eqnarray}
\mlabel{coLie}
&&\delta_r(a)=(\ad_A(a)\otimes I+I\otimes \ad_{A}(a))r,\;\;a\in A.
\end{eqnarray}
Then $(A, \circ, [\cdot,\cdot], \Delta_{1,r}, \delta_r)$ is a GD-bialgebra.
\end{pro}
\begin{pro}\label{cob-P}\cite[Theorem 2]{NB}
Let $(A, \cdot, [\cdot,\cdot])$ be a Poisson algebra and $r\in A\otimes A$ be a skew-symmetric solution of both AYBE and CYBE in $A$. Define $\delta_r$ and $\Delta_{2,r}$ by Eqs. (\ref{coLie}) and (\ref{cobass}).
Then $(A, \cdot, [\cdot,\cdot], \Delta_{2,r}, \delta_r)$ is a Poisson bialgebra.
\end{pro}

\begin{thm}\label{cob-PGD}
Let $(A, \circ, \cdot, [\cdot,\cdot])$ be a PGD-algebra. If $r\in A\otimes A$ is a skew-symmetric solution of PGDYBE in $A$, then $(A, \circ, \cdot, [\cdot,\cdot], \Delta_{1,r}, \Delta_{2,r}, \delta_r)$ is a PGD-bialgebra where  $\Delta_{1,r}$, $\Delta_{2,r}$ and $\delta_r: A\rightarrow A\otimes A$ be the linear maps defined by Eqs. (\ref{cobnov}), (\ref{cobass}) and (\ref{coLie}).
\end{thm}
\begin{proof}
It is straightforward by Propositions \ref{cob-GD}, \ref{cob-P},  Lemmas \ref{dnp-co1} and \ref{dnp-co2}.
\end{proof}

The following corollary is straightforward.
\begin{cor}\label{cob-DNPB}
Let $(A, \circ, \cdot)$ be a differential Novikov-Poisson algebra. If $r\in A\otimes A$ is a skew-symmetric solution of the DNPYBE in $A$, then $(A, \circ, \cdot, \Delta_{1,r}, \Delta_{2,r})$ is a differential Novikov-Poisson bialgebra where  $\Delta_{1,r}$ and $\Delta_{2,r}: A\rightarrow A\otimes A$ be the linear maps defined by Eqs. (\ref{cobnov}) and (\ref{cobass}).
\end{cor}

Finally, we investigate the relationship between PGDYBE and PCYBE is given as follows.
\begin{pro}\label{conformal Yang-Baxter equation}
Let $(A, \circ, \cdot, [\cdot,\cdot])$ be a PGD-algebra and $r=\sum_i x_i\otimes y_i
\in A\otimes A$ be skew-symmetric. Let $P={\bf k}[\partial]A$
be the Poisson conformal algebra corresponding to $(A, \circ, \cdot, [\cdot,\cdot])$. Then
$r\in P\otimes P$ is a solution of the PCYBE in
$P$ if and only if $r$ is a solution of the PGDYBE in $A$.
\end{pro}
\begin{proof}
By \cite[Proposition 4.17]{HBG1}, $r\in P\otimes P$ is a solution of the CCYBE in $P$ if and only if  $r$ is a solution of both NYBE and CYBE in $A$. It is obvious that $r\in P\otimes P$ is a solution of $r\bullet r\equiv 0 \;\; \text{mod}\;\; (\partial^{\otimes^3})$ in $P$ if and only if  $r$ is a solution of the AYBE in $A$. Then the conclusion follows.
\end{proof}

\begin{cor}\label{constr-cob}
Let $r\in A\otimes A$ be a skew-symmetric solution of the
PGDYBE in a PGD-algebra $(A, \circ, \cdot, [\cdot,\cdot])$ and
$(A, \circ, \cdot, [\cdot,\cdot], \Delta_{1,r}, \Delta_{2,r}, \delta_r)$ be the PGD-bialgebra given in Theorem \ref{cob-PGD} where  $\Delta_{1,r}$ , $\Delta_{2,r}$ and $\delta_r: A\rightarrow A\otimes A$ be the linear maps defined by Eqs. (\ref{cobnov}), (\ref{cobass}) and (\ref{coLie}). Then the coproducts $\delta$ and $\Delta$ defined by
Eq. (\ref{corr-co}) really endows the Poisson conformal algebra
$P={\bf k}[\partial]A$ corresponding to $(A, \circ, \cdot, [\cdot,\cdot])$ a
coboundary Poisson conformal bialgebra associated to $r$,
i.e. $\delta(a)=({\ad_P(a)}_\lambda\otimes I+I\otimes {\ad_P(a)}_\lambda)
r|_{\lambda=-\partial^{\otimes^2}}$ and $\Delta(a)=(I\otimes {L_P(a)}_\lambda-{L_P(a)}_\lambda \otimes I)r|_{\lambda=-\partial^{\otimes^2}}$ for all $a\in R$.
\end{cor}
\begin{proof}
By \cite[Corollary 4.18]{HBG1}, $\delta(a)=({\text{ad}_P(a)}_\lambda\otimes I+I\otimes {\text{ad}_P(a)}_\lambda)
r|_{\lambda=-\partial^{\otimes^2}}=(\partial\otimes I)\Delta_{1,r}(a)-\tau (\partial\otimes I)\Delta_{1,r}(a)+\delta_r(a)$ for all $a\in A$. It is obvious that
$\Delta(a)=(I\otimes {L_P(a)}_\lambda-{L_P(a)}_\lambda \otimes I)r|_{\lambda=-\partial^{\otimes^2}}=\Delta_{2,r}(a)$ for all $a\in A$. Since $\delta$ and $\Delta$ are ${\bf k}[\partial]$-module homomorphisms, then this conclusion follows.
\end{proof}

\subsection{Representation and the operator form of PGDYBE}

Recall \mcite{O3} that
a {\bf representation} of a Novikov algebra $(A,\circ)$ is a triple $(V, l_A,r_A)$, where $V$ is a vector space and   $l_A$, $r_A: A\rightarrow {\rm
End}_{\bf k}(V)$ are linear maps satisfying
\begin{eqnarray}
\mlabel{lef-mod1}
&&l_A(a\circ b-b\circ a)v=l_A(a)l_A(b)v-l_A(b)l_A(a)v,\\
\mlabel{lef-mod2}
&&l_A(a)r_A(b)v-r_A(b)l_A(a)v=r_A(a\circ b)v-r_A(b)r_A(a)v,\\
\mlabel{Nov-mod1}\
&&l_A(a\circ b)v=r_A(b)l_A(a)v,
\\
\mlabel{Nov-mod2}
&&r_A(a)r_A(b)v=r_A(b)r_A(a)v, \;\;
\;a, b\in A, v\in V.
\end{eqnarray}

Let $(A, \circ, [\cdot, \cdot])$ be a GD-algebra and $V$ be a vector space. Let $\rho_A$, $l_A$, and $r_A$: $A \rightarrow {\rm
        End}_{\bf k}(V)$ be three linear maps. Then $(V, l_A, r_A, \rho_A)$ is called a {\bf representation} of $(A, \circ, [\cdot, \cdot])$ if $(V,\rho_A)$ is a representation of the Lie algebra $(A, [\cdot,\cdot])$, $(V, l_A, r_A)$ is a representation of the Novikov algebra $(A, \circ)$, and the following conditions are satisfied
    \begin{eqnarray}
        &&\mlabel{GF-rep-1}\;\;\rho_A(a)l_A(b)v+\rho_A(b\circ a)v+l_A([b,a])v-r_A(a)\rho_A(b)v-l_A(b)\rho_A(a) v=0,\\
        &&\mlabel{GF-rep-2}\;\;\rho_A(a)r_A(b)v-\rho_A(b)r_A(a)v-r_A(b)\rho_A(a) v+r_A(a)\rho_A(b) v-r_A([a,b])v =0,\;a, b\in A,\;v\in V.
    \end{eqnarray}

Let $(A, [\cdot,\cdot], \cdot)$ be a Poisson algebra and $V$ be a vector space. Let $\rho_A$ and $\sigma_A$: $A \rightarrow {\rm
        End}_{\bf k}(V)$ be linear maps. Then $(V, \rho_A, \sigma_A)$ is called a {\bf representation} of $(A, [\cdot, \cdot], \cdot)$ if $(V,\rho_A)$ is a representation of the Lie algebra $(A, [\cdot,\cdot])$, $(V, \sigma_A)$ is a representation of the commutative associative algebra $(A, \cdot)$, and the following conditions are satisfied
    \begin{eqnarray}
        &&\mlabel{P-rep-1}\;\;\rho_A(a\cdot b)v=\sigma_A(b)(\rho_A(a)v)+\sigma_A(a)(\rho_A(b)v),\\
        &&\mlabel{GF-rep-2}\;\;\rho_A(a)(\sigma_A(b)v)=\sigma_A([a,b])v+\sigma_A(b)(\rho_A(a)v),\;\;\;a, b\in A,\; v\in V.
    \end{eqnarray}

Next, we introduce the definition of representations of differential Novikov-Poisson algebras and PGD-algebras.
\begin{defi}
Let $(A, \circ, \cdot)$ be a differential Novikov-Poisson algebra and $V$ be a vector space. Let $\sigma_A$, $l_A$, and $r_A$: $A \rightarrow {\rm End}_{\bf k}(V)$ be three linear maps. Then $(V, l_A, r_A, \sigma_A)$ is called a {\bf representation} of $(A, \circ, \cdot)$ if $(V,\sigma_A)$ is a representation of the commutative associative algebra $(A, \cdot)$, $(V, l_A, r_A)$ is a representation of the Novikov algebra $(A, \circ)$, and
\begin{eqnarray}
&&\label{DNP-rep-1}l_A(a\cdot b)v=\sigma_A(a)l_A(b)v,\;\; r_A(b)\sigma_A(a)v=\sigma_A(a)r_A(b)v=\sigma_A(a\circ b)v,\\
&&\label{DNP-rep-2}l_A(a)\sigma_A(b)v=\sigma_A(a\circ b)v+\sigma_A(b)l_A(a)v,\\
&&\label{DNP-rep-2}r_A(a\cdot b)v=\sigma_A(b)r_A(a)v+\sigma_A(a)r_A(b)v,\;\;\;a, b\in A,\;\; v\in V.
\end{eqnarray}

Let $(A, \circ, \cdot, [\cdot,\cdot])$ be a PGD-algebra and $V$ be a vector space. $(V, l_A, r_A, \sigma_A, \rho_A)$ is called a {\bf representation} of $(A, \circ, \cdot, [\cdot,\cdot])$ if $(V, l_A, r_A, \sigma_A)$ is a representation of the differential Novikov-Poisson algebra $(A, \circ, \cdot)$,  $(V, l_A, r_A, \rho_A)$ is a representation of the GD-algebra $(A, \circ, \cdot)$, and $(V, \rho_A, \sigma_A)$ is a representation of the Poisson algebra $(A, [\cdot,\cdot], \cdot)$.
\end{defi}

\begin{rmk}\label{ad-resp}
Obviously, $(A, L_{A,\circ}, R_{A,\circ}, L_{A,\cdot}, \ad_A)$ is a representation of $(A, \circ, \cdot, [\cdot,\cdot])$, which is called the {\bf adjoint representation} of  $(A, \circ, \cdot, [\cdot,\cdot])$.
\end{rmk}

\begin{pro}
Let $(A, \circ, \cdot, [\cdot,\cdot])$ be a PGD-algebra and $V$ be a vector space. Let $l_A$, $r_A$, $\rho_A$, $\sigma_A: A\rightarrow {\rm
End}_{\bf k}(V)$ be four linear maps. Define two binary operations on the direct sum $A\oplus V$ as vector spaces as follows
\begin{eqnarray}
&&(a+u)\circ (b+v):=a\circ b+l_A(a)v+r_A(b)u,\\
&& (a+u)\cdot (b+v):=a\cdot b+\sigma_A(a)v+\sigma_A(b)u,\\
&&[a+u,b+v]:=[a, b]+\rho_A(a)v-\rho_A(b)u,\;\;a, b\in A, u, v\in V.
\end{eqnarray}
Then $(A\oplus V, \circ, \cdot, [\cdot,\cdot])$ is a PGD-algebra if and only if $(V, l_A, r_A, \sigma_A,\rho_A)$ is a representation of $(A, \circ, \cdot, [\cdot,\cdot])$. Denote this PGD-algebra by $A\ltimes_{l_A,r_A}^{\sigma_A,\rho_A} V$, which is called {\bf the semi-direct product} of $(A, \circ, \cdot, [\cdot,\cdot])$ and its representation $(V, l_A, r_A, \sigma_A,\rho_A)$.
\end{pro}
\begin{proof}
It is straightforward.
\end{proof}

Let $(A, \circ, \cdot, [\cdot,\cdot])$ be a PGD-algebra and $V$ be a vector space. Let $\varphi:A \rightarrow {\rm
End}_{\bf k}(V)$ be a linear map. Define a liner map $\varphi^\ast:A \rightarrow {\rm
End}_{\bf k}(V^\ast)$ by
\begin{eqnarray*}
\langle \varphi^\ast(a)f, u\rangle=-\langle f, \varphi(a)u\rangle,\;\;\; a\in A, f\in V^\ast, u\in V.
\end{eqnarray*}

\begin{pro}\label{dual-resp}
Let $(A, \circ, \cdot, [\cdot,\cdot])$ be a PGD-algebra and $(V, l_A, r_A, \sigma_A, \rho_A)$ be a representation of $(A, \circ, \cdot, [\cdot,\cdot)$.
Then $(V^\ast, l_A^\ast+r_A^\ast, -r_A^\ast, -\sigma_A^\ast, \rho_A^\ast)$ is a representation of $(A, \circ, \cdot, [\cdot,\cdot])$.
\end{pro}
\begin{proof}
It is straightforward to check.
\delete{By \cite[Proposition 3.3]{HBG}, $(V^\ast, l_A^\ast+r_A^\ast, -r_A^\ast)$ is a representation of $(A, \circ)$. Moreover, $(V^\ast, -\rho_A^\ast)$ is a representation of $(A, \cdot)$. Let $a$, $b\in A$, $f\in V^\ast$ and $v\in V$. Then we obtain
\begin{eqnarray*}
&&\langle (l_A^\ast+r_A^\ast)(a\cdot b)f+\rho_A^\ast(a)(l_A^\ast+r_A^\ast)(b)f,v\rangle\\
&=& \langle f, -l_A(a\cdot b)v-r_A(a\cdot b)v+l_A(b)\rho_A(a)v+r_A(b)\rho_A(a)v\rangle\\
&=& \langle f, -l_A(a\cdot b)v-\rho_A(a)r_A(b)v-\rho_A(b)r_A(a)v+l_A(b)\rho_A(a)v+r_A(b)\rho_A(a)v\rangle\\
&=&  \langle f, -l_A(a\cdot b)v-\rho_A(a)r_A(b)v+l_A(b)\rho_A(a)v\rangle\\
&=&\langle f, -l_A(a\cdot b)v-\rho_A(b\circ a)v+l_A(b)\rho_A(a)v\rangle\\
&=& langle f, -l_A(a\cdot b)v+\rho_A(a)l_A(b)v\rangle\\
&=& 0.
\end{eqnarray*}
Therefore, we get $(l_A^\ast+r_A^\ast)(a\cdot b)f=-\rho_A^\ast(a)(l_A^\ast+r_A^\ast)(b)f$. Similarly, other equalities can be checked.}
\end{proof}

\begin{rmk}
By Remark \ref{ad-resp} and Proposition \ref{dual-resp}, $(A^\ast, L_{A,\circ}^\ast+R_{A,\circ}^\ast, -R_{A,\circ}^\ast, -L_{A,\cdot}^\ast, \ad_A^\ast)$ is a representation of $(A, \circ, \cdot, [\cdot,\cdot])$.
\end{rmk}

\begin{cor}
Let $(A, \circ, \cdot)$ be a differential Novikov-Poisson algebra and $(V, l_A, r_A, \sigma_A)$ be a representation of $(A, \circ, \cdot)$.
Then $(V^\ast, l_A^\ast+r_A^\ast, -r_A^\ast, -\sigma_A^\ast)$ is a representation of $(A, \circ, \cdot)$.
\end{cor}

Next, we investigate the operator form of the PGDYBE.

\begin{defi}
Let $(A, \circ, \cdot, [\cdot,\cdot])$ be a PGD-algebra and $(V, l_A, r_A, \rho_A)$ be its representation. A linear map
$T: V\rightarrow A$ is called an {\bf $\mathcal{O}$-operator} on $(A, \circ, \cdot, [\cdot,\cdot])$ associated to $(V, l_A, r_A, \rho_A)$ if $T$ satisfies
\begin{eqnarray}
&&\label{operr3}T(u)\circ T(v)=T(l_A(T(u))v+r_A(T(v))u),\\
&&\label{operr4}T(u)\cdot T(v)=T(\rho_A(T(u))v+\rho_A(T(v))u),\\
&&\label{operr5}[T(u),T(v)]=T(\rho_A(T(u))v-\rho_A(T(v))u),\;\;u, v\in V.
\end{eqnarray}
\end{defi}

For a finite-dimensional vector space $A$, the isomorphism
$$A\otimes A\cong
\text{Hom}_{\bf k}(A^*,{\bf k})\otimes A\cong \text{Hom}_{\bf k}(A^*,A)
$$
identifies an $r\in A\otimes
A$ with a map from $A^*$ to $A$ which we denote by $T^r$.
Explicitly, writing $r=\sum_{i}x_i\otimes y_i$, then
\begin{equation}\label{eq:4.12} \notag
T^r:A^*\to A, \quad T^r(f)=\sum_{i}\langle f, x_i\rangle y_i,\;\; f\in A^*.
\end{equation}

\begin{pro}\label{oper-form1}
Let $(A, \circ, \cdot, [\cdot,\cdot])$ be a PGD-algebra and $r\in A\otimes A$ be skew-symmetric. Then $r$ is a solution of the PGDYBE in $(A, \circ, \cdot, [\cdot,\cdot])$ if and only if $T^r$ is an $\mathcal{O}$-operator on $(A, \circ, \cdot)$ associated to $(A^\ast, L_{A,\star}^\ast, -R_{A,\circ}^\ast, -L_{A,\cdot}^\ast, \ad_A^\ast)$.
\end{pro}
\begin{proof}
Note that $r$ is a solution of the AYBE in $(A, \cdot)$ if and only if Eq. (\ref{operr4}) holds by \cite[Theorem 2.4.7]{Bai1} and $r$ is a solution of the CYBE in $(A, [\cdot,\cdot])$ if and only if Eq. (\ref{operr5}) holds by \cite{Kuper}. By \cite[Theorem 3.27]{HBG}, $r$ is a solution of the NYBE in $(A, \circ)$ if and only if Eq. (\ref{operr3}) holds. Then this proposition follows.
\end{proof}

\delete{Proposition \ref{oper-form1} motivates us to give the following definition.

\begin{defi}
Let $(A, \circ, \cdot)$ be a differential Novikov-Poisson algebra and $(V, l_A, r_A, \rho_A)$ be its representation. A linear map
$T: V\rightarrow A$ is called an {\bf $\mathcal{O}$-operator} on $(A, \circ, \cdot)$ associated to $(V, l_A, r_A, \rho_A)$ if $T$ satisfies
\begin{eqnarray}
&&\label{operr3}T(u)\circ T(v)=T(l_A(T(u))v+r_A(T(v))u),\\
&&\label{operr4}T(u)\cdot T(v)=T(\rho_A(T(u))v+\rho_A(T(v))u),\;\;u, v\in V.
\end{eqnarray}
\end{defi}

\begin{rmk}
By Proposition \ref{oper-form1}, if $r$ is skew-symmetric, $r$ is a solution of the PGDYBE in $(A, \circ, \cdot)$ if and only if $T^r$ is an $\mathcal{O}$-operator on $(A, \circ, \cdot)$ associated to $(A^\ast, L_{A,\star}^\ast, -R_{A,\circ}^\ast, -L_{A,\cdot}^\ast)$.
\end{rmk}}

\begin{thm}\label{oper-form2}
Let $(A, \circ, \cdot, [\cdot,\cdot])$ be a PGD-algebra and $(V, l_A, r_A, \sigma_A, \rho_A)$ be its representation. Let $T: V\rightarrow A$ be a linear map which is
identified with $r_T\in A\otimes V^\ast \subseteq
(A\ltimes_{l_A^\ast+r_A^\ast,-r_A^\ast}^{ -\sigma_A^\ast, \rho_A^\ast}V^\ast) \otimes
(A\ltimes_{l_A^\ast+r_A^\ast,-r_A^\ast}^{ -\sigma_A^\ast, \rho_A^\ast} V^\ast)$ through
 ${\rm Hom}_{\bf k}(V, A)\cong A\otimes V^\ast$.
Then $r=r_T-\tau r_T$ is a solution of the PGDYBE in the PGD-algebra
$(A\ltimes_{l_A^\ast+r_A^\ast,-r_A^\ast}^{ -\sigma_A^\ast, \rho_A^\ast} V^\ast, \circ, \cdot, [\cdot,\cdot])$ if and only if $T$ is an
$\mathcal{O}$-operator on $(A,\circ, \cdot, [\cdot,\cdot])$ associated to $(V, l_A,
r_A, \sigma_A, \rho_A)$.
\end{thm}
\begin{proof}
It is straightforward by \cite{Bai}, \cite[Theorem 3.29]{HBG} and \cite[Theorem 2.5.5]{Bai1}.
\end{proof}
}
\subsection{Poisson conformal bialgebras from pre-PGD-algebras}
\begin{defi}\label{ND} \cite{HBG}
Let $A$ be a vector space with binary operations $ \lhd$ and
$\rhd$. If the following equalities are satisfied
\begin{align}
&a\rhd(b\rhd c)=(a\rhd b+a\lhd b)\rhd c+b\rhd(a\rhd c)-(b\rhd a+b\lhd a)\rhd c,\label{ND1}\\
&a\rhd(b\lhd c)=(a\rhd b)\lhd c+b\lhd(a\lhd c+a\rhd c)-(b\lhd a)\lhd c,\label{ND2}\\
&(a\lhd b+a\rhd b)\rhd c=(a\rhd c)\lhd b,\label{ND3}\\
&(a\lhd b)\lhd c=(a\lhd c)\lhd b,\label{ND4}
\end{align}
 for all $a$, $b$ and $c\in A$,
then $(A,\lhd,\rhd)$ is called a \bf{pre-Novikov algebra}.
\end{defi}

Recall \cite{Lo} that a {\bf Zinbiel algebra} $(A, \succ)$ is a
vector space $A$ with a binary operation $\succ: A\otimes
A\rightarrow A$ satisfying the following condition
\begin{eqnarray*}
a\succ (b\succ c)=(b\succ a+a\succ b)\succ c,\;\;\;a, b, c\in A.
\end{eqnarray*}

%\cm{whether it is necessary to unify the notations for a Zinbiel
%algebra to be $(A,\succ)$ and Zinbiel conformal algebra
%$(P,\succ_\lambda)$?}

Recall \cite{XH} that a {\bf pre-Gel'fand-Dorfman algebra
(pre-GD-algebra)} $(A,\lhd,\rhd,\diamond)$ is a vector space $A$
with three binary operations $ \lhd$, $\rhd$ and $\diamond$, where
$(A, \lhd, \rhd)$ is a pre-Novikov algebra, $(A, \diamond)$ is a
left-symmetric algebra and they satisfy the following
compatibility conditions {\small \begin{eqnarray*}
&&c\lhd(a\diamond b-b\diamond a)-a\diamond(c\lhd b)-(b\diamond c)\lhd a=-b\diamond(c\lhd a)-(a\diamond c)\lhd b,\mlabel{lnd1}\\
&&(a\diamond b-b\diamond a)\rhd c+(a\lhd b+a\rhd b)\diamond c
=a\rhd(b\diamond c)-b\diamond(a\rhd c)+(a\diamond c)\lhd b, \;\;a,
b, c\in A.\mlabel{lnd2}
\end{eqnarray*}}

 Recall \cite{A3} that a {\bf pre-Poisson algebra}
$(A,\diamond,\succ)$ is a vector space $A$ with two binary
operations $\diamond$ and $\succ$ such that $(A, \diamond)$ is a
left-symmetric algebra, $(A, \succ)$ is a Zinbiel algebra and the
following compatibility conditions hold.
\begin{eqnarray}
&&\label{PP-eq1}(a\diamond b-b\diamond a)\succ c=a\diamond (b\succ c)-b\succ (a\diamond c),\\
&&\label{PP-eq2}(a\succ b+b\succ a)\diamond  c=a\succ(b\diamond c)+b\succ (a\diamond c),\;\;a, b, c\in A.
\end{eqnarray}

%Next, we introduce the definitions of pre-differential Novikov-Poisson algebras and pre-PGD-algebras.
\begin{defi}
Let $A$ be a vector space with binary operations  $ \lhd$, $\rhd$
and $\succ$. If $(A, \lhd, \rhd)$ is a pre-Novikov algebra, $(A,
\succ)$ is a Zinbiel algebra and they satisfy the following
compatibility condition
\begin{eqnarray}
&&\label{PDNP-eq1}(a\succ b+b\succ a)\rhd c=a\succ (b\rhd c),\;\; (a\succ c)\lhd b=a\succ (c\lhd b)=(a\lhd b+a\rhd b)\succ c,\\
&&\label{PDNP-eq2}a\rhd (b\succ c)=(a\lhd b+a\rhd b)\succ c+b\succ (a\rhd c),\\
&& \label{PDNP-eq3}c\lhd (a\succ b+b\succ a)=b\succ (c\lhd a)+a\succ (c\lhd b),\;\;\; a, b, c\in A,
\end{eqnarray}
then $(A, \lhd, \rhd, \succ)$ is called a {\bf pre-differential Novikov-Poisson algebra}.
\end{defi}

\begin{defi} A 5-tuple
$(A, \lhd, \rhd, \succ, \diamond)$  is called a {\bf pre-Poisson-Gel'fand-Dorfman algebra (pre-PGD-algebra)} if $(A, \lhd, \rhd, \succ)$ is a pre-differential Novikov-Poisson algebra, $(A, \lhd, \rhd, \diamond)$ is a pre-GD-algebra and $(A, \diamond, \succ)$ is a pre-Poisson algebra.
\end{defi}

Obviously pre-GD-algebras, pre-Poisson algebras and
pre-differential Novikov-Poisson algebras are special
pre-PGD-algebras. Moreover, we give a construction of
pre-differential Novikov-Poisson algebras from Zinbiel algebras with a derivation as
follows.
\begin{pro}\label{constr-PDNP}
Let $(A,\succ)$ be a Zinbiel algebra with a derivation $D$. Define
\begin{eqnarray*}
a\lhd b:= D(b)\succ a, \;\; a\rhd b:= a\succ D(b),\;\; a, b\in A.
\end{eqnarray*}
Then  $(A, \lhd, \rhd, \succ)$ is a pre-differential Novikov-Poisson algebra.
\end{pro}
\begin{proof}
It is straightforward.
\end{proof}

There is a correspondence between pre-PGD-algebras and a class of pre-Poisson conformal algebras.
\begin{pro}\label{corresp-pre}
Let $A$ be a vector space with four binary operations $\lhd$,
$\rhd$, $\diamond$, $\succ$ and $P={\bf k}[\partial]\otimes_{\bf k} A$ be a free ${\bf k}[\partial]$-module. Define
\begin{eqnarray*}
a\circ_\lambda b:=\partial(b\lhd a)+\lambda (a\rhd b+b\lhd a)+a\diamond b,\;\; a\succ_\lambda b:=a\succ b,\;\;a, b\in A.
\end{eqnarray*}
%\cm{Maybe we rewrite the above as we have done in Proposition
%2.6?}
Then $(P, \circ_\lambda, \succ_\lambda)$ is a pre-Poisson
conformal algebra if and only if $(A, \lhd, \rhd, \succ, \diamond)$
is a pre-PGD-algebra. In this case, $(P, \circ_\lambda,
\succ_\lambda)$ is called the {\bf pre-Poisson conformal algebra
corresponding to $(A, \lhd, \rhd, \succ, \diamond)$.}
\end{pro}
\begin{proof}
By \cite[Theorem 2.19]{XH}, $(P, \circ_\lambda)$ is a
left-symmetric conformal algebra if and only if $(A, \lhd, \rhd,
\diamond)$ is a pre-GD-algebra. It is also obvious that $(P,
\succ_\lambda)$ is a Zinbiel conformal algebra if and only if $(A,
\succ)$ is a Zinbiel algebra. Let $a$, $b$ and $c\in A$. Then we
obtain
\begin{eqnarray*}
&&(a\circ_\lambda b-b\circ_{-\lambda-\partial}a)\succ_{\lambda+\mu}c-a\circ_\lambda(b\succ_\mu c)+b\succ_\mu (a\circ_\lambda c)\\
&=&(-\lambda-\mu)(b\lhd a+b\rhd a)\succ c+\lambda(a\rhd b+b\lhd a+b\rhd a+a\lhd b)\succ c\\
&&\quad+(a\diamond b-b\diamond a)\succ  c-\partial(b\succ c)\lhd a-\lambda(a\rhd (b\succ c)+(b\succ c)\lhd a)-a\diamond (b\succ c)\\
&&\quad +(\partial+\mu)b\succ (c\lhd a)+\lambda b\succ (a\rhd c+c\lhd a)+b\succ (a\diamond c).
\end{eqnarray*}
Comparing the coefficients of $\partial$, $\lambda$, $\mu$ and
$\lambda^0\mu^0\partial^0$, we show that Eq. (\ref{pre-P1}) holds
if and only if the following equalities hold
\begin{eqnarray}
&&\label{eq-q1}(b\succ c)\lhd a=b\succ (c\lhd a),\\
&&\label{eq-q2}(a\rhd b+a\lhd b)\succ c-a\rhd (b\succ c)-(b\succ c)\lhd a+b\succ (a\rhd c+c\lhd a)=0,\\
&&\label{eq-q3}(b\lhd a+b\rhd a)\succ c=b\succ (c\lhd a),\\
&&\label{eq-qq4} (a\diamond b-b\diamond a)\succ c-a\diamond (b\succ c)+b\succ (a\diamond c)=0.
\end{eqnarray}
Similarly, we show that  Eq. (\ref{pre-P2}) holds if and only if
the following equalities hold
\begin{eqnarray}
&&\label{eq-q4}c\lhd (a\succ b+b\succ a)=a\succ(c\lhd b)+b\succ (c\lhd a),\\
&&\label{eq-q5}(a\succ b+b\succ a)\rhd c+c\lhd (a\succ b+b\succ a)=a\succ (c\lhd b)+b\succ (a\rhd c+c\lhd a),\\
&&\label{eq-q6}(a\succ b+b\succ a)\rhd c+c\lhd (a\succ b+b\succ a)=a\succ (b\rhd c+c\lhd b)+b\succ (c\lhd a),\\
&&\label{eq-q7}(a\succ b+b\succ a)\diamond c=a\succ (b\diamond c)+b\succ (a\diamond c).
\end{eqnarray}
Then it is straightforward to show that Eqs.
(\ref{eq-q1})-(\ref{eq-q7}) hold if and only if Eqs.
(\ref{PP-eq1})-(\ref{PDNP-eq3}) hold. Thus the conclusion follows.
\end{proof}

\delete{Let $(A,\lhd,\rhd, \ast)$ be a pre-differential Novikov-Poisson algebra. Define linear maps $L_{\rhd}$, $R_{\lhd}$, $L_{\ast}: A\rightarrow \text{End}_{\bf k}(A)$ by
\begin{eqnarray*}
L_{\rhd}(a)(b):=a\rhd b, \;\;\; R_{\lhd}(a)(b):=b\lhd a,\;\;\;L_{\ast}(a)(b):=a\ast b, \;\;a, b\in A.
\end{eqnarray*}}

\begin{cor}
Let $A$ be a vector space with three binary operations $\lhd$,
$\rhd$, $\succ$ and $P={\bf k}[\partial]\otimes_{\bf k} A$ be a free ${\bf k}[\partial]$-module. Define
\begin{eqnarray*}
a\circ_\lambda b:=\partial(b\lhd a)+\lambda (a\rhd b+b\lhd a),\;\; a\succ_\lambda b:=a\succ b,\;\;a, b\in A.
\end{eqnarray*}
%\cm{Maybe we rewrite the above as we have done in Proposition
%2.6?}
Then $(P, \circ_\lambda, \succ_\lambda)$ is a pre-Poisson
conformal algebra if and only if $(A, \lhd, \rhd, \succ)$ is a
pre-differential Novikov-Poisson algebra.
\end{cor}
\begin{proof}
It follows from Proposition \ref{corresp-pre} immediately.
\end{proof}

\begin{pro}\label{PDNP-DNP}
\begin{enumerate}
\item Let $(A,\lhd,\rhd,  \succ, \diamond)$ be a pre-PGD-algebra. Define
\begin{eqnarray*}\label{GD1}
a\circ b:=a\lhd b+a\rhd b,\;\;a\cdot b:=a\succ b+b\succ a,\;\;[a,b]:=a\diamond b-b\diamond a,\;\;\;a, b\in A.
\end{eqnarray*}
Then $(A,\circ,\cdot, [\cdot,\cdot])$ is a PGD-algebra, which is called the {\bf associated PGD-algebra} of $(A, \lhd, \rhd, \succ, \diamond)$.

\item Let $(A,\lhd,\rhd,  \succ, \diamond)$ be a pre-PGD-algebra
and  $(P={\bf k}[\partial]\otimes_{\bf k} A, \circ_\lambda, \succ_\lambda)$ be the
pre-Poisson conformal algebra corresponding to $(A, \lhd, \rhd,
\succ, \diamond)$. Then the associated Poisson conformal algebra
$(P, [\cdot_\lambda \cdot], \cdot_\lambda\cdot)$ of $(P,
\circ_\lambda, \succ_\lambda)$ defined by Eq.~{\rm (\ref{eq:48})}
is exactly the one corresponding to the associated PGD-algebra
$(A,\circ,\cdot, [\cdot,\cdot])$ of $(A,\lhd,\rhd, \succ,
\diamond)$ defined by Eq.~{\rm (\ref{wq2})}.
\end{enumerate}

\delete{Moreover,  $(A,  L_{A,\rhd}, R_{A,\lhd}, L_{A,\ast}, L_{A,\diamond})$ is
a representation of the PGD-algebra $(A, \circ, \cdot, [\cdot,\cdot])$.Conversely, let $A$ be a vector space
with binary operations $\rhd$, $\lhd$ and $\ast$. If $(A, \circ, \cdot)$ defined by Eq. (\ref{GD1}) is a differential Novikov-Poisson algebra and $(A,  L_{\rhd}, R_{\lhd}, L_\ast)$ is a representation of $(A, \circ, \cdot)$, then $(A, \lhd, \rhd, \ast)$ is a pre-differential Novikov-Poisson algebra.}
\end{pro}
\begin{proof}
It is straightforward.
\end{proof}

\delete{The relationship between
pre-PGD-algebras and the $\mathcal O$-operators on the associated
PGD-algebras is given as follows.

\begin{pro}\label{cor:iden}
\begin{enumerate}
\item \label{it:1}Let $(A, \circ, \cdot, [\cdot,\cdot])$ be a PGD algebra, $(V, l_A,
r_A, \sigma_A, \rho_A)$ be a representation of $(A, \circ, \cdot, [\cdot,\cdot])$ and $T: V\rightarrow A$ be an
$\mathcal{O}$-operator on $(A, \circ, \cdot, [\cdot,\cdot])$ associated to $(V, l_A, r_A, \sigma_A, \rho_A)$. Then there
exists a pre-PGD-algebra structure on $V$ defined by
\begin{eqnarray}\label{eq:ndend} \notag
u\rhd v:=l_A(T(u))v,\;\; u\lhd v:= r_A(T(v))u, \;\; u\ast v:=\sigma_A(T(u))v,\;\;u\diamond v:=\rho_A(T(u))v,\; u,~v\in V.
\end{eqnarray}
\item\label{it:2}
Let $(A, \lhd, \rhd, \ast, \diamond)$ be a pre-PGD-algebra and $(A, \circ, \cdot, [\cdot,\cdot])$ be the
associated PGD-algebra. Then the identity map $I$ is an $\mathcal
O$-operator on $(A, \circ, \cdot, [\cdot,\cdot])$ associated to the representation $(A,
L_{A,\rhd}, R_{A,\lhd}, L_{A,\ast}, L_{A, \diamond})$.
\end{enumerate}
\end{pro}
\begin{proof}
It is straightforward.
\end{proof}

\begin{thm}\label{NYB-ND}
Let $(A, \lhd, \rhd, \ast, \diamond)$ be a  pre-PGD-algebra and $(A,
\circ, \cdot, [\cdot,\cdot])$ be the associated PGD-algebra.  Then
\begin{eqnarray}\label{eq:solu}
r:=\sum_{i=1}^n(e_i\otimes e_i^\ast-e_i^\ast\otimes e_i),
\end{eqnarray}
is a solution of the PGDYBE in
the  PGD-algebra $A\ltimes_{L_{A,\rhd}^\ast+R_{A,\lhd}^\ast,
-R_{A,\lhd}^\ast}^{ -L_{A,\ast}^\ast, L_{A,\diamond}^\ast}A^\ast$, where $\{e_1, \ldots, e_n\}$ is a linear basis of
$A$ and $\{e_1^\ast, \ldots, e_n^\ast\}$ is the dual basis of
$A^\ast$.
\end{thm}
\begin{proof}
By Proposition~\ref{cor:iden} (\ref{it:2}), the identity map
$I: A\rightarrow A$ is an $\mathcal{O}$-operator of $(A, \circ, \cdot, [\cdot,\cdot])$
associated to $(A, L_{A,\rhd}, R_{A,\lhd}, L_{A,\ast}, L_{A,\diamond})$. Therefore this
conclusion follows from Theorem \ref{oper-form2}.
\end{proof}}

Combining Theorem \ref{constr-PCB-pre-P} and Proposition
\ref{corresp-pre} together, we obtain the following conclusion.
\begin{thm}\label{Constr-P-Conf-pre-DNP}
Let $(A, \lhd, \rhd, \succ, \diamond)$ be a pre-PGD-algebra, $(P={\bf k}[\partial]\otimes_{\bf k}  A, \circ_\lambda, \succ_\lambda)$ be the pre-Poisson conformal algebra corresponding to $(A, \lhd, \rhd, \succ, \diamond)$ and $(P={\bf k}[\partial] \otimes_{\bf k} A, [\cdot_\lambda \cdot], \cdot_\lambda\cdot)$ be the associated Poisson conformal algebra of $(P={\bf k}[\partial] \otimes_{\bf k} A, \circ_\lambda, \succ_\lambda)$. Let $\{e_1, \ldots, e_n\}$ be a basis of $A$ and $\{e_1^\ast, \ldots, e_n^\ast\}$ be the dual basis of $A^\ast$.
Then
\begin{eqnarray*}
r=\sum_{i=1}^n(e_i\otimes e_i^\ast-e_i^\ast\otimes e_i)
\end{eqnarray*} is a skew-symmetric
solution of the PCYBE in the Poisson conformal algebra $\hat{P}=P\ltimes_{\mathfrak{L}_\circ^\ast,-\mathfrak{L}_{\succ}^\ast}P^{\ast c}$.
 Therefore there is a Poisson conformal bialgebra structure on the Poisson conformal algebra $\hat{P}$  given by $\delta(a)=({\mathfrak{ad}_{\hat{P}}(a)}_\lambda\otimes I+I\otimes {\mathfrak{ad}_{\hat{P}}(a)}_\lambda)
r|_{\lambda=-\partial^{\otimes^2}}$ and $\Delta(a)=(I\otimes {\mathfrak{L}_{\hat{P}}(a)}_\lambda-{\mathfrak{L}_{\hat{P}}(a)}_\lambda \otimes I)r|_{\lambda=-\partial^{\otimes^2}}$ for all $a\in \hat{P}$.
\end{thm}

\begin{rmk}
In fact, similar to the classical theory of Lie bialgebras
\cite{Bai, CP} as well as the study in Section 3.2, we can also obtain a construction of PGD-bialgebras from pre-PGD-algebras as follows without detailed explanation.\\
{\bf Claim:} {\it Let $(A, \lhd, \rhd, \succ, \diamond)$ be a
pre-PGD-algebra and $(A, \circ, \cdot, [\cdot,\cdot])$ be the
associated PGD-algebra.  Let $(\hat A=
A\ltimes_{L_{A,\rhd}^\ast+R_{A,\lhd}^\ast, -R_{A,\lhd}^\ast}^{
-L_{A,\succ}^\ast, L_{A,\diamond}^\ast}A^\ast,  \circ, \cdot,
[\cdot,\cdot])$ be the semi-direct product of $(A, \circ, \cdot,
[\cdot,\cdot])$ associated to the representation $(A^\ast,
L_{A,\rhd}^\ast+R_{A,\lhd}^\ast, -R_{A,\lhd}^\ast,
-L_{A,\succ}^\ast, L_{A,\diamond}^\ast)$. Set $
r=\sum_{i=1}^n(e_i\otimes e_i^\ast-e_i^\ast\otimes e_i)$ where
$\{e_1, \ldots, e_n\}$ is a basis of $A$ and $\{e_1^\ast, \ldots,
e_n^\ast\}$ is the dual basis of $A^\ast$. Then there is a
PGD-bialgebra structure $(\hat A, \circ, \cdot, [\cdot,\cdot],
\Delta_{1,r}, \Delta_{2,r}, \delta_r)$ defined by
\begin{eqnarray*}
\label{cobnov}&&\Delta_{1,r}(a):=-(L_{A,\circ}(a)\otimes I+I\otimes L_{A,\star}(a))r,\;\;\Delta_{2,r}(a)=(I\otimes L_{A,\cdot}(a)-L_{A,\cdot}(a)\otimes I)
r,\\
\mlabel{coLie}
&&\delta_r(a)=(\ad_A(a)\otimes I+I\otimes \ad_{A}(a))r,\;\;a\in \hat A.
\end{eqnarray*}}
%\cm{all the index and elements are in $\hat A$, not in $A$?}

Moreover, the Poisson conformal bialgebra corresponding to $(\hat
A, \circ, \cdot, [\cdot,\cdot], \Delta_{1,r}, \Delta_{2,r},
\delta_r)$ by Theorem~\ref{corr-bialg} is the same as the one
constructed by Theorem~\ref{Constr-P-Conf-pre-DNP}, that is, there
is
 the following commutative diagram
\begin{equation*}
    \begin{split}
        \xymatrix{
\text{pre-PGD-algebras} \ar[r]^{\text{Claim}} \ar@{->}[d]^{\text{Prop.} \ref{corresp-pre}}&
\text{PGD-bialgebras}   \ar@{->}[d]^{\text{Thm.} \ref{corr-bialg}} \\
\text{ pre-Poisson conformal algebras} \ar[r]^{\text{Thm.} \ref{constr-PCB-pre-P}} & \text{
Poisson conformal bialgebras  }}
    \end{split}
    \mlabel{eq:bigdiag}
   \vspace{-.2cm}
\end{equation*}

\delete{In fact, similar to the classical theory of Lie bialgebras
\cite{Bai,CP}) as well as the study in Section 3.2, one can also
develop a theory of PGD-bialgebras including the introduction of
the analogue of classical Yang-Baxter equation, say, PGD
Yang-Baxter equation (PGDYBE), and $\mathcal{O}$-operators on
PGD-algebras. In particular, there is a construction of
PGD-bialgebras from pre-PGD-algebras, which is illustrated as
follows.
\begin{equation*}
    \begin{split}
        \xymatrix{
         \txt{  pre-PGD-algebras} \ar[r]     & \txt{ $\mathcal{O}$-operators on\\ PGD-algebras}\ar[r] &
            \txt{solutions of\\the PGDYBE} \ar[r]  & \text{PGD-bialgebras}  }
        \vspace{-.1cm}
    \end{split}
    \mlabel{eq:Lieconfdiag}
\end{equation*}
Moreover from a pre-PGD-algebra,  there is a PGD-bialgebra
constructed by this way and the corresponding Poisson conformal
bialgebra is the same as the one constructed by
Theorem~\ref{Constr-P-Conf-pre-DNP}, that is, there is
 the following commutative diagram
\begin{equation*}
    \begin{split}
        \xymatrix{
\text{pre-PGD-algebras} \ar[r] \ar@{<->}[d]&
\text{PGD-bialgebras}   \ar@{<->}[d] \\
\text{a class of pre-Poisson conformal algebras} \ar[r] & \text{a class of
Poisson conformal bialgebras  }}
    \end{split}
    \mlabel{eq:bigdiag}
   \vspace{-.2cm}
\end{equation*}}
\end{rmk}
%\yy{This remark is modified.}

At the end of this paper, we illustrate the construction of
Poisson conformal bialgebras given in
Theorem~\ref{Constr-P-Conf-pre-DNP} by the following explicit
example (in fact, from a pre-differential Novikov-Poisson
algebra).
\begin{ex}
Let $(A={\bf k }e_1\oplus {\bf k} e_2 \oplus {\bf k} e_3, \succ)$
be a Zinbiel algebra in dimension $3$ whose non-trivial products
are given by
\begin{eqnarray*}
e_1\succ e_1=e_2, \;\;e_1\succ e_2=e_3,\;\; e_2\succ e_1=e_3.
\end{eqnarray*}
Let $\alpha\in {\bf k}$ and $D: A\rightarrow A$ be a derivation
given by
\begin{eqnarray*}
D(e_1)=e_1+\alpha e_2,\;\;D(e_2)=2e_2+2\alpha e_3,\;\;D(e_3)=3e_3.
\end{eqnarray*}
By Proposition \ref{constr-PDNP}, there is a pre-differential Novikov-Poisson algebra
$(A, \lhd, \rhd, \succ)$   whose non-trivial products are given by
\begin{eqnarray*}
&&e_1\lhd e_1=e_1\rhd e_1=e_2+\alpha e_3,\;\;e_1 \lhd e_2=e_1\rhd e_2=2e_3,\;\;e_2\lhd e_1=e_2\rhd e_1=e_3,\\
&&e_1\succ e_1=e_2,\;\; e_1\succ e_2=e_2\succ e_1=e_3.
\end{eqnarray*}
Then the pre-Poisson conformal algebra corresponding to $(A, \lhd, \rhd, \succ)$ is $(P={\bf k}[\partial]\otimes_{\bf k} A, \circ_\lambda, \succ_\lambda)$ whose non-trivial $\lambda$-products are given by
\begin{eqnarray*}
&&e_1\circ_{\lambda} e_1=(\partial +2\lambda)(e_2+\alpha e_3),\;\; e_1\circ_\lambda e_2=(\partial+3\lambda)e_3,\;\; e_2\circ_\lambda e_1=(2\partial+3\lambda)e_3,\\
&&e_1\succ_\lambda e_1=e_2,\;\;e_1\succ_\lambda
e_2=e_2\succ_\lambda e_1=e_3.
\end{eqnarray*}
The associated Poisson conformal algebra of $(P={\bf k}[\partial]\otimes_{\bf k} A, \circ_\lambda, \succ_\lambda)$ is  $(P={\bf k}[\partial]\otimes_{\bf k} A, [\cdot_\lambda \cdot], \cdot_\lambda\cdot)$ whose non-trivial $\lambda$-products are given by
\begin{eqnarray}
\label{eqq1}&&[{e_1}_\lambda e_1]=(\partial+2\lambda)(2e_2+2\alpha e_3),\;\;[{e_1}_\lambda e_2]=-[{e_2}_{-\lambda-\partial}e_1]=2(\partial+3\lambda) e_3,\\
\label{eqq2}&&{e_1}_\lambda e_1=2e_2,\;\; {e_1}_\lambda
e_2={e_2}_\lambda e_1=2e_3.
\end{eqnarray}
 Let $\{e_1^\ast, e_2^\ast, e_3^\ast\}$ be the basis of $A^\ast$ which is dual to $\{e_1, e_2, e_3\}$.
Then the semi-direct product of  $(P={\bf k}[\partial]\otimes_{\bf k} A,
[\cdot_\lambda \cdot], \cdot_\lambda\cdot)$ and the representation
$(P^{\ast c}={\bf k}[\partial]\otimes_{\bf k} A^\ast, \mathfrak{L}_\circ^\ast,
-\mathfrak{L}_\succ^\ast)$ is
$\hat{P}=P\ltimes_{\mathfrak{L}_\circ^\ast,-\mathfrak{L}_{\succ}^\ast}P^{\ast c}$  whose
non-trivial $\lambda$-products are given by Eqs. (\ref{eqq1}),
(\ref{eqq2}) and
\begin{eqnarray*}
&&[{e_1}_\lambda e_2^\ast]=-[{e_2^\ast}_{-\lambda-\partial}e_1]=(\partial-\lambda)e_1^\ast,\;\;[{e_1}_\lambda e_3^\ast]=-[{e_3^\ast}_{-\lambda-\partial}e_1]=\partial(\alpha e_1^\ast+ e_2^\ast)-\lambda (\alpha e_1^\ast+2e_2^\ast),\\
&&[{e_2}_\lambda e_3^\ast]=-[{e_3^\ast}_{-\lambda-\partial}e_2]=(2\partial-\lambda) e_1^\ast,\\
&&{e_1}_\lambda e_2^\ast={e_2}^\ast_\lambda e_1=e_1^\ast,\;\;{e_1}_\lambda e_3^\ast={e_3^\ast}_\lambda e_1= e_2^\ast,\;\;{e_2}_\lambda e_3^\ast={e_3^\ast}_\lambda e_2=e_1^\ast.
\end{eqnarray*}
Set $r=\sum_{i=1}^3(e_i\otimes e_i^\ast-e_i^\ast\otimes e_i)$. By
Theorem \ref{Constr-P-Conf-pre-DNP}, there is a Poisson conformal
bialgebra structure on  $\hat{P}$ with $\delta$ and $\Delta$
respectively given by
\begin{eqnarray*}
&&\delta(e_3)=\delta(e_1^\ast)=0,\;\; \delta(e_2^\ast)=-2(\partial e_1^\ast\otimes e_1^\ast-e_1^\ast\otimes \partial e_1^\ast),\\
&&\delta(e_1)=-(\partial e_2\otimes e_1^\ast-e_1^\ast\otimes \partial e_2)-\alpha(\partial e_3\otimes e_1^\ast-e_1^\ast\otimes \partial e_3)-2(\partial e_3\otimes e_2^\ast-e_2^\ast\otimes \partial e_3)\\
&&\quad +2(\partial e_1^\ast\otimes e_2-e_2\otimes \partial
e_1^\ast)
+2\alpha(\partial e_1^\ast\otimes e_3-e_3\otimes \partial e_1^\ast)+3(\partial e_2^\ast\otimes e_3-e_3\otimes \partial e_2^\ast),\\
&&\delta(e_2)=-(\partial e_3\otimes e_1^\ast-e_1^\ast\otimes \partial e_3)+3(\partial e_1^\ast\otimes e_3-e_3\otimes \partial e_1^\ast),\\
&&\delta(e_3^\ast)=-4(\partial e_1^\ast\otimes e_2^\ast-e_2^\ast\otimes\partial e_1^\ast)-2(\partial e_2^\ast\otimes e_1^\ast-e_1^\ast\otimes \partial e_2^\ast)-2\alpha(\partial e_1^\ast\otimes e_1^\ast-e_1^\ast\otimes \partial e_1^\ast),\\
&&\Delta(e_1)=-e_1^\ast\otimes e_2-e_2\otimes e_1^\ast-e_2^\ast\otimes e_3-e_3\otimes e_2^\ast,\\
&&\Delta(e_2)=-e_1^\ast\otimes e_3-e_3\otimes e_1^\ast,\;\;\Delta(e_3)=\Delta(e_1^\ast)=0,\\
&&\Delta(e_2^\ast)=-2e_1^\ast\otimes e_1^\ast,\;\;\Delta(e_3^\ast)=-2e_1^\ast\otimes e_2^\ast-2e_2^\ast\otimes e_1^\ast.
\end{eqnarray*}
\end{ex}

\noindent {\bf Acknowledgments.} This work is supported by NSFC
(12171129, 11931009, 12271265, 12261131498, 12326319),
 the Fundamental Research Funds for the Central Universities and Nankai Zhide Foundation.

\smallskip

\noindent
{\bf Declaration of interests. } The authors have no conflicts of interest to disclose.

\smallskip

\noindent
{\bf Data availability. } No new data were created or analyzed in this study.

\end{document}